\documentclass[a4paper,11pt]{amsart}
\usepackage{amsmath,amsthm,amssymb,mathrsfs,stmaryrd,mathtools}
\usepackage{bm}
\usepackage{xurl}
\usepackage{color}
\usepackage{eucal}
\usepackage{slashed}
\usepackage{physics}
\usepackage{graphicx}
\usepackage{tikz}
\usetikzlibrary{nfold,arrows}
\usepackage{tikz-cd}
\usetikzlibrary{intersections, calc, arrows, arrows.meta, shadows}
\usepackage[utf8]{inputenc} 
\usepackage[T1]{fontenc}
\usepackage[style=alphabetic,sorting=nty]{biblatex}
\usepackage{geometry} 
\usepackage[inline]{enumitem} 
\setlist[enumerate]{itemsep=2pt,parsep=2pt,before={\parskip=2pt}}
\setlist[description]{style=standard}
\usepackage[colorlinks=true, hyperindex, linkcolor=magenta, pagebackref=false, citecolor=cyan, pdfpagelabels]{hyperref} 
\usepackage{aliascnt}
\usepackage{cleveref}
\usepackage{quiver}
\usepackage{mathtools}
\mathtoolsset{showonlyrefs=true}

\newcommand{\calA}{\mathcal{A}}
\newcommand{\calB}{\mathcal{B}}
\newcommand{\calC}{\mathcal{C}}
\newcommand{\calD}{\mathcal{D}}
\newcommand{\calE}{\mathcal{E}}
\newcommand{\calG}{\mathcal{G}}
\newcommand{\calH}{\mathcal{H}}

\newcommand{\calP}{\mathcal{P}}

\newcommand{\calT}{\mathcal{T}}
\newcommand{\calU}{\mathcal{U}}
\newcommand{\calV}{\mathcal{V}}

\newcommand{\frakC}{\mathfrak{C}}

\newcommand{\bbA}{\mathbb{A}}
\newcommand{\bbE}{\mathbb{E}}
\newcommand{\bbF}{\mathbb{F}}
\newcommand{\bbG}{\mathbb{G}}
\newcommand{\bbH}{\mathbb{H}}
\newcommand{\bbI}{\mathbb{I}}
\newcommand{\bbJ}{\mathbb{J}}
\newcommand{\bbK}{\mathbb{K}}
\newcommand{\bbL}{\mathbb{L}}
\newcommand{\bbN}{\mathbb{N}}
\newcommand{\bbP}{\mathbb{P}}
\newcommand{\bbR}{\mathbb{R}}
\newcommand{\bbZ}{\mathbb{Z}} 

\newcommand{\FdagSSet}{F_{\dag}^{\sSet}}
\newcommand{\UdagSSet}{U_{\dag}^{\sSet}}
\newcommand{\FdagSCat}{F_{\dag}^{\sCat}}
\newcommand{\UdagSCat}{U_{\dag}^{\sCat}}
\newcommand{\FCatDag}{F_{\sCat}^{\dag}}
\newcommand{\UCatDag}{U_{\sCat}^{\dag}}
\newcommand{\FCatDagO}{F_{\sCat,O}^{\dag}}
\newcommand{\UCatDagO}{U_{\sCat,O}^{\dag}}

\newcommand{\AInv}{\mathrm{AInv}}
\newcommand{\Bergner}{\mathrm{Bergner}}
\newcommand{\DCH}{\mathrm{DCH}}
\newcommand{\rmDH}{\mathrm{DH}}
\newcommand{\ES}{\mathrm{ES}}
\newcommand{\Fun}{\mathrm{Fun}}
\newcommand{\Kan}{\mathrm{Kan}}
\newcommand{\qCat}{\mathrm{qCat}}
\newcommand{\sCat}{\mathrm{sCat}}
\newcommand{\Set}{\mathrm{Set}}
\newcommand{\sGpd}{\mathrm{sGpd}}
\newcommand{\sGph}{\mathrm{sGph}}
\newcommand{\Sing}{\mathrm{Sing}}
\newcommand{\sSet}{\mathrm{sSet}}

\newcommand{\ai}{\mathrm{ai}}
\newcommand{\cell}{\mathrm{cell}}
\newcommand{\cod}{\mathrm{cod}}
\newcommand{\cof}{\mathrm{cof}}
\DeclareMathOperator*{\colim}{colim}
\newcommand{\cu}{\mathrm{cu}}
\renewcommand{\ev}{\mathrm{ev}}
\newcommand{\fib}{\mathrm{fib}}
\newcommand{\gph}{\mathrm{gph}}
\newcommand{\h}{\mathrm{h}} 
\newcommand{\Ho}{\mathrm{Ho}}
\DeclareMathOperator{\Hom}{Hom}
\newcommand{\inj}{\mathrm{inj}}
\DeclareMathOperator{\id}{id}
\newcommand{\Joyal}{\mathrm{Joyal}}
\newcommand{\loc}{\mathrm{loc}}
\DeclareMathOperator{\Map}{Map}
\DeclareMathOperator{\Mor}{Mor}
\DeclareMathOperator{\Ob}{Ob}
\DeclareMathOperator{\myop}{op}
\newcommand{\rmM}{\mathrm{M}}
\newcommand{\Path}{\mathrm{Path}}
\newcommand{\proj}{\mathrm{proj}}
\newcommand{\rev}{\mathrm{rev}}
\newcommand{\RMap}{\mathrm{RMap}}
\newcommand{\rmtag}{\mathrm{tag}}
\newcommand{\rmtop}{\mathrm{top}}
\newcommand{\sa}{\mathrm{sa}}
\newcommand{\sk}{\mathrm{sk}} 
\newcommand{\Tw}{\mathrm{Tw}}
\newcommand{\und}{\mathrm{und}}

\usepackage{relsize}
\usepackage[bbgreekl]{mathbbol}
\usepackage{amsfonts}
\DeclareSymbolFontAlphabet{\mathbb}{AMSb} 
\DeclareSymbolFontAlphabet{\mathbbl}{bbold}
\newcommand{\prism}{{\mathlarger{\mathbbl{\Delta}}}}

\DeclareFontFamily{U}{dmjhira}{}
\DeclareFontShape{U}{dmjhira}{m}{n}{
  <-> dmjhira
}{}
\DeclareFontSubstitution{U}{dmjhira}{m}{n}
\newcommand{\yo}{\text{{\usefont{U}{dmjhira}{m}{n}\symbol{"48}}}}

\theoremstyle{definition}
\newtheorem{theorem}{Theorem}[subsection]

\newaliascnt{conjecture}{theorem}

\aliascntresetthe{conjecture}
\crefname{conjecture}{conjecture}{conjectures}
\Crefname{conjecture}{Conjecture}{Conjectures}

\newaliascnt{construction}{theorem}
\newtheorem{construction}[construction]{Construction}
\aliascntresetthe{construction}
\crefname{construction}{construction}{constructions}
\Crefname{construction}{Construction}{Constructions}

\newaliascnt{corollary}{theorem}
\newtheorem{corollary}[corollary]{Corollary}
\aliascntresetthe{corollary}
\crefname{corollary}{corollary}{corollaries}
\Crefname{corollary}{Corollary}{Corollaries}

\newaliascnt{definition}{theorem}
\newtheorem{definition}[definition]{Definition}
\aliascntresetthe{definition}
\crefname{definition}{definition}{definitions}
\Crefname{definition}{Definition}{Definitions}

\newaliascnt{example}{theorem}
\newtheorem{example}[example]{Example}
\aliascntresetthe{example}
\crefname{example}{example}{examples}
\Crefname{example}{Example}{Examples}

\newaliascnt{lemma}{theorem}
\newtheorem{lemma}[lemma]{Lemma}
\aliascntresetthe{lemma}
\crefname{lemma}{lemma}{lemmas}
\Crefname{lemma}{Lemma}{Lemmas}

\newaliascnt{notation}{theorem}
\newtheorem{notation}[notation]{Notation}
\aliascntresetthe{notation}
\crefname{notation}{notation}{notations}
\Crefname{notation}{Notation}{Notations}

\newaliascnt{observation}{theorem}

\aliascntresetthe{observation}
\crefname{observation}{observation}{observations}
\Crefname{observation}{Observation}{Observations}

\newaliascnt{proposition}{theorem}
\newtheorem{proposition}[proposition]{Proposition}
\aliascntresetthe{proposition}
\crefname{proposition}{proposition}{propositions}
\Crefname{proposition}{Proposition}{Propositions}

\newaliascnt{remark}{theorem}
\newtheorem{remark}[remark]{Remark}
\aliascntresetthe{remark}
\crefname{remark}{remark}{remarks}
\Crefname{remark}{Remark}{Remarks}

\newaliascnt{fact}{theorem}

\aliascntresetthe{fact}
\crefname{fact}{fact}{facts}
\Crefname{fact}{Fact}{Facts}

\newaliascnt{terminology}{theorem}

\aliascntresetthe{terminology}
\crefname{terminology}{terminology}{terminologies}
\Crefname{terminology}{Terminology}{Terminologies}

\theoremstyle{definition}
\newtheorem*{theorem*}{Theorem}
\newtheorem*{conjecture*}{Conjecture}
\newtheorem*{construction*}{Construction}
\newtheorem*{corollary*}{Corollary}
\newtheorem*{definition*}{Definition}
\newtheorem*{example*}{Example}
\newtheorem*{lemma*}{Lemma}
\newtheorem*{notation*}{Notation}
\newtheorem*{observation*}{Observation}
\newtheorem*{proposition*}{Proposition}
\newtheorem*{remark*}{Remark}
\newtheorem*{terminology*}{Terminology}

\newtheorem*{maintheorema}{Theorem A}
\newtheorem*{maintheoremb}{Theorem B}
\newtheorem*{maintheoremc}{Theorem C}
\newtheorem*{maintheoremd}{Theorem D}

\title[Models for dagger $(\infty,1)$-categories I]{Models for dagger $(\infty,1)$-categories I:\\
Dagger simplicial sets and unitary cores}
\author{Keima Akasaka} 
\thanks{Please contact \url{quasi.cosmoi@gmail.com}, if you find any typos or incorrect arguments.}
\date{\today}

\begin{document}

\begin{abstract}
  We develop simplicial-set models for dagger $(\infty,1)$-categories.
  We first establish a model structure on dagger simplicial categories whose weak equivalences are local weak equivalences that are essentially surjective up to coherently unitary equivalence.
  We then construct the dagger Joyal model structure on dagger simplicial sets and prove that the dagger rigidification--nerve adjunction is a Quillen equivalence.
  We next introduce the unitary core and use it to characterize dagger Joyal equivalences intrinsically.
  Finally, we compare the dagger Joyal model structure with the DCH anti-involutive Joyal model structure.
\end{abstract}

\maketitle 
\setcounter{tocdepth}{1} 
\tableofcontents

\section{Introduction}

Dagger structures arise whenever morphisms carry a preferred adjoint or reversal.
An ordinary dagger category is a category equipped with an identity-on-objects involutive functor to its opposite.
The aim of this paper is to develop model-categorical presentations of the corresponding dagger $(\infty,1)$-categories and to compare them by Quillen equivalences (\cref{dj.thm.equivalence}).
Our constructions are dagger analogues of the Bergner model structure on simplicial categories, the Joyal model structure on simplicial sets, and the complete-Segal-space comparison of Joyal and Tierney \cite{Bergner07,Rez01,JT06}.

Passing to dagger objects introduces genuinely equivariant phenomena.
In particular, the equation $[u]^{\dag}=[u]^{-1}$ in the homotopy category does not in general encode the homotopy-coherent data required by our object-level notion of equivalence:
a weakly unitary class need not be coherently unitary (\cref{bg.prop.weak-not-coherent}).
We encode this additional coherence by natural dagger intervals, modeled by fixed-object cofibrant resolutions of the walking unitary isomorphism.
For ordinary dagger categories, regarded as simplicial categories with discrete mapping spaces, coherent unitarity agrees with the usual notion of a unitary isomorphism (\cref{bg.lem.discrete}).

\subsection*{The main theorems}

Our first main result constructs the dagger Bergner model structure.

\begin{maintheorema}[\cref{bg.thm.main,bg.thm.left-proper}]
  There exists a left proper combinatorial model structure
  $\sCat^{\dag}_{\Bergner}$ on $\sCat^{\dag}$ in which:
  \begin{enumerate}
    \item the weak equivalences are the dagger DK-equivalences $W_{\dag}$;
    \item the generating cofibrations are $I_{\dag}$ and the generating trivial cofibrations are $J_{\dag}$;
    \item the fibrations are $J_{\dag}\text{-}\inj$, and the trivial fibrations are exactly the functors which are surjective on objects and locally trivial Kan fibrations.
  \end{enumerate}
\end{maintheorema}

The construction proceeds through fixed-object model structures, interval extraction and realization, and a global recognition argument.
Our second model is carried by dagger simplicial sets.
Unlike the ordinary Joyal structure, its cofibrations cannot be all monomorphisms:
a self-conjugate positive-dimensional simplex gives an explicit obstruction.
We therefore introduce free cofibrations and obtain the following model structure and comparison.

\begin{maintheoremb}[\cref{dj.thm.main,dj.thm.equivalence}]
  There exists a left proper combinatorial model structure
  $\sSet^{\dag}_{\Joyal}$ on $\sSet^{\dag}$ in which:
  \begin{enumerate}
    \item the cofibrations are the free cofibrations of \cref{dj.def.free-cof},
    i.e.\ the saturation of $I^{\sSet}_{\dag}$;
    \item the weak equivalences are the dagger Joyal equivalences $W^{\sSet}_{\dag}$;
    \item the trivial fibrations are exactly the morphisms $p$ such that $\UdagSSet(p)$ is a trivial Kan fibration.
  \end{enumerate}
  Moreover, the dagger rigidification adjunction (constructed in \cref{dj.lem.lift})
  \begin{align}
    \frakC_{\dag}:
    \sSet^{\dag}_{\Joyal}
    \rightleftarrows
    \sCat^{\dag}_{\Bergner}
    :N_{\dag}
  \end{align}
  is a Quillen equivalence.
\end{maintheoremb}

Every fibrant dagger simplicial set has an underlying quasi-category, but no converse characterization of the fibrant objects is asserted.
For a dagger quasi-category $X$, its unitary core $\calU(X)$ is the full Kan subcomplex of $(X^{\simeq})^{hC_2}$ spanned by the canonical homotopy fixed points associated to the vertices of $X$; 
see \cref{def.unitary_core}.
It gives an intrinsic recognition theorem for the weak equivalences of the dagger Joyal model structure.

\begin{maintheoremc}[\cref{dj.thm.intrinsic-equivalence}]
  Let $f:X\to Y$ be a dagger functor between dagger quasi-categories.
  The following conditions are equivalent:
  \begin{enumerate}
    \item $f$ is a dagger Joyal equivalence, equivalently a weak equivalence in $\sSet^{\dag}_{\Joyal}$;
    \item $f$ is fully faithful and the map
    $\pi_0\calU(f):\pi_0\calU(X)\to\pi_0\calU(Y)$ is surjective;
    \item $f$ is fully faithful and
    $\calU(f):\calU(X)\to\calU(Y)$ is a weak homotopy equivalence.
  \end{enumerate}
\end{maintheoremc}

Our final main result shows that the inclusion of dagger simplicial sets into anti-involutive simplicial sets does not give a Quillen comparison between the two model structures.

\begin{maintheoremd}[\cref{pointset.cor.no-naive-quillen}]
  Neither of the following adjunctions is a Quillen adjunction:
  \begin{align}
    q_{\rmtag,!}:
    \sSet^{\AInv}_{\DCH}
    \rightleftarrows
    \sSet^{\dag}_{\Joyal}:
    q_{\rmtag}^*
    \quad \text{and} \quad
    q_{\rmtag}^*:
    \sSet^{\dag}_{\Joyal}
    \rightleftarrows
    \sSet^{\AInv}_{\DCH}
    :q_{\rmtag,*}.
  \end{align}
  In particular, the dagger Joyal model structure is not obtained by restricting the DCH model structure to its strict fixed-vertex objects.
\end{maintheoremd}

\subsection*{Organization}

This first part is organized as follows:
\begin{itemize}
  \item \cref{bg.section} constructs the dagger Bergner model structure and develops coherent unitarity;
  \item \cref{dj.section} constructs the dagger Joyal model structure and proves the rigidification equivalence;
  \item \cref{dj.core-section} introduces the unitary core and proves the intrinsic recognition theorem for dagger Joyal equivalences;
  \item \cref{pointset.ainv-section} compares dagger simplicial sets with anti-involutive simplicial sets and proves Theorem D.
\end{itemize}

The companion second part \cite{Aka26-2} develops the dagger Rezk models and establishes the dagger Joyal--Tierney equivalence.

\subsection*{Notation}

Unless stated otherwise, weak equivalences and fibrations of simplicial sets refer to the Kan--Quillen model structure.
Terminology relative to another model structure is indicated by the ambient notation or specified explicitly.

We introduce our main notational conventions.
\begin{itemize}
  \item We let $\sCat$ denote the category of simplicial categories and simplicial functors.
  \item We let $\sSet$ denote the category of simplicial sets.
  \item We let $\sCat^{\dag}$ denote the category of dagger simplicial categories and dagger simplicial functors.
  \item We let $\sSet^{\dag}$ denote the category of dagger simplicial sets and dagger morphisms.
  \item We write $\UdagSCat:\sCat^{\dag}\to\sCat$ and $\UdagSSet:\sSet^{\dag}\to\sSet$ for the forgetful functors.
  \item For a simplicial category $\calC$, we write $\h\calC$ for its homotopy category.
  \item We let $C_2:=\{1,\omega\}$ with $\omega^2=1$ denote the cyclic group of order $2$.
\end{itemize}

\subsection*{AI declaration}

In preparing this paper, the author used GPT-5.6 Sol to improve the English and identify linguistic errors. 
It was also used to identify relevant references, allowing the author to shorten proofs developed independently by the author by citing existing results. 
The author assumes full responsibility for all mathematical content.

\subsection*{Acknowledgements}

I would like to thank Jan Steinebrunner and Theo Johnson-Freyd for their helpful comments on the early draft.
This work was supported by JST SPRING, Grant Number JPMJSP2109.

\section{The dagger Bergner model structure}\label{bg.section}

In this section, we construct a model structure on the category $\sCat^{\dag}$ of dagger simplicial categories, following the strategy of \cite{Bergner07} for $\sCat$ (\cref{bg.thm.main}).
The essential new phenomenon is that the correct object-level notion of equivalence is coherently unitary equivalence.

\subsection{Dagger simplicial categories}

We first define the notion of a dagger simplicial category, which is a simplicial category together with a dagger structure.
We also introduce dagger simplicial graphs and their free categories in order to describe the cellular constructions explicitly (\cref{bg.lem.reduction}).

\begin{definition}\label{bg.def.dagger-scat}
  A \emph{dagger simplicial category} $(\calC,\dag_{\calC})$ consists of a simplicial category $\calC$ together with a simplicial functor $\dag_{\calC}:\calC\to\calC^{\myop}$ which is the identity on objects and satisfies $\dag_{\calC}^{\myop}\dag_{\calC}=\id_{\calC}$.
  
  A \emph{dagger simplicial functor} $f:(\calC,\dag_{\calC})\to(\calD,\dag_{\calD})$ is a simplicial functor $f:\calC\to\calD$ satisfying $\dag_{\calD}f=f^{\myop}\dag_{\calC}$.
\end{definition}

\begin{remark}
  Equivalently, a dagger structure on $\calC$ is a family of morphisms of simplicial sets $(-)^{\dag}:\calC(x,y)\to\calC(y,x)$ with $(f^{\dag})^{\dag}=f$, $\id_x^{\dag}=\id_x$, and $(g\circ f)^{\dag}=f^{\dag}\circ g^{\dag}$.
\end{remark}

\begin{notation}
  We let $\sCat^{\dag}$ denote the category of dagger simplicial categories and dagger simplicial functors, and let $\UdagSCat:\sCat^{\dag}\to\sCat$ denote the forgetful functor.
\end{notation}

\begin{lemma}\label{bg.lem.creation}
  The forgetful functor $\UdagSCat:\sCat^{\dag}\to\sCat$ creates all small limits and colimits.
\end{lemma}

\begin{proof}
  Let $(\calC_i,\dag_i)$ be a small diagram and put $\calC:=\colim_i\UdagSCat\calC_i$.
  Since $(-)^{\myop}$ is an involutive autoequivalence of $\sCat$, the maps $\lambda_i^{\myop}\dag_i$ induce a unique functor $\dag_{\calC}:\calC\to\calC^{\myop}$.
  Both $\dag_{\calC}^{\myop}\dag_{\calC}$ and $\id_{\calC}$ agree after precomposition with every colimit map, so they are equal.
  Since the object functor preserves colimits and each $\dag_i$ fixes objects, $\dag_{\calC}$ fixes objects as well.
  The colimit universal property therefore restricts to $\sCat^{\dag}$.
  The limit argument is dual because $(-)^{\myop}$ and the object functor preserve limits.
\end{proof}

We next define the notion of dagger simplicial graph, which is a dagger simplicial category without composition.

\begin{definition}\label{bg.def.dagger-graph}
  A \emph{dagger simplicial graph} $G$ consists of a set $\Ob(G)$ of vertices, simplicial sets $G(x,y)$ for all ordered pairs, and involutions $\tau_{x,y}:G(x,y)\to G(y,x)$ with $\tau_{y,x}\tau_{x,y}=\id$.

  A \emph{map of dagger simplicial graphs} $f : G \to H$ consists of a map $f : \Ob(G) \to \Ob(H)$ and a morphism of simplicial sets $f_{x,y} : G(x,y) \to H(fx,fy)$ satisfying $\tau_{fx,fy}^{H} f_{x,y} = f_{y,x} \tau_{x,y}^{G}$.
\end{definition}

\begin{notation}
  We let $\sGph^{\dag}$ denote the category of dagger simplicial graphs and maps of dagger simplicial graphs, and let $\UCatDag:\sCat^{\dag}\to\sGph^{\dag}$ denote the forgetful functor.
\end{notation}

\begin{lemma}\label{bg.lem.presentable}
  The forgetful functor $\UCatDag:\sCat^{\dag}\to\sGph^{\dag}$ is monadic.
  In particular, $\sCat^{\dag}$ is locally finitely presentable.
\end{lemma}

\begin{proof}
  Given a dagger simplicial graph $G$, we define the free dagger simplicial category $\FCatDag G$ as follows:
  the objects are the vertices of $G$ and a morphism from $x$ to $y$ is given by a (possibly empty) sequence of directed edges $(e_{n}, \cdots, e_{1})$ where the target of each edge matches the source of the next, starting at $x$ and ending at $y$.
  For a morphism $f : x \to y$, the dagger $f^{\dag} : y \to x$ is given by $(\tau(e_{1}),\cdots,\tau(e_{n}))$.
  
  The functor $\UCatDag$ reflects isomorphisms.
  Moreover, the split coequalizer in dagger simplicial graphs of every $\UCatDag$-split pair inherits a unique composition by the usual degreewise argument for ordinary free categories.
  Beck's theorem therefore gives monadicity.

  The word monad is finitary, and $\sGph^{\dag}$ is locally finitely presentable as the category of models of a finitary essentially algebraic theory.
  \Cite[Theorems 3.36 and 2.78]{AR94} now gives the final assertion.
\end{proof}

\begin{notation}\label{bg.not.cells}
  We let
  \begin{itemize}
    \item $\mathbf 1_{\dag}$ denote the terminal dagger simplicial category;
    \item $\mathbf 2_{\dag}$ denote the discrete dagger simplicial category with two objects $0,1$;
    \item $G_{K}$ denote the dagger simplicial graph $G_{K}$ with objects $0,1$, edge spaces 
    \begin{align}
      G_{K}(0,1)=K, ~~ G_{K}(1,0)=K, ~\text{and}~ G_{K}(0,0)=G_{K}(1,1)=\varnothing,
    \end{align}
    and involution the identification of the two copies of $K$.
    \item $\bbA_{\dag}(K) := \FCatDag(G_{K})$ denote the free dagger simplicial category of $G_{K}$;
    \item $\bbA(K)$ denote the simplicial category which has a single nontrivial mapping space $\bbA(K)(0,1)=K$.
  \end{itemize}
\end{notation}

\begin{remark}
  By \cref{bg.lem.presentable}, for every $\calC \in \sCat^{\dag}$, we have an equivalence
  \begin{align}\label{bg.eq.A-univ}
    \Hom_{\sCat^{\dag}}(\bbA_{\dag}(K),\calC)
    \simeq
    \Hom_{\sGph^{\dag}}(G_{K},\UCatDag\calC) 
    \simeq 
    \coprod_{x,y\in\Ob(\calC)}\Hom_{\sSet}(K,\calC(x,y)),
  \end{align}
  since a map of dagger graphs $G_{K} \to \UCatDag\calC$ is determined by a pair $(x,y)$ and a map $K\to\calC(x,y)$. 
  If $K$ is a finite simplicial set, then \eqref{bg.eq.A-univ} shows that $\bbA_{\dag}(K)$ is finitely compact.
\end{remark}

\begin{definition}
  Let $i:K\to L$ be a monomorphism of simplicial sets:
  \begin{itemize}
    \item for a simplicial functor $\bbA(K)\to\calC$ classifying $u:K\to\calC(x,y)$, the \emph{ordinary free cell attachment} $\calC[i,u]$ is the pushout of $\bbA(i):\bbA(K)\to\bbA(L)$ along this map in $\sCat$:
    $\calC[i,u] := \calC \amalg_{\bbA(K)} \bbA(L)$;
    \item for a dagger simplicial functor $\bbA_{\dag}(K)\to\calC$ classifying $u:K\to\calC(x,y)$, the \emph{dagger free cell attachment} $\calC[i]_{\dag}$ is the pushout of $\bbA_{\dag}(i):\bbA_{\dag}(K)\to\bbA_{\dag}(L)$ along this map in $\sCat^{\dag}$:
    $\calC[i]_{\dag}:= \calC \amalg_{\bbA_{\dag}(K)} \bbA_{\dag}(L)$.
  \end{itemize}
\end{definition}

\begin{lemma}\label{bg.lem.ordinary-cell}
  Let $\calC$ be a simplicial category, let $x,y \in \calC$, let $i:K\to L$ be a monomorphism of simplicial sets, and let $u:K\to\calC(x,y)$ be morphisms of simplicial sets.
  Then we have:
  \begin{enumerate}
    \item the functor $\calC\to\calC[i,u]$ is bijective on objects and each $\calC(a,b)\to\calC[i,u](a,b)$ is a monomorphism;
    \item if $i$ is in addition a weak homotopy equivalence, then each $\calC(a,b)\to\calC[i,u](a,b)$ is a weak homotopy equivalence;
  \end{enumerate}
\end{lemma}

\begin{proof}
  (1)
  Since $i$ is a cofibration of simplicial sets, $\bbA(i)$ is a cofibration in the ordinary Bergner model structure.
  Hence its pushout $\calC\to\calC[i,u]$ is a Bergner cofibration.
  It is the identity on objects, and ordinary Bergner cofibrations induce monomorphisms on all mapping spaces by \cite[Remark A.3.2.5]{HTT}.

  (2)
  If $i$ is also a weak homotopy equivalence, then $\bbA(i)$ is a local weak equivalence and hence a trivial Bergner cofibration.
  Its pushout is again a trivial Bergner cofibration.
  In particular, it is a local weak equivalence.
\end{proof}

\begin{lemma}\label{bg.lem.reduction}
  There exists an equivalence of simplicial categories
  \begin{align}
    \UdagSCat(\calC[i]_{\dag})
    \simeq 
    ((\UdagSCat\calC)[i,u])[i,u^{\dag}],
  \end{align}
  the ordinary free cell attachment of $i$ at $(x,y)$ along $u$, followed by the ordinary free cell attachment of $i$ at $(y,x)$ along the composite $u^{\dag}:K\xrightarrow{u}\calC(x,y)\xrightarrow{\dag}\calC(y,x)$.
\end{lemma}

\begin{proof}
  By \cref{bg.lem.creation}, the functor $\UdagSCat$ preserves the pushout:
  \begin{align}
    \UdagSCat(\calC[i]_{\dag})
    \simeq
    \UdagSCat\calC
    \amalg_{\UdagSCat\bbA_{\dag}(K)}
    \UdagSCat\bbA_{\dag}(L).
  \end{align}

  The underlying simplicial category of $\bbA_{\dag}(K)$ is the free simplicial category on the graph $G_K$, which is the coproduct, over the discrete category $\mathbf 2$, of the one-edge graphs $K$ at $(0,1)$ and $K$ at $(1,0)$; 
  hence 
  \begin{align}
    \UdagSCat\bbA_{\dag}(K)=\bbA(K)^{01}\amalg_{\mathbf 2}\bbA(K)^{10},
  \end{align}
  where the superscripts record the position of the nontrivial edge space.
  Then we have 
  \begin{align}
    \UdagSCat(\calC[i]_{\dag})
    \simeq
    (\UdagSCat\calC \amalg_{\bbA(K)^{01}} \bbA(L)^{01})
    \amalg_{\UdagSCat\calC}
    (\UdagSCat\calC \amalg_{\bbA(K)^{10}} \bbA(L)^{10}).
  \end{align}
  The first and the second complements are determined by $u : K \to \calC(x,y)$ and $u^{\dag} : K \to \calC(y,x)$, respectively.
  Thus it is equivalent to the two successive ordinary free cell attachments $((\UdagSCat\calC)[i,u])[i,u^{\dag}]$.
\end{proof}

\subsection{The fixed-object model structure}

The category of dagger graphs on $O$ decomposes into ordinary simplicial-set factors off the diagonal and projective $C_2$-factors on the diagonal.
Transferring this product model structure along the free-category monad gives the desired fixed-object model structures on $\sCat^{\dag}_O$ (\cref{bg.thm.fixed-object}).

\begin{notation}\label{bg.not.fixed-object}
  For a set $O$, we let $\sCat^{\dag}_O\subseteq\sCat^{\dag}$ and $\sGph^{\dag}_O\subseteq\sGph^{\dag}$ denote the subcategories of objects with object set exactly $O$ and morphisms which are the identity on $O$.
\end{notation}

\begin{definition}
  A morphism of $\sCat^{\dag}_O$ is a \emph{local weak equivalence} (resp.\ \emph{local fibration}) if every induced map of mapping spaces is a weak homotopy equivalence (resp.\ Kan fibration).
\end{definition}

\begin{lemma}\label{bg.lem.gph-decomposition}
  Choose a total order on $O$.  
  The construction $G \mapsto ((G(x,y))_{x<y},(G(x,x),\tau_{x,x})_x)$ defines an equivalence of categories
  \begin{align}
    \sGph^{\dag}_O
    \simeq
    \prod_{x<y}\sSet \times \prod_{x\in O}\Fun(\mathbf BC_2,\sSet).
  \end{align}
\end{lemma}

\begin{proof}
  We define a functor 
  \begin{align}
    \Psi : \prod_{x<y}\sSet \times \prod_{x\in O}\Fun(\mathbf BC_2,\sSet) \to \sGph^{\dag}_O
  \end{align}
  as follows:
  For a given $((K_{x,y})_{x<y}, (H_{x},\sigma_{x})_{x})$, we define an object $G \in \sGph^{\dag}_O$ by $\Ob(G) = O$,
  \begin{align}
    G(x,y) := \begin{cases}
      K_{x,y} & (x < y) \\
      K_{y,x} & (y < x) \\
      H_{x}   & (x = y)
    \end{cases}
    \quad \text{and} \quad
    \tau_{x,y}
    := \begin{cases}
      \id_{K_{x,y}} & (x < y) \\
      \id_{K_{y,x}} & (y < x) \\
      \sigma_{x}    & (x = y).
    \end{cases}
  \end{align}
  We can check that $\Psi$ is an inverse of the construction $G \mapsto ((G(x,y))_{x<y},(G(x,x),\tau_{x,x})_x)$.
\end{proof}

\begin{proposition}\label{gb.cor.sGph-model}
  The category $\sGph^{\dag}_O$ admits a cofibrantly generated model structure in which weak equivalences and fibrations are detected on the underlying spaces $G(x,y)$ for all $(x,y)$.

  Consequently, its generating (trivial) cofibrations are the generating (trivial) cofibrations of $\sSet$ placed in a single factor $x<y$, together with $C_2\cdot(\partial\Delta^n\to\Delta^n)$ (resp. $C_2\cdot(\Lambda^n_k\to\Delta^n)$) placed in a single diagonal factor, where $C_{2}\cdot(-) : \sSet \to \Fun(\mathbf B C_{2},\sSet)$ is the left adjoint to the forgetful functor.
\end{proposition}

\begin{proof}
  For the off-diagonal factor $\sSet$, we put the Kan--Quillen model structure.
  For the diagonal factor $\Fun(\mathbf BC_2,\sSet)$, we put the projective model structure.
  Then consider the product model structure on $\sGph^{\dag}_O$.
  By \cref{bg.lem.gph-decomposition}, weak equivalences and fibrations can be judged by underlying simplicial sets, respectively.
\end{proof}

\begin{theorem}\label{bg.thm.fixed-object}
  For every set $O$, the category $\sCat^{\dag}_O$ admits a right proper, simplicial, and cofibrantly generated model structure satisfying:
  \begin{itemize}
    \item weak equivalences are the local weak equivalences;
    \item fibrations are the local fibrations;
    \item generating (trivial) cofibrations are the images under the free functor $\FCatDagO:\sGph^{\dag}_O\to\sCat^{\dag}_O$ of those of \cref{gb.cor.sGph-model}.
  \end{itemize}
\end{theorem}

\begin{proof}
  We transfer the model structure of \cref{gb.cor.sGph-model} along the monadic adjunction $\FCatDagO\dashv \UCatDagO$ using \cite[Theorem 11.3.2]{Hir02}.
  
  \emph{Small hypotheses:}
  The domains of the graph generating cofibrations and generating trivial cofibrations are finitely presentable.
  Since the word monad preserves filtered colimits, the forgetful functor $\UCatDagO$ creates them, and the adjunction shows that the corresponding free-category domains are finitely presentable.

  \emph{The acyclic condition:}
  We must show that a relative $\FCatDagO(J_{\gph})$-cell complex is a local weak equivalence.
  For an off-diagonal cell, a pushout along $\FCatDagO$ of a generating trivial cofibration of $\sGph^{\dag}_O$ is a composite of two ordinary free cell attachments along the trivial cofibration $\Lambda^n_k\to\Delta^n$ (by \cref{bg.lem.reduction}).
  By \cref{bg.lem.ordinary-cell} (2), each is a local weak equivalence.
  For a diagonal projective cell, the free $C_2$-orbit attaches a loop and its dagger at the same object; 
  after forgetting the dagger it is again the composite of the two ordinary horn attachments of \cref{bg.lem.reduction}.  
  Hence it is a local weak equivalence as well.
  We obtain the same result for general cases, since local weak equivalences are stable under transfinite composition because filtered colimits of simplicial sets along monomorphisms (by \cref{bg.lem.ordinary-cell} (1)) preserve weak equivalences.
  
  \emph{Existence:}
  By \cite[Theorem 11.3.2]{Hir02}, weak equivalences and fibrations are detected after applying $\UCatDagO: \sCat^{\dag}_O \to \sGph^{\dag}_O$.
  By \cref{gb.cor.sGph-model}, they are local weak equivalences and local fibrations, respectively.

  \emph{Simplicialness:}
  Put $\calG:=\sGph^{\dag}_O$, $\calA:=\sCat^{\dag}_O$, $F:=\FCatDagO$, $U:=\UCatDagO$, and $T:=UF$.
  By \cref{bg.lem.gph-decomposition} and \cite[Theorem 11.7.3]{Hir02}, the product model structure on $\calG$ is cofibrantly generated and simplicial:
  The tensor of a dagger graph with a simplicial set is defined mapping-spacewise, with trivial $C_2$-action on the second factor.

  The word formula for the monad is
  \begin{align}
    TG(x,y)
    =
    \coprod_{\substack{n\geq0\\x=x_0,\ldots,x_n=y}}
    \prod_{r=1}^{n}G(x_{r-1},x_r).
  \end{align}
  It has the simplicial strength
  \begin{align}
    \operatorname{st}_{G,K}:
    TG\otimes K&\to T(G\otimes K)
    :((g_n,\ldots,g_1),k) \mapsto ((g_n,k),\ldots,(g_1,k)).
  \end{align}
  On the summand $n=0$, it is the unique map $K\to\Delta^0$:
  it sends the empty word labelled by $k$ to the empty word.

  The formula is natural in $G$ and $K$ and satisfies the unit and associativity identities for a simplicial strength:
  in the associativity identity, both composites label every letter by the same simplex.
  
  It also commutes with the dagger and with the unit and multiplication of $T$:
  the dagger reverses the word and applies the dagger letterwise, the unit inserts a one-letter word, and the multiplication concatenates words.
  Thus $T$ is a simplicial monad.

  Reflexive coequalizers in $\calG$ are computed mapping-spacewise and degreewise in sets.
  They are sifted colimits, so they commute with finite products in sets.
  The word formula therefore shows that $T$ preserves reflexive coequalizers.

  Hence $T$ satisfies the hypotheses of \cite[Proposition~2.14]{JN14};
  in particular, $\calA$ is a bicomplete simplicial category.
  The smallness and acyclic conditions verified above are precisely the remaining hypotheses of \cite[Theorem~3.2]{JN14}.
  That theorem gives a right-induced cofibrantly generated simplicial model structure on $\calA$.
  Its weak equivalences and fibrations are created by $U$, so it is the transferred model structure constructed above.
  Thus the transferred model structure is simplicial.

  \emph{Right properness:}
  Let $w : \calC \to \calD$ be a weak equivalence and let $p : \calE \to \calD$ be a fibration in this model structure.
  Consider the pullback $\calP := \calC \times_{\calD} \calE$.
  For every $x,y \in P$, 
  \begin{align}
    \calP(x,y) \simeq \calC(p_{1}x,p_{1}y) \times_{\calD(wp_{1}x,wp_{1}y)} \calE(p_{2}x,p_{2}y).
  \end{align}
  Then $\calC(p_{1}x,p_{1}y) \to \calD(wp_{1}x,wp_{1}y)$ is a weak homotopy equivalence and $\calE(p_{2}x,p_{2}y) \to \calD(wp_{1}x,wp_{1}y)$ is a fibration. 
  Since Kan--Quillen model structure is right proper, the map 
  \begin{align}
    \calC(p_{1}x,p_{1}y) \times_{\calD(wp_{1}x,wp_{1}y)} \calE(p_{2}x,p_{2}y) \to \calE(p_{2}x,p_{2}y)
  \end{align}
  is a weak equivalence.
  Thus $P \to \calE$ is a local weak equivalence.
\end{proof}

\begin{definition}\label{bg.def.amalgam}
  Let $\calB,\calB'\in\sCat^{\dag}$ and choose objects $b\in\Ob(\calB)$ and $b'\in\Ob(\calB')$.
  The \emph{amalgam $ \bbP$ of $\calB$ and $\calB'$ at $b$ and $b'$} is the pushout $\bbP:=\calB\amalg_{\mathbf 1_{\dag}}\calB'$ in $\sCat^{\dag}$, where the maps $\mathbf 1_{\dag}\to\calB$ and $\mathbf 1_{\dag}\to\calB'$ select $b$ and $b'$, respectively.  
  
  We write $o\in\Ob(\bbP)$ for their common image and identify the other objects with their images in $\bbP$.
\end{definition}

\begin{notation}\label{bg.not.fixed-free-product}
  For a set $O$, we write $*_{O}$ for the coproduct in $\sCat_O$.
  If $\calA$ is a simplicial category with object set $S\subseteq O$, we define $\calA^{O}\in\sCat_O$ and $\bbE_S^{O}$ of $\sCat_O$ by 
  \begin{align}
    \calA^{O}(p,q):=
    \begin{cases}
      \calA(p,q) & p,q\in S,\\
      \Delta^0 & p=q\in O\setminus S\\
      \varnothing & \text{otherwise}
    \end{cases}
    \quad \text{and} \quad 
    \bbE_S^{O}(p,q):=
    \begin{cases}
      \Delta^0 & p,q\in S,\text{ or }p=q,\\
      \varnothing & \text{otherwise},
    \end{cases}
  \end{align}
  and abbreviate $\bbE_O^{O}$ to $\bbE_O$.
\end{notation}

\begin{lemma}\label{bg.prop.amalgam}
  Let $\bbP=\calB\amalg_{\mathbf 1_{\dag}}\calB'$ be an amalgam and suppose that every mapping space of $\calB$ and of $\calB'$ is nonempty and weakly contractible.
  Then every mapping space of $\bbP$ is nonempty and weakly contractible.
\end{lemma}

\begin{proof}
  Set $O:=\Ob(\calB)\amalg_{\{o\}}\Ob(\calB')$.  
  The collapse functors $(\UdagSCat\calB)^O \to \bbE_{\Ob(\calB)}^O$ and $(\UdagSCat\calB')^O \to \bbE_{\Ob(\calB')}^O$ are local weak equivalences by the hypothesis on the mapping spaces.

  Homotopy invariance of fixed-object free products \cite[Proposition 2.7]{DK80} therefore gives a local weak equivalence
  \begin{align}
    \UdagSCat\bbP\cong(\UdagSCat\calB)^O*_{O}(\UdagSCat\calB')^O
    \to \bbE_{\Ob(\calB)}^O*_{O}\bbE_{\Ob(\calB')}^O.
  \end{align}
  Since $\Ob(\calB)\cap \Ob(\calB')=\{o\}$ in the tagged pushout $O$, the reduced-word normal form of \cite[\S 1.4 (iii)]{DK80} canonically identifies the target with $\bbE_O$.
  Thus every mapping space of the target is $\Delta^0$, and the result follows.
\end{proof}

\begin{lemma}\label{bg.prop.extension}
  Let $\calC\in\sCat^{\dag}$, let $x\in\calC$, and let $\calB$ be a dagger simplicial category with object set $\{0,1\}$ all of whose mapping spaces are nonempty and weakly contractible.
  Form the amalgam $\calC\langle\calB\rangle:=\calC\amalg_{\mathbf 1_{\dag}}\calB$ of
  \cref{bg.def.amalgam} by gluing $x=0$, and write $y$ for the image of $1$.
  Then:
  \begin{enumerate}
    \item for all $a,b\in\Ob(\calC)$ the map $\calC(a,b)\to\calC\langle\calB\rangle(a,b)$ is a monomorphism and a weak homotopy equivalence;
    \item for every $a\in\Ob(\calC)$ and every vertex $h\in\calB(0,1)$, the maps $h\circ- : \calC(a,x)\to\calC\langle\calB\rangle(a,y)$ and $-\circ h^{\dag} : \calC(x,a)\to\calC\langle\calB\rangle(y,a)$ are weak homotopy equivalences.
  \end{enumerate}
\end{lemma}

\begin{proof}
  (1)
  Put $O:=\Ob(\calC)\amalg\{y\}$.  
  Collapsing its four mapping spaces gives a local weak equivalence $(\UdagSCat\calB)^O\to\bbE_{\{x,y\}}^O$.
  Applying \cite[Proposition 2.7]{DK80} to this map and the identity of $(\UdagSCat\calC)^O$ gives a local weak equivalence
  \begin{align}
    q:\UdagSCat(\calC\langle\calB\rangle)
    \cong(\UdagSCat\calC)^O*_{O}(\UdagSCat\calB)^O
    \to Q:=(\UdagSCat\calC)^O*_{O}\bbE_{\{x,y\}}^O.
  \end{align}
  Write $u:x\to y$ for the unique non-identity arrow of $\bbE_{\{x,y\}}^O$ with this source and target.  
  For $a,b\in\Ob(\calC)$, the reduced-word normal form gives isomorphisms 
  \begin{align}
    \calC(a,b)\xrightarrow{\cong}Q(a,b) &: f \mapsto f, \\
    \calC(a,x)\xrightarrow{\cong}Q(a,y) &: f \mapsto uf, \\
    \calC(x,a)\xrightarrow{\cong}Q(y,a) &: g \mapsto gu^{-1}.
  \end{align}
  For $a,b\in\Ob(\calC)$, the composite 
  \begin{align}
    \calC(a,b)\to\calC\langle\calB\rangle(a,b) \xrightarrow{q}Q(a,b)
  \end{align}
  is the first isomorphism.
  Since the second map is a weak equivalence, the first is one by two-out-of-three.  
  In each simplicial degree, its image consists of the length-one words from the first factor, together with the empty identity word when $a=b$, so uniqueness of reduced normal forms makes it injective.
  It is therefore a monomorphism.

  (2)
  Under the collapse $(\UdagSCat\calB)^O\to\bbE_{\{x,y\}}^O$, the vertices $h$ and $h^{\dag}$ map to $u$ and $u^{-1}$.  
  Consequently, maps in (2), followed by the corresponding component of $q$, is one of the isomorphisms in the above.  
  Since $q$ is a local weak equivalence, two-out-of-three proves (2).
\end{proof}

\subsection{Natural dagger intervals and coherently unitary equivalences}

We define natural dagger intervals and coherently unitary equivalences (\cref{bg.def.cu}), and prove that the coherently unitary morphisms form a subgroupoid stable under, and reflected by, local weak equivalences (\cref{bg.prop.cu-groupoid}).
These intervals refine weak unitarity by retaining the homotopy-coherent data relating an inverse to the dagger.

\begin{notation}\label{bg.not.intervals}
  We let 
  \begin{itemize}
    \item $\bbJ_{\dag}:=\bbA_{\dag}(\Delta^0)$ denote the \emph{walking dagger arrow}: 
    the dagger category with generating arrow $a:0\to1$ and formal dagger $a^{\dag}:1\to0$, no relations imposed.
    \item $\bbI_{\dag}$ denote the \emph{walking unitary isomorphism}: 
    the dagger ($1$-)category with objects $0,1$, morphisms generated by $u:0\to1$ and $u^{\dag}$ subject to $u^{\dag}u=\id_0$ and $uu^{\dag}=\id_1$ (so all four hom-sets are singletons), regarded as a dagger simplicial category with discrete mapping spaces.
    \item $q:\bbJ_{\dag}\to\bbI_{\dag}$ denote the quotient dagger functor with $q(a)=u$.
  \end{itemize}
\end{notation}

\begin{definition}\label{bg.def.interval}
  A \emph{natural dagger interval} is a factorization 
  \begin{align}
    \bbJ_{\dag}\xrightarrow{j_{\bbH}}\bbH\xrightarrow{p_{\bbH}}\bbI_{\dag}
  \end{align}
  of $q$ in $\sCat^{\dag}_{\{0,1\}}$ satisfying:
  \begin{enumerate}
    \item $j_{\bbH}$ is a cofibration in the fixed-object model structure of \cref{bg.thm.fixed-object} and $p_{\bbH}$ is a local weak equivalence;
    \item every mapping space of $\bbH$ is countable.
  \end{enumerate}
  We write $h_{\bbH}:=j_{\bbH}(a)\in\bbH(0,1)_0$.
  We fix once and for all a set $\calH_{\dag}$ of representatives of the isomorphism classes of natural dagger intervals.
  Here an isomorphism of intervals is an isomorphism under $\bbJ_{\dag}$ and over $\bbI_{\dag}$, so it preserves endpoints and the distinguished generator.
\end{definition}

\begin{lemma}\label{bg.lem.interval-existence}
  Natural dagger intervals exist, and their isomorphism classes form a set.
\end{lemma}

\begin{proof}
  Apply the small object argument to $q$ using the finite generating cofibrations of \cref{bg.thm.fixed-object}, and choose its standard sequential, or $\omega$-stage, form.
  This factors $q$ as a cofibration followed by a trivial fibration.
  
  The category $\bbJ_{\dag}$ has degreewise countable mapping spaces.
  Since the generating domains are finite, a map from one of them into a stage with countable mapping spaces is determined by finitely many simplices;
  hence at each stage there are only countably many lifting problems and only countably many finite cells are attached.
  Free categories on countable graphs have countable mapping spaces.
  Hence every stage, and the sequential colimit $\bbH$, has countable mapping spaces.
  
  A dagger simplicial category on $\{0,1\}$ with countable mapping spaces is specified by a set-indexed family of countable data, so the isomorphism classes form a set.
\end{proof}

\begin{remark}\label{bg.rem.interval-correction}
  Although the construction of \cref{bg.lem.interval-existence} produces an interval for which $p_{\bbH}$ is also a local fibration, the latter property is not required of an arbitrary natural dagger interval.
  The condition that $p_{\bbH}$ is a local weak equivalence and local fibration is incompatible with the extraction property needed below.

  Here is an explicit obstruction;
  let $P=\{0<1\}$, regard $N(P)$ as a simplicial monoid under $\max$ with unit $0$, and let $\calC$ be the one-object
  dagger simplicial category with $\calC(*,*)=N(P)$ and dagger the identity.
  In the two-object copy $\bbG=\calC\langle *,*\rangle$, send the generator of $\bbJ_{\dag}$ to the vertex $1$.  
  If this map factors through an interval satisfying the old local-fibration condition, then $\bbH(0,0)\to\Delta^0$ is a trivial Kan fibration.  
  Thus $\bbH(0,0)$ would be a connected Kan complex.  
  Every simplicial map from a connected Kan complex to $N(P)$ is constant on vertices: 
  its functor on fundamental categories has groupoidal source, whereas the only isomorphisms in the poset $P$ are identities.  
  A dagger functor $\bbH\to\bbG$ must therefore send both $\id_0$ and $h_{\bbH}^{\dag}h_{\bbH}$ to $0$.  
  On the other hand, the chosen image of $h_{\bbH}$ forces the latter image to be $\max(1,1)=1$, a contradiction.
\end{remark}

\begin{lemma}\label{bg.lem.interval-basic}
  Let $\bbH$ be a natural dagger interval.
  Then:
  \begin{enumerate}
    \item every mapping space $\bbH(i,j)$ is nonempty and weakly contractible;
    \item $[h_{\bbH}]$ is an isomorphism in $\h\bbH$ and $[h_{\bbH}]^{\dag}=[h_{\bbH}]^{-1}$;
    \item the relabeled category $\bbH^{\rev}$, obtained by exchanging the names of the objects $0,1$ and equipped with $j_{\bbH^{\rev}}(a):=j_{\bbH}(a)^{\dag}$, is again a natural dagger interval.
  \end{enumerate}
\end{lemma}

\begin{proof}
  (1) 
  Since the mapping spaces of $\bbI_{\dag}$ are points and $p_{\bbH}$ is a local weak equivalence, 
  \begin{align}
    \bbH(i,j) \to \bbI_{\dag}(i,j) \simeq \Delta^{0}
  \end{align}
  is a weak homotopy equivalence.
  Nonemptiness holds since $h_{\bbH}$, its dagger, their composites and the identities provide vertices in all four spaces.
  
  (2) 
  Since $p_{\bbH}$ is a weak homotopy equivalence, $\h p_{\bbH}$ induces a bijection on hom-sets.
  In $\h\bbI_{\dag}$, the relations $[u]^{\dag}[u] = \id$ and $[u][u]^{\dag} = \id$ hold.
  Then we also have $[h_{\bbH}]^{\dag}[h_{\bbH}] = \id$ and $[h_{\bbH}][h_{\bbH}]^{\dag} = \id$.
  
  (3)
  Let $s_{\bbJ}$ and $s_{\bbI}$ be the endpoint-swapping automorphisms of $\bbJ_{\dag}$ and $\bbI_{\dag}$, respectively, and let $r_{\bbH}:\bbH^{\rev}\to\bbH$ be the relabeling isomorphism.
  Define $j_{\bbH^{\rev}}:=r_{\bbH}^{-1}j_{\bbH}s_{\bbJ}$ and $p_{\bbH^{\rev}}:=s_{\bbI}p_{\bbH}r_{\bbH}$.
  Since $s_{\bbI}q=qs_{\bbJ}$, we have 
  \begin{align}
    p_{\bbH^{\rev}}j_{\bbH^{\rev}} = s_{\bbI}qs_{\bbJ} = q.
  \end{align}
  The cofibration, local weak-equivalence and countability conditions are preserved by the displayed isomorphisms.
  Thus $\bbH^{\rev}$ is a natural dagger interval.
\end{proof}

\begin{lemma} \label{bg.1toH_is_cofibration}
  The functor $\mathbf 1_{\dag}\to\bbH$ selecting the object $0$ (or $1$) lies in $I_{\dag}\text{-}\cof$, for the set $I_{\dag}$ of \cref{bg.def.IJ} below.
\end{lemma}

\begin{proof}
  We first construct the inclusion of the selected object into the walking dagger arrow.
  The pushout of the object generator $\varnothing\to\mathbf1_{\dag}$ along $\varnothing\to\mathbf1_{\dag}$ is the discrete two-object dagger category $\mathbf2_{\dag}$.
  Since $\partial\Delta^0=\varnothing$, there are canonical identifications $\bbA_{\dag}(\partial\Delta^0)\cong\mathbf2_{\dag}$.
  Attaching this $0$-cell therefore freely adjoins the arrow $a:0\to1$ and its dagger $a^{\dag}:1\to0$.
  Hence the composite $\mathbf1_{\dag}\to\mathbf2_{\dag}\to\bbJ_{\dag}$ lies in $I_{\dag}\text{-}\cof$.

  By definition of a natural dagger interval, $j_{\bbH}:\bbJ_{\dag}\to\bbH$ is a cofibration in the fixed-object model structure on $\sCat^{\dag}_{\{0,1\}}$.
  Since that model structure is cofibrantly generated by \cref{bg.thm.fixed-object}, there exist a relative fixed-object generating-cofibration complex $k:\bbJ_{\dag}\to\bbK$ and maps $\bbH\xrightarrow{s}\bbK\xrightarrow{r}\bbH$ under $\bbJ_{\dag}$ such that $rs=\id_{\bbH}$, $s j_{\bbH}=k$, and $r k=j_{\bbH}$.

  We compare the fixed-object generating cells with the generators in $I_{\dag}$.
  By the universal property \eqref{bg.eq.A-univ}, an off-diagonal fixed-object generating cell is a pushout in $\sCat^{\dag}$ of $\bbA_{\dag}(\partial\Delta^n)\to\bbA_{\dag}(\Delta^n)$ along a classifying pair of distinct objects.
  A diagonal projective $C_2$-cell is the same pushout with the two classifying objects both sent to the chosen object; the two freely adjoined copies of the simplex are then exchanged by the dagger.
  In both cases the attachment is the identity on the fixed object set, so the global pushout retains that object set and agrees with the corresponding fixed-object pushout.
  Thus every cell of $k$ is a pushout of an element of $I_{\dag}$, and $k$ is a relative $I_{\dag}$-cell complex in $\sCat^{\dag}$.

  Composing $k$ with the $I_{\dag}$-cofibration $\mathbf1_{\dag}\to\bbJ_{\dag}$ constructed above shows that $\mathbf1_{\dag}\to\bbK$ lies in $I_{\dag}\text{-}\cof$.
  Since $s$ and $r$ are maps under $\bbJ_{\dag}$, precomposition with $\mathbf1_{\dag}\to\bbJ_{\dag}$ exhibits $\mathbf1_{\dag}\to\bbH$ as a retract of $\mathbf1_{\dag}\to\bbK$.
  The class $I_{\dag}\text{-}\cof$ is closed under retracts, so the endpoint inclusion $\mathbf1_{\dag}\to\bbH$ selecting $0$ belongs to it.
  The assertion for the endpoint $1$ follows by the same construction with the two objects exchanged, or by applying the result for endpoint $0$ to the reversed interval of \cref{bg.lem.interval-basic} (3).
\end{proof}

\begin{construction}\label{bg.constr.two-object-copy}
  For objects $x,y$ of $\calC\in\sCat^{\dag}$ (possibly equal), let $ \calC\langle x,y\rangle\in\sCat^{\dag}_{\{0,1\}}$ have mapping spaces $\calC\langle x,y\rangle(i,j):=\calC(x_i,x_j)$ with $x_0=x$, $x_1=y$, composition and dagger inherited from $\calC$.
  Note that $x=y$, the two formal objects remain distinct.
  Then a dagger functor $f:\calC\to\calD$ induces a morphism $f\langle x,y\rangle:\calC\langle x,y\rangle\to\calD\langle fx,fy\rangle$ in $\sCat^{\dag}_{\{0,1\}}$
\end{construction}

\begin{notation}
  Let $\Theta \in \sCat^{\dag}_{\{0,1\}}$.
  We fix a functorial fibrant replacement $\Theta \xrightarrow{r_{\Theta}} R\Theta \to *$, where $r_{\Theta}$ is a trivial cofibration and $R\Theta\to\ast$ is a fibration in the model structure of \cref{bg.thm.fixed-object}.
\end{notation}

We can consider two notions of \emph{unitary isomorphism} in dagger simplicial categories:
weakly unitary and coherently unitary.
We show that they are different and the right notion is the latter one.

\begin{definition}\label{bg.def.weakly-unitary}
  Let $\calC$ be a dagger simplicial category.
  Then the homotopy category $\h\calC$ (with $\Hom_{\h\calC}(x,y)=\pi_0\calC(x,y)$) inherits a dagger structure.
  An isomorphism $[u]:x\to y$ in $\h\calC$ is \emph{weakly unitary} if $[u]^{\dag}=[u]^{-1}$.
\end{definition}

\begin{definition}\label{bg.def.cu}
  Let $\calC\in\sCat^{\dag}$ and $u\in\calC(x,y)_0$.
  The class $[u]\in\h\calC$ is a \emph{coherently unitary equivalence} if there exist a natural dagger interval $\bbH$ and a dagger simplicial functor $\Phi:\bbH\to R\calC\langle x,y\rangle$ in $\sCat^{\dag}_{\{0,1\}}$ such that $[\Phi(h_{\bbH})]=[r(u)]$ in $\pi_0R\calC\langle x,y\rangle(0,1)$, where $r := r_{\calC\langle x,y\rangle} : \calC\langle x,y\rangle \to R\calC\langle x,y\rangle$. 
\end{definition}

\begin{notation}
  We let $(\h\calC)^{\cu}\subseteq\h\calC$ denote the subgraph of coherently unitary classes;
  we prove that it is in fact a subgroupoid (see \cref{bg.cor.cohuni_is_subgroupoid}).
\end{notation}

\begin{remark}\label{bg.rem.weak-vs-coherent}
  Weak unitarity is only a condition on path components.
  \Cref{bg.prop.weak-not-coherent} exhibits a one-object dagger simplicial category and a weakly unitary $[u]$ which is not \emph{coherently} unitary in the sense of \cref{bg.def.cu}.
  
  Consequently, the object-level condition in the definition of the weak equivalences of $\sCat^{\dag}$ cannot be phrased in terms of $\h\calC$ alone, in contrast with Bergner's setting, where the object-level condition only refers to isomorphisms in the homotopy category.
  This is the essential point of departure from \cite{Bergner07}.
\end{remark}

\begin{lemma}\label{bg.lem.invariance}
  Let $\Theta,\Theta'\in\sCat^{\dag}_{\{0,1\}}$ be fibrant and let $s:\calC\langle x,y\rangle\to\Theta$, $s':\calC\langle x,y\rangle\to\Theta'$ be local weak equivalences.
  Then there exists a natural dagger interval witness for $[s(u)]$ in $\Theta$ if and only if there exists one for $[s'(u)]$ in $\Theta'$.
  
  In particular, \cref{bg.def.cu} is independent of the choice of functorial fibrant replacement, and may be tested on any fibrant local-weak-equivalence receptacle of $\calC\langle x,y\rangle$.
\end{lemma}

\begin{proof}
  First suppose that a local weak equivalence $g:\Theta\to\Theta'$ satisfies $gs=s'$. 
  The fixed-object initial category $\mathbf2_{\dag}$ is cofibrant, $\bbJ_{\dag}$ is obtained from it by one generating cell, and $j_{\bbH}$ is a cofibration. 
  Hence every natural dagger interval $\bbH$ is cofibrant.  
  Since the fixed-object model structure is simplicial by \cref{bg.thm.fixed-object}, a weak equivalence between fibrant objects induces a bijection $ g_*:[\bbH,\Theta]\xrightarrow{\cong}[\bbH,\Theta']$ on homotopy classes.  
  Evaluation at $h_{\bbH}$ sends homotopic maps to the same element of $\pi_0$ of the $(0,1)$-mapping space.

  If $\Phi'$ witnesses $[s'(u)]$, choose $\Phi$ with $g\Phi\simeq\Phi'$.  
  Then we have 
  \begin{align}
    [g\Phi(h)]=[\Phi'(h)]=[gs(u)];
  \end{align}
  the bijection on $\pi_0$ of mapping spaces gives $[\Phi(h)]=[s(u)]$.  
  The reverse implication follows by composition with $g$.

  For arbitrary $s,s'$, factor $s=gc$ with $c$ a cofibration and $g$ a trivial fibration.  
  Since $s$ is a weak equivalence, $c$ is a trivial cofibration, and $\Theta_0:=\cod(c)$ is fibrant.  
  The lifting property of $c$ against $\Theta'\to *$ supplies $g':\Theta_0\to\Theta'$ with $g'c=s'$.  
  Both $g$ and $g'$ are local weak equivalences, the latter by two-out-of-three.  
  Applying the preceding case to the two cospan legs identifies witnesses in $\Theta$ and $\Theta'$ with  witnesses in $\Theta_0$, and proves the claim.
\end{proof}

\begin{lemma}\label{bg.lem.discrete}
  Let $\calA$ be a dagger $1$-category, regarded as a dagger simplicial category with discrete mapping spaces.
  Then $[u]=u$ is coherently unitary if and only if $u$ is unitary, i.e.\ $u^{\dag}=u^{-1}$.
\end{lemma}

\begin{proof}
  Discrete simplicial sets are Kan complexes, so $\calA\langle x,y\rangle$ is fibrant and may serve as the receptacle in \cref{bg.lem.invariance}.
  A simplicial functor from $\bbH$ to a category with discrete mapping spaces factors uniquely through the component category $\pi_0\bbH$ (apply $\pi_0$ to all mapping spaces; composition and dagger descend).
  By \cref{bg.lem.interval-basic} (1), all hom-sets of $\pi_0\bbH$ are singletons, so $\pi_0\bbH\simeq\bbI_{\dag}$ as dagger categories, with $h_{\bbH}\mapsto u_{\bbI}$.
  Hence witnesses $\bbH\to\calA\langle x,y\rangle$ hitting $u$ correspond to dagger functors $\bbI_{\dag}\to\calA\langle x,y\rangle$ with $u_{\bbI}\mapsto u$, which exist precisely when $u^{\dag}u=\id_x$ and $uu^{\dag}=\id_y$.
  Conversely, if $u$ is unitary, precompose such a functor with $p_{\bbH}$.
\end{proof}

\begin{proposition}\label{bg.prop.weak-not-coherent}
  Let $G=S^1\times C_2$ with componentwise multiplication, let $\theta:G\to G$ be the involutive automorphism $\theta(z,1)=(z,1)$, $\theta(z,\omega)=(-z,\omega)$, and let $\calC$ be the one-object dagger simplicial category with $\calC(\ast,\ast)=\Sing(G)$ and $g^{\dag}:=\theta(g^{-1})$ (anti-multiplicative and involutive since $\theta$ is a homomorphism and $G$ is abelian).
  Then $u:=(1,\omega)$ is weakly unitary, but $[u]$ is not coherently unitary.
\end{proposition}

\begin{proof}
  \emph{Weak unitarity:} 
  Since $u^{-1}=u$, we have
  \begin{align}
    u^{\dag} = \theta(u^{-1}) = \theta(u) = (-1,\omega).
  \end{align}
  Since $(1,\omega),(-1,\omega) \in S^{1} \times \{\omega\}$, we obtain $[u^{\dag}] = [u]$ in $\pi_{0}(G)$.
  Since $\pi_{0}G \simeq C_{2}$ and the nontrivial element $\pi_0(G)\simeq C_2$ is its own inverse; hence $[u]^{-1} = [u]$.
  Thus 
  \begin{align}
    [u]^{\dag} = [u^{\dag}] = [u] = [u]^{-1}.
  \end{align}

  \emph{Non-coherently unitary:}
  Suppose $[u]$ is coherently unitary.
  The two-object copy $\Theta:=\calC\langle\ast,\ast\rangle$ has all four mapping spaces equal to the Kan complex $\Sing(G)$, hence is fibrant in $\sCat^{\dag}_{\{0,1\}}$; 
  by \cref{bg.lem.invariance} there are a natural dagger interval $\bbH$ and a dagger functor $\Phi:\bbH\to\Theta$ with $[\Phi(h)]=[u]$, where $h:=h_{\bbH}$.
  Write $u':=\Phi(h)\in G$.
  Then $u'\in S^1\times\{\omega\}$ since $u'$ and $u$ are in the same component.
  
  Consider the vertex $h^{\dag}h\in\bbH(0,0)_0$.
  Its image is 
  \begin{align}
    \Phi(h^{\dag}h) = \Phi(h)^{\dag}\Phi(h) = u'^{\dag}u' = \theta(u'^{-1})u'.
  \end{align}
  When we put $u' = (z,\omega)$, then $u'^{-1} = (z^{-1},\omega)$ and $\theta(u'^{-1}) = (-z^{-1},\omega)$.
  Thus $\Phi(h^{\dag}h) = (-1,1)$ which is independent of $u'$.
  On the other hand, we have $\Phi(\id_{0}) = \id_{0} = (1,1)$.
  Therefore two vertices $h^{\dag}h$ and $\id_{0}$ are sent to $(-1,1)$ and $(1,1)$, respectively.

  Since $\bbH(0,0)$ is weakly contractible (by \cref{bg.lem.interval-basic}), the fundamental groupoid $\Pi_1|\bbH(0,0)|$ of its geometric realization is connected and has exactly one morphism between any two objects;
  in particular, it is equivalent to the terminal groupoid.
  Thus there is an edge zigzag $c:h^{\dag}h\to\id_0$, and any two such zigzags represent the same morphism in $\Pi_1|\bbH(0,0)|$.

  The dagger of $\bbH$ restricts to a simplicial involution of $\bbH(0,0)$ fixing both vertices $h^{\dag}h$ (as $(h^{\dag}h)^{\dag}=h^{\dag}h$) and $\id_0$ (as $\id_{0}^{\dag} = \id_{0}$).
  Then we obtain an another zigzag $c^{\dag} : h^{\dag}h \to \id_{0}$ with the same endpoints.
  Triviality of $\Pi_1|\bbH(0,0)|$ gives $[c]=[c^{\dag}]$ in $\Pi_1|\bbH(0,0)|(h^{\dag}h,\id_0)$.

  Thus $\Phi(c) : \Phi(h^{\dag}h) \to \Phi(\id_{0})$ represents a zigzag in $\Sing(G)$, i.e. a path class from $(-1,1) \to (1,1)$ in $G$.
  Since both endpoints lie in the component $S^{1}\times\{1\}$, the zigzag in $\Sing(G)$ represents a path class in that component.  
  Choose a representative path $p : [0,1] \to S^{1}$ with $p(0) = -1$ and $p(1) = 1$.

  Since the path in $S^{1} \times \{1\}$, the dagger is given by $(z,1) \mapsto (z^{-1},1)$.
  Then $\Phi(c^{\dag})$ is given by the path $\overline{p}(t) := p(t)^{-1}$.
  Since $(-1)^{-1} = -1$ and $1^{-1} = 1$, it has the same endpoints as those of $p(t)$;
  so $\overline{p}(0) = -1$ and $\overline{p}(1) = 1$.
  Since $[c] = [c^{\dag}]$, $[\Phi(c)] = [\Phi(c^{\dag})]$;
  So $p$ and $\overline{p}$ are homotopic in $S^{1}$.
  
  Consider the lifting $\exp : \bbR \to S^{1} : t \mapsto e^{it}$.
  Take a lift $\tilde{p} : [0,1] \to \bbR$ of $p$.
  Consider $d(p) := \tilde{p}(1)-\tilde{p}(0)$, which is independent of the lifts.
  Since $p(0) = -1$ and $p(1) = 1$, $\tilde{p}(0) \in \pi + 2\pi\bbZ$ and $\tilde{p}(1) \in 2\pi\bbZ$;
  so $d(p) \in \pi + 2\pi\bbZ$.
  On the other hand, the lift of $\overline{p}(t) = p(t)^{-1}$ is given by $-\tilde{p}(t)$.
  Thus $d(\overline{p}) = -d(p)$.
  However, $d(p)$ is a homotopy invariant (with the same endpoints).
  So $d(\overline{p}) = d(p)$.
  Therefore $d(p) = 0$, which is contradiction. 
\end{proof}

\begin{lemma}\label{bg.lem.extraction}
  Let $\bbG\in\sCat^{\dag}_{\{0,1\}}$ have all four mapping spaces nonempty and weakly contractible, and let $v:\bbJ_{\dag}\to\bbG$ be a dagger functor.
  Then $v$ factors as $\bbJ_{\dag}\xrightarrow{j_{\bbH}}\bbH \to\bbG$ with $\bbH$ a natural dagger interval.
\end{lemma}

\begin{proof}
  We first explain how to realize a finite simplicial extension by dagger free-cell attachments.
  Let $\bbK\in\sCat^{\dag}_{\{0,1\}}$, let $\phi:\bbK\to\bbG$ be a dagger functor, and fix an ordered pair $(i,j)$.
  Suppose that a map $\alpha:S^m\to|\bbK(i,j)|$ becomes null-homotopic after composition with $|\phi|$.
  By \cite[Lemma 4.2]{Bergner07}, the map $\bbK(i,j)\to\bbG(i,j)$ admits a factorization 
  \begin{align}
    \bbK(i,j)\to A'\xrightarrow{\bar\phi}\bbG(i,j)
  \end{align}
  such that $A'$ is obtained from $\bbK(i,j)$ by adjoining finitely many nondegenerate simplices and $\alpha$ is null-homotopic in $|A'|$.
  Order the new nondegenerate simplices by dimension and choose a finite filtration 
  \begin{align}
    \bbK(i,j)=A_0\to A_1\to\cdots\to A_N=A'
  \end{align}
  in which every $A_s\to A_{s+1}$ is a pushout of a boundary inclusion $\partial\Delta^{r_s}\to\Delta^{r_s}$.

  We inductively construct dagger simplicial categories $\bbK_s$, dagger functors $\phi_s:\bbK_s\to\bbG$, and maps $\theta_s:A_s\to\bbK_s(i,j)$ satisfying $\phi_s(i,j)\theta_s=\bar\phi|_{A_s}$
  Put $\bbK_0:=\bbK$, $\phi_0:=\phi$, and $\theta_0:=\id_{\bbK(i,j)}$.
  Suppose that the construction has been made through stage $s$, and let $\lambda_s:\partial\Delta^{r_s}\to A_s$ be the next attaching map.
  Form the pushout 
  \begin{align}
    \bbK_{s+1} :=\bbK_s\amalg_{\bbA_{\dag}(\partial\Delta^{r_s})} \bbA_{\dag}(\Delta^{r_s}),
  \end{align}
  where the classifying map corresponds under \eqref{bg.eq.A-univ} to $(i,j)$ and $\theta_s\lambda_s:\partial\Delta^{r_s}\to\bbK_s(i,j)$.
  Let $\chi_s:\Delta^{r_s}\to A_{s+1}\to A'$ be the characteristic simplex of $A_s\to A_{s+1}$ followed by the inclusion into $A'$.
  
  Put $b_s:=\bar\phi\chi_s:\Delta^{r_s}\to\bbG(i,j)$.
  Its boundary is $\bar\phi|_{A_s}\lambda_s=\phi_s(i,j)\theta_s\lambda_s$.
  The datum $(i,j,b_s)$ defines a dagger functor $\bbA_{\dag}(\Delta^{r_s})\to\bbG$ which agrees with $\phi_s$ on $\bbA_{\dag}(\partial\Delta^{r_s})$.
  The pushout therefore gives a dagger functor $\phi_{s+1}:\bbK_{s+1}\to\bbG$.
  Its characteristic simplex also gives a map $\theta_{s+1}:A_{s+1}\to\bbK_{s+1}(i,j)$ extending $\theta_s$ and satisfying $\phi_{s+1}\theta_{s+1}=\bar\phi|_{A_{s+1}}$.

  If $i\neq j$, this pushout adjoins the chosen simplex at $(i,j)$ and its dagger-conjugate at $(j,i)$.
  If $i=j$, it is the free projective $C_2$-cell on the diagonal described in \cref{gb.cor.sGph-model,bg.thm.fixed-object}.
  After forgetting the dagger, the latter is the composite of the two ordinary free-cell attachments of \cref{bg.lem.reduction}.
  Thus, even when the attaching map is self-conjugate, two formal simplices are attached and exchanged by the dagger, and they are sent to $b_s$ and $b_s^{\dag}$.
  No equality $b_s=b_s^{\dag}$ and no equivariant null-homotopy in $\bbG(i,i)$ are required.

  Iterating this construction gives a finite fixed-object cofibration $\bbK\to\widetilde{\bbK}$ over $\bbG$, together with a map $A'\to\widetilde{\bbK}(i,j)$.
  Consequently, $\alpha$ is null-homotopic in $|\widetilde{\bbK}(i,j)|$.
  By \cref{bg.lem.reduction,bg.lem.ordinary-cell}, the transition is a monomorphism on every mapping simplicial set.
  If the mapping spaces of $\bbK$ are countable, then so are those of $\widetilde{\bbK}$:
  after forgetting the dagger, only finitely many ordinary free cells are attached, and all new simplices are represented by finite composable words in a countable collection of old and new simplices.

  We now repeat the countable diagonal construction in the proof of \cite[Lemma~2.4]{Bergner07}, using the preceding finite dagger extension in place of the ordinary extension.
  Put $\bbH_0:=\bbJ_{\dag}$ and $\psi_0:=v$.
  Given a countable $\bbH_n\to\bbG$, the sphere classes in its four mapping spaces form a countable set by \cite[Lemma~4.3]{Bergner07}.
  Enumerate these classes and apply the finite dagger extension successively to kill them over $\bbG$;
  let $\bbH_{n+1}$ be the resulting sequential colimit.
  Thus every sphere class in $\bbH_n(i,j)$ becomes null-homotopic in $\bbH_{n+1}(i,j)$.
  The transition maps are monomorphisms on mapping spaces, and the mapping spaces remain countable because only countably many finite cells and finite composable words are added.

  Put $\bbH:=\colim_{n}\bbH_n$ and let $\psi:\bbH\to\bbG$ be the induced dagger functor.
  Since the free-category monad is finitary, this colimit is created on the mapping dagger graph.
  A map from a sphere into $|\bbH(i,j)|$ has compact image and therefore factors through $|\bbH_N(i,j)|$ for some $N$;
  its class is killed in $\bbH_{N+1}(i,j)$.
  Hence every sphere map into $|\bbH(i,j)|$ is null-homotopic.
  The four mapping spaces are nonempty already in $\bbJ_{\dag}$, so the case $m=0$ gives path-connectedness and the cases $m\geq1$ give vanishing homotopy groups.
  Therefore all four mapping spaces of $\bbH$ are weakly contractible.
  They are countable because $\bbH$ is obtained by a countable union of finite cell attachments.

  Flattening the double sequence of cells into an $\omega^2$-indexed transfinite composite shows that $j_{\bbH}:\bbJ_{\dag}=\bbH_0\to\bbH$ is a relative fixed-object cofibration, and it satisfies $\psi j_{\bbH}=v$.
  The unique endpoint-preserving dagger functor $p_{\bbH}:\bbH\to\bbI_{\dag}$ is a local weak equivalence and satisfies $p_{\bbH}j_{\bbH}=q$.
  Hence $\bbH$ is a natural dagger interval and $\psi$ is the required factorization.
\end{proof}

\begin{proposition}\label{bg.lem.realization}
  Let $\calC\in\sCat^{\dag}$ and let $[u]:x\to y$ be coherently unitary.
  For every vertex $u'\in\calC(x,y)_0$ representing $[u]$, there exist a natural dagger interval $\bbH$ and a dagger functor $\phi:\bbH\to\calC$ with $\phi(0)=x$, $\phi(1)=y$ and $\phi(h_{\bbH})=u'$ \emph{on the nose}.
\end{proposition}

\begin{proof}
  Let $r:\calC\langle x,y\rangle\to R(\calC\langle x,y\rangle)$ be the chosen trivial cofibration.  
  Choose a witness $\Phi_0:\bbH_0\to R(\calC\langle x,y\rangle)$ with $[\Phi_0(h_0)]=[r(u')]$.  
  Since $R(\calC\langle x,y\rangle)(0,1)$ is a Kan complex, a finite zigzag between these vertices composes, by horn filling, to a $1$-simplex $e$ with endpoints $\Phi_0(h_0)$ and $r(u')$.

  Attach the free dagger $1$-simplex along its initial vertex $\bbH'_0:=\bbH_0 \amalg_{\bbA_{\dag}(\Delta^{\{0\}})} \bbA_{\dag}(\Delta^1)$ and extend $\Phi_0$ by $e$;
  call the extension $\Phi'_0$ and the terminal vertex of the new simplex $h'_0$. 
  We use the copy of $\bbJ_{\dag}$ classified by $h'_0$ as the new interval source.  
  It factors as
  \begin{align}
    \bbJ_{\dag} \cong \bbA_{\dag}(\Delta^{\{1\}}) \to \bbA_{\dag}(\Delta^1) \to \bbH'_0.
  \end{align}
  The first arrow is a fixed-object cofibration because $\Delta^{\{1\}}\to\Delta^1$ is a monomorphism.
  The second arrow is the pushout of the cofibration $j_{\bbH_0}:\bbA_{\dag}(\Delta^{\{0\}})\to\bbH_0$.
  Hence the composite $\bbJ_{\dag}\to\bbH'_0$ is a fixed-object cofibration.
  On the other hand, $\bbH_0\to\bbH'_0$ is the pushout of the free cell induced by $\Delta^{\{0\}}\to\Delta^1$, hence is a local weak equivalence by \cref{bg.lem.reduction,bg.lem.ordinary-cell}.  
  Thus $\bbH'_0$ has weakly contractible mapping spaces and is countable.  
  Its collapse to $\bbI_{\dag}$ makes it a natural dagger interval, and $\Phi'_0(h'_0)=r(u')$ strictly.

  Factor $\Phi'_0$ as a trivial cofibration $k:\bbH'_0\to\bbH_1$ followed by a fibration $q:\bbH_1\to R(\calC\langle x,y\rangle)$, and form $\bbG:=\calC\langle x,y\rangle\times_{R(\calC\langle x,y\rangle)}\bbH_1$.
  Right properness of \cref{bg.thm.fixed-object}, applied to the weak equivalence $r$ and the fibration $q$, shows that $\bbG\to\bbH_1$ is a local weak equivalence.  
  Since $\bbH'_0\to\bbH_1$ is also one, every mapping space of $\bbG$ is nonempty and weakly contractible. 
  Moreover $(u',k(h'_0))$ is a vertex of $\bbG(0,1)$ because $qk(h'_0)=\Phi'_0(h'_0)=r(u')$. 
  It classifies a dagger functor $v:\bbJ_{\dag}\to\bbG$.

  Apply \cref{bg.lem.extraction} to $v$.  
  The resulting factorization through a natural dagger interval $\bbH$ and the composite $\bbH\to\bbG\to\calC\langle x,y\rangle\to\calC$ send the two formal objects to $x,y$ and $h_{\bbH}$ to $u'$ on the nose.
\end{proof}

\begin{proposition}\label{bg.prop.cu-groupoid}
  Let $\calC,\calD\in\sCat^{\dag}$. 
  Then we have:
  \begin{enumerate}
    \item every identity $[\id_x]$ is coherently unitary;
    \item if $[u]$ is coherently unitary, then $[u]$ is an isomorphism, $[u]^{\dag}=[u]^{-1}$, and $[u^{\dag}]$ is coherently unitary;
    \item composites of composable coherently unitary equivalences are coherently unitary;
    \item every dagger simplicial functor preserves coherent unitarity;
    \item if $f:\calC\to\calD$ is a local weak equivalence, then $[u]$ is coherently unitary if and only if $[f(u)]$ is.
  \end{enumerate}
\end{proposition}

\begin{proof}
  (1) Choose any natural dagger interval, which exists by \cref{bg.lem.interval-existence}, and compose 
  \begin{align}
    \bbH\xrightarrow{p_{\bbH}}\bbI_{\dag}\to\calC\langle x,x\rangle\xrightarrow{r}R\calC\langle x,x\rangle,
  \end{align}
  where the middle functor sends both objects to the formal objects and $u\mapsto\id_x$ (a dagger functor since $\id_x^{\dag}=\id_x$); the generator maps into the class $[\id_x]$.

  (2) By \cref{bg.lem.realization}, we can choose an honest witness $\phi:\bbH\to\calC$ with $\phi(h)=u'$ representing $[u]$.
  The functor $\h\phi$ carries the relations of \cref{bg.lem.interval-basic} (2), so $[u']$ is an isomorphism with $[u']^{\dag}=[u']^{-1}$.
  Moreover $\phi$, read through the relabeled interval $\bbH^{\rev}$ of \cref{bg.lem.interval-basic} (3), is an honest witness for $[u'^{\dag}]=[u]^{\dag}$; composing with $r$ gives coherence of $[u^{\dag}]$.

  (3) Let $[u]:x\to y$ and $[v]:y\to z$ be coherently unitary with honest witnesses $\phi:\bbH\to\calC$, $\psi:\bbK\to\calC$ hitting representatives $u'$, $v'$ (by \cref{bg.lem.realization}).
  Since $\phi(1)=y=\psi(0)$, the two witnesses glue to a dagger functor $\chi:\bbA:=\bbH\amalg_{\mathbf 1_{\dag}}\bbK\to\calC$ on the amalgam.
  By \cref{bg.prop.amalgam}, all mapping spaces of $\bbA$ are nonempty and weakly contractible; hence so are those of the two-object copy $\bbA\langle 0,2\rangle$ (\cref{bg.constr.two-object-copy}, applied to the outer objects).
  The vertex $h_{\bbK}\circ h_{\bbH}\in\bbA(0,2)_0$ defines $\bbJ_{\dag}\to\bbA\langle 0,2\rangle$; \cref{bg.lem.extraction} produces a natural dagger interval $\bbL$ and $\bbL\to\bbA\langle0,2\rangle$ under $\bbJ_{\dag}$.
  Composing with the functor $\bbA\langle0,2\rangle\to\calC\langle x,z\rangle$ induced by $\chi$ and with $r$, we obtain a witness whose generator lands in $[r(v'\circ u')]=[r(v\circ u)]$.

  (4) If $\phi:\bbH\to\calC$ is an honest witness for $[u]$ (which exists by \cref{bg.lem.realization}), then $f\phi$ is an honest witness for $[f(u)]$; compose with $r_{\calD\langle fx,fy\rangle}$.

  (5) The two-object copy $f\langle x,y\rangle:\calC\langle x,y\rangle\to\calD\langle fx,fy\rangle$ is a local weak equivalence.
  Apply \cref{bg.lem.invariance} with $\Theta:=R\calC\langle x,y\rangle$, $s:=r_{\calC\langle x,y\rangle}$ and $\Theta':=R\calD\langle fx,fy\rangle$, $s':=r_{\calD\langle fx,fy\rangle}\circ f\langle x,y\rangle$ (a composite of local weak equivalences into a fibrant object): a witness for $[s'(u)]=[r f(u)]$ exists if and only if one exists for $[s(u)]=[r(u)]$.
  The former existence is, by definition, coherent unitarity of $[f(u)]$, and the latter is coherent unitarity of $[u]$.
\end{proof}

\begin{corollary} \label{bg.cor.cohuni_is_subgroupoid}
  The coherently unitary classes are the morphisms of a wide subgroupoid $(\h\calC)^{\cu}\subseteq(\h\calC)^{\simeq}$.
\end{corollary}

\subsection{The dagger Bergner model structure}

We finally construct the dagger Bergner model structure on the category $\sCat^{\dag}$ (\cref{bg.thm.main}).
We combine local weak equivalences with coherent-unitary essential surjectivity to define dagger DK-equivalences, and include natural intervals among the generating trivial cofibrations.

\begin{definition}\label{bg.def.W}
  A dagger simplicial functor $f:\calC\to\calD$ is a \emph{dagger DK-equivalence} if it satisfies the following conditions:
  \begin{enumerate}
    \item for all $x,y\in\Ob(\calC)$, the map $\calC(x,y)\to\calD(fx,fy)$ is a weak homotopy equivalence;
    \item for every $d\in\Ob(\calD)$ there exist $c\in\Ob(\calC)$ and a coherently unitary equivalence $[v]:fc\to d$ in $\h\calD$.
  \end{enumerate}
\end{definition}

\begin{notation}\label{bg.def.IJ}
  We set
  \begin{align}
    W_{\dag}
    &:= 
    \{\text{dagger DK-equivalences}\},
    \\
    I_{\dag}
    &:=
    \{\bbA_{\dag}(\partial\Delta^n)\to\bbA_{\dag}(\Delta^n)\}_{n\geq0}
    \cup
    \{\varnothing\to\mathbf 1_{\dag}\},
    \\
    J_{\dag}
    &:=
    \{\bbA_{\dag}(\Lambda^n_k)\to\bbA_{\dag}(\Delta^n)\}_{n\geq1,0\leq k\leq n}
    \cup
    \{\iota_{0,\bbH}:\mathbf 1_{\dag}\to\bbH\mid\bbH\in\calH_{\dag}\},
    \\
    J^{\loc}_{\dag}
    &:=
    \{\bbA_{\dag}(\Lambda^n_k)\to\bbA_{\dag}(\Delta^n)\}_{n\geq1,0\leq k\leq n}.
  \end{align}
  Here $\iota_{0,\bbH}$ denotes the endpoint inclusion selecting the object $0$ of $\bbH$.
\end{notation}

\begin{lemma}\label{bg.lem.I-inj}
  Let $f : \calC \to \calD$ be a dagger simplicial functor.
  Then these following are equivalent:
  \begin{enumerate}
    \item the functor $f$ lies in $I_{\dag}\text{-}\inj$;
    \item the functor $f$ is surjective on objects and every morphism $\calC(x,y)\to\calD(fx,fy)$ is a trivial Kan fibration.
  \end{enumerate}
\end{lemma}

\begin{proof}
  Lifting against $\varnothing\to\mathbf 1_{\dag}$ is surjectivity on objects.
  By \eqref{bg.eq.A-univ}, a lifting problem against $\bbA_{\dag}(\partial\Delta^n)\to\bbA_{\dag}(\Delta^n)$ is the same as a pair $(x,y)$ together with an ordinary lifting problem of $\partial\Delta^n\to\Delta^n$ against $\calC(x,y)\to\calD(fx,fy)$; the right lifting property against all these is the trivial Kan fibration condition.
\end{proof}

\begin{corollary}\label{bg.cor.Iinj-in-W}
  $I_{\dag}\text{-}\inj\subseteq W_{\dag}$.
\end{corollary}

\begin{proof}
  Let $f \in I_{\dag}\text{-}\inj$.
  By \cref{bg.lem.I-inj}, it is surjective on objects and locally a weak equivalence.
  That is, for every $d \in \calD$, there exists $c \in \calC$ satisfying $fc=d$.
  Then the object-level condition is witnessed by the identity $\id_{d} : fc = d \to d$, which is coherently unitary by \cref{bg.prop.cu-groupoid} (1).
\end{proof}

\begin{lemma}\label{bg.lem.translation}
  Let $\calC$ be a simplicial category and let $a:x'\to x$ and $b:y\to y'$ be vertices whose classes in $\h\calC$ are isomorphisms.
  Then $b\circ-\circ a:\calC(x,y)\to\calC(x',y')$ is a weak homotopy equivalence.
\end{lemma}

\begin{proof}
  Choose vertices $\bar a:x\to x'$ and $\bar b:y'\to y$ representing the inverse classes of $[a]$ and $[b]$.
  Put $T(f):=b\circ f\circ a$ and $S(g):=\bar b\circ g\circ\bar a$.
  The vertices $\bar b b$, $b\bar b$, $a\bar a$ and $\bar a a$ lie in the same path components as the corresponding identities.
  An edge in a mapping space induces, by simplicial composition, a homotopy after geometric realization between the corresponding composition maps.
  Edge zigzags therefore show that $|ST|$ and $|TS|$ are homotopic to the respective identity maps.
  Thus $|S|$ and $|T|$ are homotopy inverses, so $T$ is a weak homotopy equivalence.
\end{proof}

\begin{proposition}\label{bg.prop.W-23}
  $W_{\dag}$ is closed under retracts and satisfies two-out-of-three.
\end{proposition}

\begin{proof}
  \emph{Retracts.}
  The local condition is closed under retracts since weak equivalences of simplicial sets are.
  Let $f:\calC\to\calD$ be a retract of $g:\calC'\to\calD'$ via $(i_{\calC},i_{\calD},r_{\calC},r_{\calD})$ with $g\in W_{\dag}$, and let $d\in\calD$.
  Choose $c'$ and coherently unitary $[v]:gc'\to i_{\calD}(d)$.
  Then $[r_{\calD}(v)]:f(r_{\calC}c')\to d$ is coherently unitary by \cref{bg.prop.cu-groupoid} (4).

  \emph{Two-out-of-three.}
  Let $\calC\xrightarrow{f}\calD\xrightarrow{g}\calE$.
  \begin{enumerate}
    \item If $f,g\in W_{\dag}$: we show $gf\in W_{\dag}$.
    Local condition:
    it follows objectwise from two-out-of-three for weak homotopy equivalences.
    Object condition: 
    given $e\in\calE$, choose $[w]:gd\to e$ and $[v]:fc\to d$ coherently unitary; then $[w]\circ[g(v)]:gfc\to e$ is coherently unitary by \cref{bg.prop.cu-groupoid} (3,4).
    \item If $f,gf\in W_{\dag}$, we show $g\in W_{\dag}$.
    Object condition: 
    for $e\in\calE$, a coherently unitary $[w]:gf(c)\to e$ serves for $g$ with the object $f(c)$.
    Local condition: 
    fix $d,d'\in\calD$ and choose coherently unitary $[\alpha]:fc\to d$, $[\beta]:fc'\to d'$ with representatives $\alpha,\beta$, and a representative $\bar\beta$ of $[\beta]^{-1}$.
    The maps $\varphi\mapsto\bar\beta\circ\varphi\circ\alpha$ and $\psi\mapsto g(\bar\beta)\circ\psi\circ g(\alpha)$ are weak equivalences $\calD(d,d')\to\calD(fc,fc')$ and $\calE(gd,gd')\to\calE(gfc,gfc')$ by \cref{bg.lem.translation} (the classes $[\alpha],[\beta],[g\alpha],[g\beta]$ are isomorphisms by \cref{bg.prop.cu-groupoid} (2,4)), and they commute strictly with $g$.
    The composite $\calC(c,c')\to\calE(gfc,gfc')$ is a weak equivalence because $gf$ is local, while $\calC(c,c')\to\calD(fc,fc')$ is one because $f$ is local.  
    Hence $\calD(fc,fc')\to\calE(gfc,gfc')$ is a weak equivalence.  
    The strictly commutative translation square and two-out-of-three now give the result for $\calD(d,d')\to\calE(gd,gd')$.
    \item If $g,gf\in W_{\dag}$, we show $f\in W_{\dag}$.
    Local condition:
    for $x,y\in\calC$, the maps $\calC(x,y)\to\calE(gfx,gfy)$ and $\calD(fx,fy)\to\calE(gfx,gfy)$ are weak equivalences, hence so is $\calC(x,y)\to\calD(fx,fy)$.
    Object condition: 
    let $d\in\calD$.
    Choose coherently unitary $[w]:gf(c)\to g(d)$ ($gf\in W_{\dag}$).
    Since $g$ is a local weak equivalence, $\pi_0\calD(fc,d)\to\pi_0\calE(gfc,gd)$ is bijective; let $[u]:fc\to d$ be the class with $[g(u)]=[w]$.
    By \cref{bg.prop.cu-groupoid} (5), $[u]$ is coherently unitary.
  \end{enumerate}
\end{proof}

\begin{lemma}\label{bg.lem.interval-attachment}
  Let $\calC\in\sCat^{\dag}$, $x\in\calC$, and $\bbH\in\calH_{\dag}$.
  Then the pushout inclusion $\calC\to\calC\langle\bbH\rangle:=\calC\amalg_{\mathbf 1_{\dag}}\bbH$ (gluing $x=0$) lies in $W_{\dag}\cap I_{\dag}\text{-}\cof$.
\end{lemma}

\begin{proof}
  It is a pushout of $\mathbf 1_{\dag}\to\bbH\in I_{\dag}\text{-}\cof$ (\cref{bg.1toH_is_cofibration}), hence in $I_{\dag}\text{-}\cof$.
  The mapping spaces of $\bbH$ are nonempty and weakly contractible (\cref{bg.lem.interval-basic} (1)), so \cref{bg.prop.extension} (1) gives the local condition.

  For the object condition, old objects are hit strictly.
  The canonical functor $\bbH\to\calC\langle\bbH\rangle$ with endpoints $x$ and $y$ induces a fixed-object functor $\bbH\to\calC\langle\bbH\rangle\langle x,y\rangle$.
  Composing it with $r_{\calC\langle\bbH\rangle\langle x,y\rangle}$ gives the required coherent-unitary witness for the image of $h_{\bbH}$.
\end{proof}

\begin{proposition}\label{bg.prop.Jcell}
  $J_{\dag}\text{-}\cell\subseteq W_{\dag}\cap I_{\dag}\text{-}\cof$.
\end{proposition}

\begin{proof}
  A pushout of $\bbA_{\dag}(\Lambda^n_k)\to\bbA_{\dag}(\Delta^n)$ is, by \cref{bg.lem.reduction}, a composite of two ordinary free cell attachments along the trivial cofibration $\Lambda^n_k\to\Delta^n$; by \cref{bg.lem.ordinary-cell} (1,2) it is bijective on objects, locally a monomorphism, and locally a weak equivalence; the object condition holds via identities.
  It lies in $I_{\dag}\text{-}\cof$ since $\Lambda^n_k\to\Delta^n$ is a relative $\{\partial\Delta^m\to\Delta^m\}$-cell complex and $\bbA_{\dag}(-)$ preserves cell structures.
  A pushout of an interval generator is covered by \cref{bg.lem.interval-attachment}.

  For transfinite composites $\calC_0\to\colim_{\alpha}\calC_{\alpha}$: membership in $I_{\dag}\text{-}\cof$ is automatic.
  All transition maps are locally monomorphisms (\cref{bg.lem.ordinary-cell} (1) and \cref{bg.prop.extension} (1)) and locally weak equivalences on the objects present at each stage; since filtered colimits of weak equivalences of simplicial sets along monomorphisms are weak equivalences, the composite is a local weak equivalence.
  Every object of the colimit appears at some stage, where it is either present in $\calC_0$ or connected to an earlier object by a coherently unitary equivalence (\cref{bg.lem.interval-attachment}); coherent unitarity is preserved by the leg functors into the colimit (\cref{bg.prop.cu-groupoid} (4)) and composes (\cref{bg.prop.cu-groupoid} (3)), so the composite is coherently essentially surjective.
\end{proof}

\begin{lemma}\label{bg.lem.Iinj-in-Jinj}
  $I_{\dag}\text{-}\inj\subseteq W_{\dag}\cap J_{\dag}\text{-}\inj$.
\end{lemma}

\begin{proof}
  Let $f\in I_{\dag}\text{-}\inj$.
  By \cref{bg.lem.I-inj}, $f$ is surjective on objects and locally a trivial Kan fibration.
  By \cref{bg.cor.Iinj-in-W}, it lies in $W_{\dag}$.
  Trivial Kan fibrations have the right lifting property against all monomorphisms, in particular against horns, so $f$ lifts against $J^{\loc}_{\dag}$ by \eqref{bg.eq.A-univ}.
  
  For an interval generator, consider a square with upper map $c:\mathbf 1_{\dag}\to\calC$ and lower map $\phi:\bbH\to\calD$, $\phi(0)=fc$.
  Choose $c'\in f^{-1}(\phi(1))$ and regard $\phi$ as a map $\bbH\to\calD\langle fc,\phi(1)\rangle$ in $\sCat^{\dag}_{\{0,1\}}$.
  The induced morphism $f\langle c,c'\rangle:\calC\langle c,c'\rangle\to\calD\langle fc,\phi(1)\rangle$ is a local trivial Kan fibration, i.e.\ a trivial fibration in the fixed-object model structure, and
  $\bbH$ is cofibrant there because $\bbJ_{\dag}$ is cofibrant and $j_{\bbH}$ is a cofibration.
  Hence $\phi$ lifts to $\ell:\bbH\to\calC\langle c,c'\rangle$; composing with the evaluation $\calC\langle c,c'\rangle\to\calC$ solves the original lifting problem.
\end{proof}

\begin{lemma}\label{bg.lem.WJinj-in-Iinj}
  $W_{\dag}\cap J_{\dag}\text{-}\inj\subseteq I_{\dag}\text{-}\inj$.
\end{lemma}

\begin{proof}
  We use \cref{bg.lem.I-inj}.
  Let $f:\calC\to\calD$ lie in $W_{\dag}\cap J_{\dag}\text{-}\inj$.
  Lifting against $J^{\loc}_{\dag}$ makes every $\calC(x,y)\to\calD(fx,fy)$ a Kan fibration; being also a weak equivalence, it is a trivial Kan fibration.
  
  It remains to prove surjectivity on objects.
  Let $d\in\calD$ and choose a coherently unitary $[v]:fc\to d$.
  By \cref{bg.lem.realization}, there exist a natural dagger interval and an honest map into $\calD$ with endpoints $fc$ and $d$; we obtain $\psi:\bbH\to\calD$ with $\psi(0)=fc$, $\psi(1)=d$.
  Replace $\bbH$ by its chosen isomorphic representative in $\calH_{\dag}$ and precompose $\psi$ with that isomorphism.  
  Thus $\mathbf1_{\dag}\to\bbH$ is one of the actual generators in $J_{\dag}$.
  The square with upper map $c$ and lower map $\psi$ admits a lift $\ell$ since $f\in J_{\dag}\text{-}\inj$; then $c':=\ell(1)$ satisfies $f(c')=d$.
\end{proof}

\begin{theorem}[The dagger Bergner model structure]\label{bg.thm.main}
  There exists a combinatorial model structure $\sCat^{\dag}_{\Bergner}$ on $\sCat^{\dag}$ in which:
  \begin{enumerate}
    \item the weak equivalences are the dagger DK-equivalences $W_{\dag}$;
    \item the generating cofibrations are $I_{\dag}$ and the generating trivial cofibrations are $J_{\dag}$;
    \item the fibrations are $J_{\dag}\text{-}\inj$, and the trivial fibrations are exactly the functors which are surjective on objects and locally trivial Kan fibrations.
  \end{enumerate}
\end{theorem}

\begin{proof}
  We apply \cite[Theorem 2.1.19]{Hov99} to $(\sCat^{\dag},W_{\dag},I_{\dag},J_{\dag})$.
  The collection $\calH_{\dag}$ is a set by \cref{bg.lem.interval-existence}, so $I_{\dag}$ and $J_{\dag}$ are sets.
  The domains of their members are $\varnothing$, $\mathbf1_{\dag}$ and objects $\bbA_{\dag}(K)$ with $K$ finite.
  The objects $\bbA_{\dag}(K)$ are finitely presentable by \eqref{bg.eq.A-univ}, and $\mathbf1_{\dag}$ is finitely compact because filtered colimits are created on objects and mapping simplicial sets.
  Hence the smallness hypotheses of the recognition theorem hold.

  $W_{\dag}$ is closed under retracts and satisfies two-out-of-three (\cref{bg.prop.W-23}).
  $J_{\dag}\text{-}\cell\subseteq W_{\dag}\cap I_{\dag}\text{-}\cof$ is \cref{bg.prop.Jcell};
  $I_{\dag}\text{-}\inj\subseteq W_{\dag}\cap J_{\dag}\text{-}\inj$ is \cref{bg.lem.Iinj-in-Jinj};
  and $W_{\dag}\cap J_{\dag}\text{-}\inj\subseteq I_{\dag}\text{-}\inj$ is \cref{bg.lem.WJinj-in-Iinj}.
  The theorem yields the cofibrantly generated model structure with fibrations $J_{\dag}\text{-}\inj$ and trivial fibrations $I_{\dag}\text{-}\inj$, identified in \cref{bg.lem.I-inj}; 
  combinatoriality follows from presentability.
\end{proof}

\begin{proposition}\label{bg.thm.left-proper}
  The dagger Bergner model structure $\sCat^{\dag}_{\Bergner}$ is left proper.
\end{proposition}

\begin{proof}
  Let $f:\calC\to\calC'$ be a dagger DK-equivalence and let $i:\calC\to\calB$ be a cofibration.  
  Write $f':\calB\to\calB':=\calB\amalg_{\calC}\calC'$ for its cobase change.

  We compare this square with the ordinary Bergner model structure.  
  The forgetful functor $\UdagSCat:\sCat^{\dag}\to\sCat$ creates pushouts (by \cref{bg.lem.creation}).  
  It also carries dagger Bergner cofibrations to ordinary Bergner cofibrations.  
  Indeed, it carries $\varnothing\to\mathbf1_{\dag}$ to the ordinary object generator, and it carries every local dagger generating cofibration to a composite of two ordinary local generating cofibrations (by \cref{bg.lem.reduction}). 
  The assertion follows for relative cell complexes and then for their retracts.

  The functor $\UdagSCat f$ is an ordinary DK-equivalence. 
  It is locally a weak equivalence, and a coherent-unitary class is in particular an isomorphism in the ordinary homotopy category, so coherent-unitary essential surjectivity implies ordinary essential surjectivity.  
  The ordinary Bergner model structure is left proper by \cite[Proposition~A.3.2.4]{HTT}.
  Applying ordinary left properness to the pushout square created by $\UdagSCat$ shows that $\UdagSCat f'$ is an ordinary DK-equivalence.  
  In particular, $f'$ is a local weak equivalence.

  It remains only to recover the coherent-unitary object condition.  
  Every object of $\calB'$ is represented either by an object of $\calB$ or by an object $d$ of $\calC'$.  
  Objects of the first kind lie strictly in the image of $f'$.  
  For the second kind choose $c\in\calC$ and a coherent-unitary class $[v]:f(c)\to d$.  
  The canonical functor $\calC'\to\calB'$ carries $[v]$ to a coherent-unitary class by \cref{bg.prop.cu-groupoid} (4), and this class joins $f'(i(c))$ to $d$.  
  Thus $f'$ is coherently unitarily essentially surjective and hence belongs to $W_{\dag}$.
\end{proof}

\begin{remark}\label{bg.rem.not-transfer}
  The model structure $\sCat^{\dag}_{\Bergner}$ is not obtained by transferring the Bergner model structure along the free--forgetful adjunction:
  a transferred structure would impose essential surjectivity up to arbitrary homotopy equivalences, whereas \cref{bg.def.W} imposes it up to coherently unitary ones.
  For dagger $1$-categories the coherent notion agrees with the classical one (by \cref{bg.lem.discrete}), so the classical dagger category theory is recovered.
\end{remark}

\begin{corollary}\label{bg.cor.fibrant}
  An object $\calC\in\sCat^{\dag}$ is fibrant if and only if every mapping space $\calC(x,y)$ is a Kan complex.
\end{corollary}

\begin{proof}
  Lifting of $\calC\to\mathbf 1_{\dag}$ against $J^{\loc}_{\dag}$ is the Kan condition by \eqref{bg.eq.A-univ}.
  Lifting against an interval generator $\mathbf 1_{\dag} \to \bbH$ with an upper map $c:\mathbf 1_{\dag}\to\calC$ always exists:
  we can take it as the composite $\bbH\xrightarrow{p_{\bbH}}\bbI_{\dag}\to\calC$, where the second functor sends both objects to $c$ and $u\mapsto\id_c$.
\end{proof}

\begin{corollary}\label{bg.prop.semantic}
  Every fibration $f:\calC\to\calD$ is locally a Kan fibration, and $(\h f)^{\cu}:(\h\calC)^{\cu}\to(\h\calD)^{\cu}$ is an isofibration: 
  for every $x\in\calC$ and every coherently unitary $[v]:fx\to d$, there exist $x'\in\calC$ with $fx'=d$ and a coherently unitary $[u]:x\to x'$ with $[f(u)]=[v]$.
\end{corollary}

\begin{proof}
  Lifting against $J_{\dag}^{\loc}$ gives the local Kan-fibration condition.
  Given $[v]$, \cref{bg.lem.realization} provides $\psi:\bbH\to\calD$ with $\psi(0)=fx$, $\psi(1)=d$ and $\psi(h)$ representing $[v]$.
  After replacing $\bbH$ by its representative in $\calH_{\dag}$, this is a lifting problem against an actual interval generator.

  Consider a lifting problem against the interval generator $\mathbf 1_{\dag} \to \bbH$ with an upper map $x:\mathbf 1_{\dag}\to\calC$ selecting the given object.
  By assumption, we have a lift $\ell : \bbH \to \calC$.
  Take $x' = \ell(1)$ and $u = \ell(h_{\bbH})$, then $fx' = \psi(1) = d$ and $f(u) = \psi(h_{\bbH})$.
  The functor $\ell$ induces $\bbH\to\calC\langle x,x'\rangle$.
  Composing with $r_{\calC\langle x,x'\rangle}$ witnesses coherent unitarity of $[u]$.
\end{proof}

\begin{remark}\label{bg.rem.semantic-converse}
  We do not assert the converse of \cref{bg.prop.semantic}.
  The claim that every locally-Kan-fibered functor which is an isofibration on \emph{weakly} unitary classes has the right lifting property against the interval generators is precisely the step invalidated by \cref{bg.prop.weak-not-coherent}; whether the converse holds with coherently unitary classes is not needed in this paper and is left open.
\end{remark}

\begin{proposition}\label{pointset.thm.discrete-sector}
  Let $F:\calA\to\calB$ be a dagger functor between dagger $1$-categories, regarded as dagger simplicial categories with discrete mapping spaces.
  Then:
  \begin{enumerate}
    \item $F$ is a dagger DK-equivalence if and only if it is fully faithful and unitarily essentially surjective;
    \item $F$ is a dagger Bergner fibrantion if and only if it is a unitary isofibration.
  \end{enumerate}
\end{proposition}

\begin{proof}
  (1)
  A map between discrete simplicial sets is a weak homotopy equivalence if and only if it is a bijection.
  Thus the local weak-equivalence condition for $F$ is full faithfulness.
  By \cref{bg.lem.discrete}, coherent unitarity agrees with strict unitarity in a dagger one-category.
  Hence coherent-unitary essential surjectivity is ordinary unitary essential surjectivity.

  (2)
  Suppose that $F$ is a dagger Bergner fibration.
  By \cref{bg.prop.semantic}, it is an isofibration on coherently unitary morphisms.
  The discrete identification just recalled makes it a unitary isofibration.

  Conversely, suppose that $F$ is a unitary isofibration.
  Every map between discrete simplicial sets is a Kan fibration, so $F$ is locally a Kan fibration.
  It remains to solve a lifting problem against a natural dagger interval inclusion $\mathbf1_{\dag}\to\bbH$.
  
  A dagger functor $\psi:\bbH\to\calB$ factors through $\pi_0\bbH$, because the mapping spaces of $\calB$ are discrete.
  Since $p_{\bbH}:\bbH\to\bbI_{\dag}$ is an identity-on-objects local weak equivalence, \cref{bg.lem.interval-basic} gives an isomorphism $\pi_0p_{\bbH}: \pi_0\bbH \xrightarrow{\cong} \bbI_{\dag}$.

  If $\overline\psi:\pi_0\bbH\to\calB$ is the induced dagger functor, put $ \widetilde\psi := \overline\psi(\pi_0p_{\bbH})^{-1}: \bbI_{\dag} \to \calB$.
  Then $\psi=\widetilde\psi p_{\bbH}$, so $\psi$ is determined by a unitary morphism $v:F(a)\to b$.
  Since $F$ is a unitary isofibration, there are an object $a'$ and a unitary morphism $u:a\to a'$ with $F(a')=b$ and $F(u)=v$.
  The morphism $u$ determines a dagger functor $\bbI_{\dag}\to\calA$.
  Composing it with $p_{\bbH}$ gives the required lift.
  Hence $F$ is a dagger Bergner fibration.
\end{proof}

\begin{corollary}\label{pointset.cor.bunke}
  On the full subcategory of ordinary dagger categories, the dagger DK-equivalences and dagger Bergner fibrations are respectively the weak equivalences and fibrations of Bunke's model structure in the sense of \cite{Bunke19}.
\end{corollary}

\begin{proof}
  Bunke's weak equivalences are the unitary equivalences, namely the fully faithful functors which are unitarily essentially surjective \cite[Lemma~5.3]{Bunke19}.
  
  Indeed, full faithfulness and a choice of unitary arrows from objects in the image to all target objects construct a quasi-inverse whose unit and counit are unitary.
  
  Conversely, a quasi-inverse with unitary unit and counit implies full faithfulness and unitary essential surjectivity.
  Bunke's fibrations are the unitary isofibrations \cite[Definition~9.5 and Corollary~9.9]{Bunke19}.
  The two identifications therefore follow from \cref{pointset.thm.discrete-sector}.
\end{proof}

\begin{proposition}\label{pointset.prop.discrete-cofibrations}
  Under the identification of ordinary dagger categories with dagger simplicial categories having discrete mapping spaces, every dagger Bergner cofibration between such discrete objects is a cofibration in Bunke's model structure.
\end{proposition}

\begin{proof}
  Bunke's cofibrations are precisely the dagger functors which are injective on objects \cite[Definition~9.1]{Bunke19}.
  The proof of \cref{bg.thm.left-proper} shows that the forgetful functor $\UdagSCat:\sCat^{\dag}\to\sCat$ sends dagger Bergner cofibrations to ordinary Bergner cofibrations.
  In particular, it is injective on objects and is therefore a Bunke cofibration.
\end{proof}

\begin{remark}
  The converse of \cref{pointset.prop.discrete-cofibrations} does not hold in general:
  let $\calC=BC_2$ be the one-object dagger groupoid with dagger given by inversion \cite[Example~2.2]{Bunke19}, and let $r:\calC\to\mathbf1_{\dagger}$ be the collapse functor.
  The functor $r$ is bijective on objects, and hence is a Bunke cofibration.
  Its map on the unique mapping simplicial set is $C_2\to\Delta^0$, whose map in degree zero is not injective.
  If $r$ is a dagger Bergner cofibration, then its underlying simplicial functor is an ordinary Bergner cofibration.
  Ordinary Bergner cofibrations induce cofibrations, hence monomorphisms, on the mapping simplicial sets \cite[Remark~A.3.2.5]{HTT}.
  This contradicts the non-injectivity in degree zero.
  Thus $r$ is not a dagger Bergner cofibration.
\end{remark}

\section{The dagger Joyal model structure}\label{dj.section}

In this section, we transport the dagger Bergner model structure of \cref{bg.thm.main} along the rigidification adjunction, and obtain a model structure on $\sSet^{\dag}$ whose weak equivalences are created by $\frakC_{\dag}$, together with a Quillen equivalence $\frakC_{\dag}: \sSet^{\dag}_{\Joyal} \rightleftarrows \sCat^{\dag}_{\Bergner} :N_{\dag}$ (\cref{dj.thm.equivalence}).

It is tempting to take the cofibrations of $\sSet^{\dag}_{\Joyal}$ to be \emph{all} monomorphisms, as in the ordinary Joyal model structure.
This is impossible if the displayed adjunction is to be a Quillen adjunction: we exhibit in \cref{dj.prop.counterexample} a monomorphism of dagger simplicial sets whose image under $\frakC_{\dag}$ is \emph{not} a cofibration in the dagger Bergner model structure.
The culprit is the existence of \emph{self-conjugate} simplices ($\sigma^{\dag}=\sigma$), which have no analogue in the ordinary theory.

\subsection{Dagger quasi-categories}

We begin with simplicial sets carrying a strict reversal involution that fixes every vertex (\cref{def.dagger_simplicial_sets}).
Dagger quasi-categories and dagger Kan complexes are then defined by imposing the usual quasi-category and Kan conditions on the underlying simplicial sets.

\begin{definition} \label{def.dagger_simplicial_sets}
  A \emph{dagger simplicial set} $(K,\dag_{K})$ consists of a simplicial set $K : \prism^{\myop} \to \Set$ and a natural transformation $\dag_{K} : K \to K^{\myop}$ satisfying 
  $\dag_{K}^{\myop} \dag_{K} = \id_{K}$ and $\dag_{K}|_{0} = \id_{K_{0}}$.

  Let $(K,\dag_{K})$ and $(L,\dag_{L})$ be dagger simplicial sets.
  A \emph{dagger morphism of dagger simplicial sets} $f : (K,\dag_{K}) \to (L,\dag_{L})$ is a morphism of simplicial sets $f : K \to L$ satisfying $\dag_{L}f = f^{\myop} \dag_{K}$.

  We let $\sSet^{\dag}$ denote the category of dagger simplicial sets with dagger morphisms.
\end{definition}

\begin{remark}
  A dagger simplicial set $(K,\dag_{K})$ is equivalently a simplicial set $K$ together with maps $(-)^{\dag} : K_n \to K_n$ for all $n \geq 0$, satisfying
  \begin{align}
    d_i(\sigma^{\dag})=(d_{n-i}\sigma)^{\dag}, 
    \quad
    s_i(\sigma^{\dag})=(s_{n-i}\sigma)^{\dag}, 
    \quad
    (\sigma^{\dag})^{\dag}=\sigma,
    \quad \text{and} \quad
    x^{\dag}=x
  \end{align}
  for every $n$-simplex $\sigma \in K$ and every vertex $x \in K$.
\end{remark}

\begin{remark}
  \Cref{def.dagger_simplicial_sets} is the strict simplicial analogue of an ordinary dagger category.
  The dagger $\dag_{K} : K \to K^{\myop}$ reverses the orientation of every simplex and the condition $\dag_{K}|_{0}=\id_{K_{0}}$ records the usual convention that a dagger operation fixes objects.
  
  A dagger morphism is therefore not just a map of the underlying simplicial sets: 
  it must commute with this reversal, just as an ordinary dagger functor satisfies $ F(f^{\dag})=F(f)^{\dag}$.
\end{remark}

\begin{notation}
  We will usually denote a dagger simplicial set $(K,\dag_{K})$ simply by $K$, omitting the dagger structure $\dag_{K}$ from the notation.
  Similarly, we will denote a morphism of dagger simplicial sets $f : (K,\dag_{K}) \to (L,\dag_{L})$ simply by $f : K \to L$.
\end{notation}

\begin{definition}
  Let $K$ be a dagger simplicial set.
  We will say that it is a \emph{dagger quasi-category} if the underlying simplicial set $K$ is a quasi-category.

  We will refer to dagger morphisms of dagger quasi-categories as \emph{dagger functors}.

  We let $\qCat^{\dag} \subseteq \sSet^{\dag}$ denote the full subcategory spanned by the dagger quasi-categories.
\end{definition}

\begin{definition}
  Let $K$ be a dagger simplicial set.
  We will say that it is a \emph{dagger Kan complex} if the underlying simplicial set $K$ is a Kan complex.

  We let $\Kan^{\dag} \subseteq \qCat^{\dag}$ denote the full subcategory spanned by the dagger Kan complexes.
\end{definition}

\begin{example}
  Let $G$ be a group, and let $\theta : G \to G$ be an involutive anti-automorphism on $G$,
  i.e. it satisfies $\theta(gh) = \theta(h)\theta(g)$ and $\theta^{2} = \id_{G}$.
  Then $\theta$ defines a dagger structure on a one-object groupoid $\mathbf{B}G$ by setting $g^{\dag} := \theta(g)$ for every $g \in G$.

  Since $\mathbf{B}G$ is a groupoid, its nerve $N\mathbf{B}G$ is a Kan complex.  
  The dagger structure on $\mathbf{B}G$ induces a dagger structure $\dag_{N\mathbf{B}G} : N\mathbf{B}G \to (N\mathbf{B}G)^{\myop}$.
  Explicitly, for an $n$-simplex $(g_1,\dots,g_n) \in G^{n}$ of $N\mathbf{B}G$, the dagger structure is given by $(g_1,\dots,g_n)^{\dag} := (\theta(g_n),\dots,\theta(g_1))$.
  Therefore the pair $(N\mathbf{B}G,\dag_{N\mathbf{B}G})$ defines a dagger Kan complex.

  Let $(G,\theta)$ and $(H,\eta)$ be groups equipped with involutive anti-automorphisms.  
  A morphism $ N\mathbf{B}G \to N\mathbf{B}H$ of simplicial sets is the same as a functor $\mathbf{B}G \to \mathbf{B}H$, hence the same as a group homomorphism $\varphi : G \to H$.
  It defines a morphism in $\Kan^{\dag}$ if and only if it commutes with the chosen anti-involutions,
  i.e. it satisfies $\varphi(\theta(g)) = \eta(\varphi(g))$ for every $g \in G$.
\end{example}

\subsection{The rigidification adjunction}

We first identify dagger simplicial sets with presheaves on $\prism_{\rev}$ (\cref{prop.sSetdag_is_equivalent_to_Fun}).
We then lift the ordinary rigidification--nerve adjunction and record its compatibility with free dagger completion and with the cofibrations used below.

\begin{construction}\label{constr.alpha_rev}
  Let $\alpha : [m] \to [n]$ be a morphism in $\prism$.
  We define a morphism $\alpha^{\rev} : [m] \to [n]$ by $\alpha^{\rev}(i) := n - \alpha(m-i)$.
  The assignments $[n] \mapsto [n]$ and $\alpha \mapsto \alpha^{\rev}$ define an involutive functor $\rev : \prism \to \prism$.
\end{construction}

\begin{definition}
  We define the category $\prism_{\rev}$ as follows:
  \begin{itemize}
    \item the objects are the same as $\prism$;
    \item a morphism $\alpha:[m]\to[n]$ in $\prism_{\rev}$ is a map which is either order-preserving or order-reversing:
    that is, either $i \leq j  \to  \alpha(i) \leq \alpha(j)$ for all $i,j\in[m]$, or $i \leq j  \to  \alpha(i) \geq \alpha(j)$ for all $i,j\in[m]$.
  \end{itemize}
\end{definition}

\begin{proposition} \label{prop.sSetdag_is_equivalent_to_Fun}
  There exists an equivalence of categories 
  \begin{align}
    \sSet^{\dag} \simeq \Fun(\prism_{\rev}^{\myop},\Set);
  \end{align}
  in particular, the category $\sSet^{\dag}$ is presentable.
\end{proposition}

\begin{proof}
  For every $n\geq0$, let $r_n:[n]\to[n] : i \mapsto n-i$ denote the order-reversing involution. 
  Thus $r_n$ is an automorphism of $[n]$ in $\prism_{\rev}$ and $r_n^2=\id_{[n]}$.

  For a dagger simplicial set $(X,\dag_X)$, define a functor $\widetilde X:\prism_{\rev}^{\myop}\to\Set$ by 
  \begin{align} 
    \widetilde X([n]) &:= X_n, \\
    \widetilde X(\alpha) := \alpha^* : X_n \to X_m 
    \quad &\text{or} \quad
    \widetilde X(\beta) := (r_n\beta)^*\circ \dag_n : X_n \to X_m,
  \end{align}
  for an order-preserving map $\alpha:[m]\to[n]$ and an order-reversing map $\beta:[m]\to[n]$.
  These definitions agree on constant maps.

  For morphisms $[\ell]\xrightarrow{\mu}[m]\xrightarrow{\nu}[n]$ in $\prism_{\rev}$, we can show $\widetilde X(\nu\mu) = \widetilde X(\mu)\circ \widetilde X(\nu)$ for $\nu$ and $\mu$ are order-preserving or order-reversing, respectively.

  Conversely, let $F:\prism_{\rev}^{\myop}\to\Set$ be a presheaf.  
  Its restriction to $F|_{\prism^{\myop}}$ is a simplicial set.
  The order-reversing automorphism $r_n:[n]\to[n]$ gives a map 
  \begin{align}
    \dag_n:=F(r_n):(F|_{\prism^{\myop}})_n\to (F|_{\prism^{\myop}})_n.
  \end{align}
  Since $r_n^2=\id_{[n]}$ and $r_0=\id_{[0]}$, we have $\dag_n^2=\id_{(F|_{\prism^{\myop}})_n}$ and $\dag_0$ is the identity on vertices.
  For an order-preserving map $\alpha:[m]\to[n]$, the equality $\alpha r_m = r_n\alpha^{\rev}$ in $\prism_{\rev}$ gives 
  \begin{align}
    \dag_m\circ \alpha^* = (\alpha^{\rev})^*\circ \dag_n.
  \end{align}
  Hence the maps $\dag_n$ assemble to a natural transformation $\dag_{F|_{\prism^{\myop}}}:F|_{\prism^{\myop}}\to (F|_{\prism^{\myop}})^{\myop}$.
  Therefore $(F|_{\prism^{\myop}},\dag_{F|_{\prism^{\myop}}})$ is a dagger simplicial set.
  We can check that these two constructions are inverse to each other.
\end{proof}

\begin{notation}\label{dj.not.adjunctions}
  We recall:
  \begin{itemize}
    \item $\FdagSSet:\sSet\rightleftarrows\sSet^{\dag}:\UdagSSet$ for the free--forgetful adjunction on simplicial sets, $\FdagSSet(A)=A\amalg_{A^{(0)}}A^{\myop}$ with the swap involution.
    Here, for a simplicial set $A$, we write $A^{(0)}=\sk_0A$ for the simplicial subset consisting of the totally degenerate simplices;
    \item $\FdagSCat:\sCat\rightleftarrows\sCat^{\dag}:\UdagSCat$ for the free--forgetful adjunction on simplicial categories;
    \item $\FCatDag:\sGph^{\dag}\rightleftarrows\sCat^{\dag}:\UCatDag$ for the free-category--underlying-graph adjunction in the dagger setting;
    \item $\frakC:\sSet\rightleftarrows\sCat:N$ for the ordinary rigidification adjunction.
  \end{itemize}
\end{notation}

\begin{remark}
  Explicitly, there exists a natural isomorphism $\UdagSCat\FdagSCat(\calA)\cong \calA\amalg_{\Ob\calA}\calA^{\myop}$.
  Here $\Ob\calA$ is regarded as a discrete simplicial category, and the dagger exchanges the two factors and reverses alternating words.

  This is different from the graph-to-category functor $\FCatDag$ used in \cref{bg.lem.presentable}: 
  the latter freely adds composition to a dagger graph, whereas $\FdagSCat$ freely adds formal daggers to an already composed simplicial category.  
  The asserted category-level adjunction follows directly from the pushout universal property: 
  restriction to the first factor gives a natural bijection 
  \begin{align}
    \Hom_{\sCat^{\dag}}(\FdagSCat\calA,\calC) \cong \Hom_{\sCat}(\calA,\UdagSCat\calC).
  \end{align}
  Indeed, once a simplicial functor on the first factor is fixed, its value on the opposite factor is forced by the dagger, and the two restrictions agree on the common discrete object category.  
  This also proves naturality in both variables.
\end{remark}

\begin{lemma}\label{dj.lem.lift}
  The adjunction $\frakC\dashv N$ lifts to an adjunction
  \begin{align}
    \frakC_{\dag}:\sSet^{\dag}\rightleftarrows\sCat^{\dag}:N_{\dag},
  \end{align}
  with $\UdagSCat\circ\frakC_{\dag}=\frakC\circ \UdagSSet$, $\UdagSSet\circ N_{\dag}=N\circ \UdagSCat$, and $\frakC_{\dag}\circ \FdagSSet \cong \FdagSCat\circ\frakC$ naturally.
\end{lemma}

\begin{proof}
  There exists a canonical isomorphism $\chi_X:\frakC(X^{\myop})\xrightarrow{\cong}\frakC(X)^{\myop}$.
  Since $r_n^2=\id$, we have
  \begin{align}
    \chi_X^{\myop}\chi_{X^{\myop}}
    =
    \id_{\frakC(X)}
    \quad \text{and} \quad
    \chi_{X^{\myop}}^{\myop}\chi_X
    =
    \id_{\frakC(X^{\myop})}.
  \end{align}
  The corresponding comparison for the nerve is $\lambda_{\calC}:N(\calC^{\myop})\xrightarrow{\cong}N(\calC)^{\myop}$.
  It satisfies
  \begin{align}
    \lambda_{\calC}^{\myop}\lambda_{\calC^{\myop}}
    =
    \id_{N(\calC)}
    \quad \text{and} \quad
    \lambda_{\calC^{\myop}}^{\myop}\lambda_{\calC}
    =
    \id_{N(\calC^{\myop})}.
  \end{align}

  Given $(X,\dag_X)\in\sSet^{\dag}$, define $\frakC_{\dag}(X)$ to be $\frakC(\UdagSSet X)$ with the dagger
  \begin{align}
    D_{\frakC X}
    :=
    \chi_{\UdagSSet X}\frakC(\dag_X):
    \frakC(\UdagSSet X)
    \to
    \frakC(\UdagSSet X)^{\myop}.
  \end{align}
  It is the identity on objects because $\dag_X$ and $\chi_{\UdagSSet X}$ fix vertices.
  Naturality of $\chi$, its strict coherence, and $\dag_X^{\myop}\dag_X=\id$ give $D_{\frakC X}^{\myop}D_{\frakC X} = \id$.
  Thus it is a dagger.

  Dually, for $(\calC,\dag_{\calC})\in\sCat^{\dag}$, define the dagger of $N_{\dag}(\calC)$ by
  \begin{align}
    D_{N\calC}
    :=
    \lambda_{\UdagSCat\calC}N(\dag_{\calC}):
    N(\UdagSCat\calC)
    \to
    N(\UdagSCat\calC)^{\myop}.
  \end{align}
  The map $\lambda_{\UdagSCat\calC}$ is the identity in degree zero because $r_0=\id$, and $\dag_{\calC}$ is the identity on objects.
  Hence $D_{N\calC}$ fixes vertices.
  Naturality and strict coherence of $\lambda$ give $D_{N\calC}^{\myop}D_{N\calC} = \id$.
  Both assignments are natural on dagger morphisms.

  We spell out the mate compatibility.
  If $F:\frakC(X)\to\calC$ has ordinary adjoint $f:X\to N(\calC)$, then the mate of $F^{\myop}\chi_X$ is $\lambda_{\calC}^{-1}f^{\myop}:X^{\myop}\to N(\calC^{\myop})$.
  Therefore
  \begin{align}
    \dag_{\calC}F
    =
    F^{\myop}\chi_X\frakC(\dag_X)
    \quad\Leftrightarrow\quad
    N(\dag_{\calC})f
    =
    \lambda_{\calC}^{-1}f^{\myop}\dag_X
    \quad\Leftrightarrow\quad
    \lambda_{\calC}N(\dag_{\calC})f
    =
    f^{\myop}\dag_X.
  \end{align}
  Thus a simplicial functor $\frakC(\UdagSSet X)\to \UdagSCat\calC$ is dagger-compatible if and only if its adjoint $\UdagSSet X\to N(\UdagSCat\calC)$ is dagger-compatible.
  Hence the ordinary adjunction bijection restricts to the dagger morphisms.
  This proves $\frakC_{\dag}\dashv N_{\dag}$ and also shows that its unit and counit are dagger morphisms.

  Finally, the two composites $\frakC_{\dag}\circ \FdagSSet$ and $\FdagSCat\circ\frakC$ are left adjoint to the same functor
  $\UdagSSet\circ N_{\dag}=N\circ \UdagSCat$.
  We take the displayed natural isomorphism to be the mate of the identity of this common right adjoint.
  By its construction, it is compatible with the units and counits of the two composite adjunction structures.
\end{proof}

\begin{lemma}\label{dj.lem.F-preserves-cof}
  We have:
  \begin{enumerate}
    \item the functor $\FdagSCat$ sends Bergner cofibrations to dagger Bergner cofibrations;
    \item if a Bergner trivial cofibration is bijective on objects, its image under $\FdagSCat$ is a dagger Bergner trivial cofibration;
    \item the functor $\frakC_{\dag}$ sends every free cofibration of dagger simplicial sets to a dagger Bergner cofibration.
  \end{enumerate}
\end{lemma}

\begin{proof}
  (1)
  By \cref{bg.lem.I-inj,bg.thm.main}, a trivial fibration $p$ in $\sCat^{\dag}_{\Bergner}$ is surjective on objects and locally a trivial Kan fibration.  
  Thus $\UdagSCat p$ is a trivial fibration in the ordinary Bergner model structure.  
  Under the adjunction $\FdagSCat\dashv \UdagSCat$, a lifting problem for $\FdagSCat(j)$ against $p$ is therefore a lifting problem for $j$ against $\UdagSCat p$. 
  If $j$ is a Bergner cofibration, the latter problem has a lift; hence $\FdagSCat(j)$ has the left lifting property against every dagger Bergner trivial fibration and is a dagger Bergner cofibration.

  (2)
  Let $j:\calA\to\calB$ be an ordinary Bergner trivial cofibration which is bijective on objects, with common object set $O$.  
  Then the underlying map of its free dagger completion is $\calA\amalg_O\calA^{\myop}\to\calB\amalg_O\calB^{\myop}$.
  It factors as 
  \begin{align}
    \calA\amalg_O\calA^{\myop} \to \calB\amalg_O\calA^{\myop} \to \calB\amalg_O\calB^{\myop}.
  \end{align}
  The first map is a pushout of $j$ and the second a pushout of $j^{\myop}$; 
  hence both are ordinary Bergner trivial cofibrations.  
  Thus $\FdagSCat(j)$ is a local weak equivalence.  
  It is the identity on objects, so identities witness coherent essential surjectivity, and it is a dagger Bergner cofibration by (1).  
  It is therefore a dagger Bergner trivial cofibration.

  (3) 
  For every monomorphism $i:K\to L$, ordinary rigidification gives a Bergner cofibration $\frakC(i)$, and \cref{dj.lem.lift} gives $\frakC_{\dag}\FdagSSet(i) \cong \FdagSCat\frakC(i)$.
  Hence $\frakC_{\dag}$ sends each member of $I^{\sSet}_{\dag}$ to a dagger Bergner cofibration.  
  Since it preserves colimits and retracts, it sends the saturation of $I^{\sSet}_{\dag}$, which is the class of free cofibrations by \cref{dj.lem.free-cof}, to dagger Bergner cofibrations.
\end{proof}

\subsection{The obstruction: self-conjugate simplices}

Self-conjugate simplices behave as fixed cells rather than as freely paired dagger cells.
The self-conjugate $1$-simplex rigidifies to a self-adjoint endomorphism, and an explicit lifting obstruction shows that its boundary inclusion is not sent to a dagger Bergner cofibration (\cref{dj.prop.counterexample}).

\begin{notation}\label{dj.not.Q}
  We fix our notational conventions:
  \begin{itemize}
    \item under the equivalence $\sSet^{\dag}\simeq\Fun(\prism_{\rev}^{\myop},\Set)$ in \cref{prop.sSetdag_is_equivalent_to_Fun}, the representable presheaf $\yo_{\rev}([n])$ is $\FdagSSet(\Delta^n)$, and postcomposition with the order reversal $r_n:[n]\to[n]$ induces a $C_2$-action on it;
    \item we write $Q^n:=\yo_{\rev}([n])/C_2$ the \emph{self-conjugate $n$-simplex}.
    By the Yoneda lemma, $\Hom_{\sSet^{\dag}}(Q^n,Z)=\{\sigma\in Z_n\mid\sigma^{\dag}=\sigma\}$;
    \item we let $\partial Q^n\subseteq Q^n$ denote the image of the composite
    \begin{align}
      \FdagSSet(\partial\Delta^n)
      \to
      \FdagSSet(\Delta^n)
      \cong
      \yo_{\rev}([n])
      \to
      Q^n;
    \end{align}
    \item we let $\calE_{\sa}$ denote the free dagger simplicial category on one self-adjoint endomorphism: 
    it has a single object $x$, $\calE_{\sa}(x,x)=\bbN$ (discrete), composition given by addition, and identity dagger.
  \end{itemize}
\end{notation}

\begin{lemma}\label{dj.lem.CQ1}
  The morphism $\frakC_{\dag}(\partial Q^1\to Q^1)$ is $\mathbf 1_{\dag}\to\calE_{\sa}$.
\end{lemma}

\begin{proof}
  Since $\partial Q^1$ is a single vertex because the reversal exchanges the two vertices of $\yo_{\rev}[1]$, $\frakC_{\dag}(\partial Q^1)\cong \mathbf 1_{\dag}$.
  For every $\calC\in\sCat^{\dag}$, the adjunction and the defining universal property of $Q^1$ give
  \begin{align}
    \Hom_{\sCat^{\dag}}(\frakC_{\dag}Q^1,\calC)
    =\Hom_{\sSet^{\dag}}(Q^1,N_{\dag}\calC)
    =\{\sigma\in(N_{\dag}\calC)_1\mid \sigma^{\dag}=\sigma\}.
  \end{align}
  A $1$-simplex of $N_{\dag}\calC$ is a triple $(x_0,x_1,c)$ with $c\in\calC(x_0,x_1)_0$, and its conjugate is $(x_1,x_0,c^{\dag})$; self-conjugacy means $x_0=x_1=:x$ and $c^{\dag}=c$.
  The same data classifies dagger functors out of $\calE_{\sa}$ (a dagger functor is determined by the image of the additive generator, which must be a self-adjoint vertex).
  Hence $\frakC_{\dag}Q^1\cong\calE_{\sa}$, compatibly with the vertex inclusions.
\end{proof}

\begin{proposition}\label{dj.prop.counterexample}
  The functor $\frakC_{\dag}$ does not send all monomorphisms of $\sSet^{\dag}$ to cofibrations of $\sCat^{\dag}_{\Bergner}$: the map $\mathbf 1_{\dag}\to\calE_{\sa}$ of \cref{dj.lem.CQ1} is not a cofibration.
\end{proposition}

\begin{proof}
  Let $E:=N(\bbE[1])$ be the nerve of the indiscrete groupoid on $\{a,b\}$, a weakly contractible Kan complex, and let $s:E\to E$ be the involution induced by exchanging $a$ and $b$; note that $s$ has no fixed vertex.
  Let $G$ be the dagger simplicial graph with one object $x$, $G(x,x):=E$, and involution $\tau_{x,x}:=s$, and let $\calC':=\FCatDag(G)$ be the free dagger simplicial category on it (\cref{bg.lem.presentable}).
  Its unique mapping space is the free simplicial monoid on $E$, $\calC'(x,x)=\coprod_{n\geq0}E^{\times n}$, with dagger given by reversing words and applying $s$ letterwise (an anti-multiplicative involution fixing the empty word).

  Define $p:\calC'\to\calE_{\sa}$ to be the dagger functor induced by the graph map $E\to\Delta^0$ over the generator; on the mapping space it is the length map $\coprod_nE^{\times n}\to\bbN$.
  Its fiber over $n$ is exactly $E^{\times n}$, a contractible Kan complex; since a boundary sphere $\partial\Delta^m\to\coprod_nE^{\times n}$ lands in a single summand whenever its image in $\bbN$ is constant, $p$ is locally a trivial Kan fibration.
  More explicitly, $E^{\times n}\to\Delta^0$ is a trivial Kan fibration for every $n$, and every lifting problem over the discrete simplicial set $\bbN$ lies in one such fiber.
  It is bijective on objects.
  Hence $p$ is a trivial fibration of $\sCat^{\dag}_{\Bergner}$ (by \cref{bg.lem.I-inj}).

  Consider the lifting problem $\mathbf 1_{\dag}\xrightarrow{x}\calC',\calE_{\sa}\xrightarrow{\id}\calE_{\sa}$ over $p$.
  A lift is a dagger functor $s':\calE_{\sa}\to\calC'$ with $ps'=\id$, i.e.\ a self-adjoint vertex $c\in\calC'(x,x)_0$ with $p(c)=1$.
  The vertices over $1$ are the vertices of $E$, namely $a$ and $b$, and the dagger acts on them by $s$: $a^{\dag}=b\neq a$.
  So no lift exists, and $\mathbf 1_{\dag}\to\calE_{\sa}$ fails the left lifting property against a trivial fibration.
\end{proof}

\begin{corollary}
  There is no model structure on $\sSet^{\dag}$ with cofibrations all monomorphisms for which $\frakC_{\dag}\dashv N_{\dag}$ is a Quillen adjunction into $\sCat^{\dag}_{\Bergner}$.
\end{corollary}

\begin{proof}
  In any model structure with cofibrations the monomorphisms, $\partial Q^1\to Q^1$ is a cofibration; 
  a Quillen left adjoint must send it to a cofibration in the codomain model structure, contradicting \cref{dj.lem.CQ1,dj.prop.counterexample}.
\end{proof}

\subsection{Free cofibrations}

Motivated by the preceding counterexample, we restrict to monomorphisms whose new nondegenerate simplices of positive dimension occur in free dagger orbits.
A cellular description identifies these maps with the saturation of the free dagger boundary inclusions, providing the cofibrations for the model structure (\cref{dj.lem.free-cof}).

\begin{definition}\label{dj.def.free-cof}
  A monomorphism $i:A\to B$ in $\sSet^{\dag}$ is a \emph{free cofibration} if no nondegenerate simplex $\sigma$ of $B$ of positive dimension with $\sigma\notin A$ satisfies $\sigma^{\dag}=\sigma$.
  An object $X$ is \emph{free} if $\varnothing\to X$ is a free cofibration.
\end{definition}

\begin{notation}
  We write
  \begin{align}
    I^{\sSet}_{\dag}:=\{\FdagSSet(\partial\Delta^n)\to \FdagSSet(\Delta^n)\}_{n\geq0}.
  \end{align}
\end{notation}

\begin{proposition}\label{dj.lem.free-cof}
  The class of free cofibrations coincides with the saturation of $I^{\sSet}_{\dag}$ (closure of pushouts of coproducts under transfinite composition and retracts).
  Moreover, every free cofibration is a relative $I^{\sSet}_{\dag}$-cell complex.
\end{proposition}

\begin{proof}
  Let $i:A\to B$ be a free cofibration.
  We identify $A$ with its image in $B$.
  Since the dagger of $B$ is a degreewise bijection commuting with the simplicial operators up to reversal, it preserves and reflects nondegeneracy and dimension;
  it fixes every vertex, and by hypothesis it permutes the nondegenerate simplices of positive dimension of $B\setminus A$ in free orbits $\{\sigma,\sigma^{\dag}\}$ of size exactly two.

  Put $B_{\leq-1}:=A$, and for $n\geq0$ let $B_{\leq n}$ be the dagger simplicial subset of $B$ generated by $A$ and the nondegenerate simplices of $B\setminus A$ of dimension at most $n$.
  Let $\Sigma_0:=B_0\setminus A_0$.
  For $n\geq1$, let $\Sigma_n$ contain one representative from each dagger orbit of the nondegenerate simplices in $B_n\setminus A_n$.
  Every proper face of $\sigma\in\Sigma_n$ lies in $B_{\leq n-1}$:
  if that face is degenerate, its unique nondegenerate ancestor has dimension less than $n$.
  The boundary maps and the free--forgetful adjunction give the following commutative square:
  \begin{align}
    \begin{matrix}
      \displaystyle\coprod_{\sigma\in\Sigma_n}
      \FdagSSet(\partial\Delta^n)
      &\to&
      B_{\leq n-1}\\
      \downarrow&&\downarrow\\
      \displaystyle\coprod_{\sigma\in\Sigma_n}
      \FdagSSet(\Delta^n)
      &\to&
      B_{\leq n}.
    \end{matrix}
  \end{align}
  For $n=0$, the same square attaches precisely the vertices in $\Sigma_0$, since $\FdagSSet(\partial\Delta^0)=\varnothing$ and $\FdagSSet(\Delta^0)=\Delta^0$.
  For $n\geq1$, each summand attaches $\sigma$, $\sigma^{\dag}$, and their degeneracies.
  The Eilenberg--Zilber uniqueness lemma shows that no other nondegenerate simplex is added and that distinct chosen orbits produce distinct new simplices.
  Hence the canonical map from the pushout of the displayed span to $B_{\leq n}$ is an isomorphism, so the square is a pushout.
  Finally, $\colim_{n\geq-1}B_{\leq n}=B$.
  Thus $i$ is a relative $I^{\sSet}_{\dag}$-cell complex.

  Conversely, the class of monomorphisms satisfying \cref{dj.def.free-cof} contains $I^{\sSet}_{\dag}$.
  Indeed, for $n\geq1$ the two new nondegenerate $n$-simplices of a map in $I^{\sSet}_{\dag}$ are exchanged by the dagger, while in dimension zero no freeness condition is imposed.
  Coproducts of free cofibrations are free.

  For a pushout of a free cofibration $i:A\to B$ along $g:A\to C$, put $P:=B\amalg_A C$.
  Degreewise in sets, $C\to P$ is a monomorphism and the map $B_k\setminus A_k\to P_k\setminus C_k$ is a dagger-equivariant bijection.
  For $k\geq1$, a simplex represented by $b\in B_k\setminus A_k$ is degenerate in $P$ if and only if $b$ is degenerate in $B$.
  Indeed, if $b=s_jb'$ in $B$, then $b'\notin A$, so the same relation holds in $P$.
  Conversely, suppose that the image of $b$ is $s_jp$ in $P$.
  Under the degreewise decomposition
  \begin{align}
    P_{k-1}
    \cong
    C_{k-1}\amalg(B_{k-1}\setminus A_{k-1}),
  \end{align}
  the element $p$ cannot lie in the first summand, since then $s_jp$ would lie in $C_k$.
  Thus $p$ is represented by some $b'\in B_{k-1}\setminus A_{k-1}$, and the equality in the complement implies $b=s_jb'$ in $B$.
  Hence pushouts of free cofibrations are free.

  Consider a continuous transfinite composite of free cofibrations $X_0\to X_{\lambda}$.
  The structure maps are monomorphisms.
  If a positive-dimensional nondegenerate simplex $\sigma\in X_{\lambda}\setminus X_0$ were fixed by the dagger, it would first occur at a successor stage $X_{\alpha}\to X_{\alpha+1}$.
  Indeed, at a limit stage the simplicial set is the union of its preceding stages.
  It is nondegenerate at that stage, since a degeneracy remains a degeneracy under every subsequent map.
  Moreover, the two simplices $\sigma$ and $\sigma^{\dag}$ could not become equal at a later stage because all structure maps are monomorphisms.
  This contradicts the freeness of $X_{\alpha}\to X_{\alpha+1}$.
  Thus transfinite composites are free.

  Finally, suppose that $i':A'\to B'$ is a retract of $i:A\to B$ in the arrow category, with section $s$ and retraction $r$.
  The map $i'$ is a monomorphism because monomorphisms are closed under retracts in the arrow category.
  A self-conjugate nondegenerate simplex $\sigma\in B'\setminus A'$ of positive dimension would give the self-conjugate simplex $s_B(\sigma)$.
  It is nondegenerate, since the split monomorphism $s_B$ preserves nondegeneracy.
  It lies outside $A$:
  if $s_B(\sigma)=i(a)$, then $\sigma=r_Bi(a)=i'r_A(a)$, contrary to $\sigma\notin A'$.
  This contradicts freeness of $i$.
  Therefore retracts are free.

  The class of free cofibrations contains $I^{\sSet}_{\dag}$ and is closed under coproducts, pushouts, transfinite compositions, and retracts.
  It contains the saturation of $I^{\sSet}_{\dag}$.
  The relative cell description proved above gives the reverse inclusion.
\end{proof}

\begin{corollary}\label{dj.cor.Fdag-free}
  For every monomorphism $i:K\to L$ of simplicial sets, $\FdagSSet(i):\FdagSSet(K)\to \FdagSSet(L)$ is a free cofibration;
  in particular, $\FdagSSet(A)$ is free for every $A\in\sSet$.
\end{corollary}

\begin{proof}
  In each simplicial degree, $\FdagSSet(L)=L\amalg_{L^{(0)}}L^{\myop}$ is the pushout of two copies of $L$ along the totally degenerate part; a nondegenerate simplex of positive dimension therefore lies in exactly one of the two summands, and the involution exchanges the summands while fixing the glued part.
  Hence for $\sigma\in \FdagSSet(L)\setminus \FdagSSet(K)$ nondegenerate of positive dimension, the simplex $\sigma^{\dag}$ lies in the other summand, so $\sigma^{\dag}\neq\sigma$.
  The new vertices of $\FdagSSet(L)\setminus \FdagSSet(K)$ are self-conjugate, as all vertices are, but \cref{dj.def.free-cof} imposes no condition in dimension zero.
\end{proof}

\begin{remark}\label{dj.rem.cofibrant-objects}
  By contrast, dagger nerves are usually not free; 
  their cofibrant replacements in the model structure below are ``free resolutions'', in which every self-adjoint cell of positive dimension is replaced by a freely permuted conjugate pair, the vertices being retained.
  This is forced by \cref{dj.prop.counterexample} and is the simplicial shadow of the passage from strict to homotopy fixed points.
\end{remark}

\subsection{The dagger Joyal model structure and the Quillen equivalence}

We define dagger Joyal equivalences through rigidification and verify the accessibility and closure conditions required by Smith's theorem.
This constructs the left proper combinatorial dagger Joyal model structure and proves the rigidification--nerve Quillen equivalence (\cref{dj.thm.main,dj.thm.equivalence}).

\begin{definition}\label{dj.def.W}
  A morphism $f$ of $\sSet^{\dag}$ is a \emph{dagger Joyal equivalence} if $\frakC_{\dag}(f)$ is a dagger DK-equivalence.
  We write $W^{\sSet}_{\dag}:=\frakC_{\dag}^{-1}(W_{\dag})$.
\end{definition}

\begin{lemma}\label{dj.lem.W-properties}
  The class $W^{\sSet}_{\dag}$ contains all isomorphisms, satisfies two-out-of-three and is closed under retracts.
\end{lemma}

\begin{proof}
  This is immediate from $W^{\sSet}_{\dag}=\frakC_{\dag}^{-1}(W_{\dag})$ and \cref{bg.prop.W-23}.
\end{proof}

\begin{lemma}\label{dj.lem.ES}
  A dagger simplicial functor $f:\calC\to\calD$ lies in $W_{\dag}$ if and only if it is a local weak equivalence and satisfies:
  \begin{itemize}
    \item[$(\ES')$] for every $d\in\Ob\calD$, there exist $c\in\Ob\calC$, a natural dagger interval $\bbH\in\calH_{\dag}$ and a dagger functor $\psi:\bbH\to\calD$ with $\psi(0)=fc$ and $\psi(1)=d$.
  \end{itemize}
\end{lemma}

\begin{proof}
  Let $f\in W_{\dag}$.
  Take $d\in\calD$, and choose a coherently unitary $[v]:fc\to d$ and a representing vertex; \cref{bg.lem.realization} produces the required $\psi$ (replacing the interval by its representative in $\calH_{\dag}$).
  
  Conversely, assume that $f$ is a local weak equivalence and satisfies (ES').
  Given $\psi$ as in $(\ES')$, the induced map $\bbH\to\calD\langle fc,d\rangle\xrightarrow{r}R\calD\langle fc,d\rangle$ witnesses that the class of $\psi(h_{\bbH})$ is a coherently unitary equivalence $fc\to d$.
\end{proof}

\begin{proposition}\label{dj.lem.accessible}
  The full subcategories $W_{\dag}\subseteq\Fun(\Delta^1,\sCat^{\dag})$ and $W^{\sSet}_{\dag}\subseteq\Fun(\Delta^1,\sSet^{\dag})$ are accessible and accessibly embedded.
\end{proposition}

\begin{proof}
  The dagger Bergner model category is combinatorial (by \cref{bg.thm.main}), and its underlying category is presentable (by \cref{bg.lem.presentable}).  
  The general accessibility theorem for combinatorial model categories therefore implies that its weak-equivalence subcategory $W_{\dag}\subseteq\Fun(\Delta^1,\sCat^{\dag})$ is accessible and accessibly embedded (by \cite[Corollary A.2.6.6]{HTT}).

  The functor $\frakC_{\dag}$ is a left adjoint between presentable categories, hence accessible.  Its induced arrow-category functor is accessible, and $ W^{\sSet}_{\dag}=\Fun(\Delta^1,\frakC_{\dag})^{-1}(W_{\dag})$.
  By the inverse-image theorem for accessible, accessibly embedded full subcategories (by \cite[Corollary A.2.6.5]{HTT}), $W^{\sSet}_{\dag}$ is accessible and accessibly embedded as well.
\end{proof}

\begin{theorem}[The dagger Joyal model structure]\label{dj.thm.main}
  There exists a left proper combinatorial model structure $\sSet^{\dag}_{\Joyal}$ on $\sSet^{\dag}$ in which:
  \begin{enumerate}
    \item the cofibrations are the free cofibrations of \cref{dj.def.free-cof}, 
    i.e.\ the saturation of $I^{\sSet}_{\dag}$;
    \item the weak equivalences are the dagger Joyal equivalences $W^{\sSet}_{\dag}$;
    \item the trivial fibrations are exactly the morphisms $p$ such that $\UdagSSet(p)$ is a trivial Kan fibration.
  \end{enumerate}
\end{theorem}

\begin{proof}
  We apply Smith's theorem for combinatorial model categories \cite[Proposition A.2.6.8]{HTT} to the presentable category $\sSet^{\dag}$, the set $I^{\sSet}_{\dag}$, and the class $W^{\sSet}_{\dag}$.
  Its hypotheses are: (a) $W^{\sSet}_{\dag}$ satisfies two-out-of-three and is accessible and accessibly embedded; (b) $I^{\sSet}_{\dag}\text{-}\inj\subseteq W^{\sSet}_{\dag}$; (c) $\cof(I^{\sSet}_{\dag})\cap W^{\sSet}_{\dag}$ is closed under pushout, transfinite composition and retracts.

  (a) is \cref{dj.lem.W-properties,dj.lem.accessible}.

  (b) by the adjunction $\FdagSSet\dashv \UdagSSet$, a morphism $p$ lies in $I^{\sSet}_{\dag}\text{-}\inj$ if and only if $\UdagSSet(p)$ has the right lifting property against all $\partial\Delta^n\to\Delta^n$, i.e.\ is a trivial Kan fibration (this also proves (3) once the model structure exists).
  Then $\UdagSSet(p)$ is a weak equivalence of the Joyal model structure, so $\frakC(\UdagSSet p)$ is a DK-equivalence, ordinary rigidification being left Quillen for the Joyal model structure and all simplicial sets being cofibrant; thus $\frakC_{\dag}(p)$ is a local weak equivalence.
  Moreover $\UdagSSet(p)$ is surjective on vertices (lift against $\varnothing\to\Delta^0$), so $\frakC_{\dag}(p)$ is surjective on objects and coherent essential surjectivity is witnessed by identities (by \cref{bg.prop.cu-groupoid} (1)).
  Hence $\frakC_{\dag}(p)\in W_{\dag}$, i.e.\ $p\in W^{\sSet}_{\dag}$.

  (c) Let $j\in\cof(I^{\sSet}_{\dag})\cap W^{\sSet}_{\dag}$.
  By \cref{dj.lem.F-preserves-cof}, $\frakC_{\dag}(j)$ is a dagger Bergner cofibration; by the definition of $W^{\sSet}_{\dag}$ it is also a dagger DK-equivalence, hence a trivial cofibration.
  If $j'$ is a pushout of $j$, then $j'$ remains a free cofibration by \cref{dj.lem.free-cof}, while $\frakC_{\dag}(j')$ is the corresponding pushout of the trivial cofibration $\frakC_{\dag}(j)$.  
  Therefore $\frakC_{\dag}(j')$ is a trivial cofibration and $j'\in W^{\sSet}_{\dag}$.
  Likewise, $\frakC_{\dag}$ carries a transfinite composite of maps in $\cof(I^{\sSet}_{\dag})\cap W^{\sSet}_{\dag}$ to a transfinite composite of dagger Bergner trivial cofibrations.  
  Trivial cofibrations are closed under such composites, and the source composite is free by \cref{dj.lem.free-cof}; 
  hence the composite again belongs to $\cof(I^{\sSet}_{\dag})\cap W^{\sSet}_{\dag}$.
  Closure under retracts follows from the retract closure of free cofibrations in \cref{dj.lem.free-cof} and of $W^{\sSet}_{\dag}$.

  Smith's theorem now yields the combinatorial model structure with cofibrations $\cof(I^{\sSet}_{\dag})$ (identified in \cref{dj.lem.free-cof}) and weak equivalences $W^{\sSet}_{\dag}$.

  \emph{Left properness:}
  let $f\in W^{\sSet}_{\dag}$ and let $i$ be a free cofibration.  
  If $f'$ is their pushout, then $\frakC_{\dag}(f')$ is the pushout of $\frakC_{\dag}(f)\in W_{\dag}$ along the dagger Bergner cofibration $\frakC_{\dag}(i)$, the latter being a cofibration by \cref{dj.lem.F-preserves-cof}.  
  Left properness of $\sCat^{\dag}_{\Bergner}$ (by \cref{bg.thm.left-proper}) gives $\frakC_{\dag}(f')\in W_{\dag}$, and therefore $f'\in W^{\sSet}_{\dag}$.
\end{proof}

\begin{theorem}[Dagger Joyal--Bergner Quillen equivalence]\label{dj.thm.equivalence}
  The adjunction
  \begin{align}
    \frakC_{\dag}:
    \sSet^{\dag}_{\Joyal}
    \rightleftarrows
    \sCat^{\dag}_{\Bergner}
    :N_{\dag}
  \end{align}
  is a Quillen equivalence.
\end{theorem}

\begin{proof}
  \emph{Quillen adjunction:}
  By \cref{dj.lem.F-preserves-cof}, $\frakC_{\dag}$ sends every free cofibration, hence every cofibration of $\sSet^{\dag}_{\Joyal}$, to a dagger Bergner cofibration.  
  If $j$ is also a dagger Joyal equivalence, then $\frakC_{\dag}(j)\in W_{\dag}$ by \cref{dj.def.W}; 
  consequently $\frakC_{\dag}(j)$ is a dagger Bergner trivial cofibration. 
  Thus $\frakC_{\dag}\dashv N_{\dag}$ is a Quillen adjunction.

  Let $\calC$ be dagger Bergner fibrant.  
  By \cref{bg.cor.fibrant}, all mapping spaces of $\calC$ are Kan complexes, so $\UdagSCat\calC$ is Bergner fibrant.  
  The counit $\varepsilon_{\calC}:\frakC_{\dag}N_{\dag}\calC\to\calC$ is the identity on objects.  
  Under $\UdagSCat$ it is the ordinary counit $\frakC N(\UdagSCat\calC)\to \UdagSCat\calC$ by the compatibility of the lifted adjunction in \cref{dj.lem.lift}.  
  The ordinary Joyal--Bergner Quillen equivalence makes this a DK-equivalence, since $\UdagSCat\calC$ is fibrant.  Hence $\varepsilon_{\calC}$ is a local weak equivalence; identities provide coherent essential surjectivity by \cref{bg.prop.cu-groupoid} (1).  
  Therefore $\varepsilon_{\calC}\in W_{\dag}$.

  \emph{Quillen-equivalence:}
  Let $X$ be cofibrant, let $\calC$ be fibrant, and let $f:\frakC_{\dag}X\to\calC$ have adjoint $f^{\sharp}:X\to N_{\dag}\calC$.  
  The triangle identity gives $f=\varepsilon_{\calC}\circ\frakC_{\dag}(f^{\sharp})$.
  By definition, $f^{\sharp}$ is a weak equivalence in $\sSet^{\dag}_{\Joyal}$ exactly when $\frakC_{\dag}(f^{\sharp})\in W_{\dag}$.  
  Since $\varepsilon_{\calC}\in W_{\dag}$, two-out-of-three shows that this is equivalent to $f\in W_{\dag}$.  

  We record the derived maps explicitly.  
  If $q:Q N_{\dag}\calC\to N_{\dag}\calC$ is a cofibrant replacement and $\calC$ is fibrant, the derived counit is represented by $ \frakC_{\dag}Q N_{\dag}\calC \xrightarrow{\frakC_{\dag}(q)} \frakC_{\dag}N_{\dag}\calC \xrightarrow{\varepsilon_{\calC}}\calC$.
  The first arrow belongs to $W_{\dag}$ by the definition of $W^{\sSet}_{\dag}$ and the second by the preceding paragraph, so the derived counit is invertible.  
  If $X$ is cofibrant and $r:\frakC_{\dag}X\to R\frakC_{\dag}X$ is a fibrant replacement, the derived unit is the adjoint $X\to N_{\dag}R\frakC_{\dag}X$ of $r$.  
  Applying the adjoint-map criterion to the weak equivalence $r$ shows that this derived unit is a weak equivalence.  
  Cofibrant replacement gives the corresponding statements for arbitrary objects.
\end{proof}

\begin{corollary}\label{dj.cor.fibrant-underlying}
  Let $X$ be fibrant in $\sSet^{\dag}_{\Joyal}$.
  Then $\UdagSSet(X)$ is a quasi-category.
  Moreover, for every fibrant $\calC\in\sCat^{\dag}_{\Bergner}$, the dagger simplicial set $N_{\dag}(\calC)$ is fibrant in $\sSet^{\dag}_{\Joyal}$ and its underlying simplicial set is a quasi-category.
\end{corollary}

\begin{proof}
  Fix $0<k<n$ and write $i:\Lambda^n_k\to\Delta^n$.  
  This is a Joyal trivial cofibration. 
  Since ordinary rigidification is left Quillen, $\frakC(i)$ is a Bergner trivial cofibration; 
  it is bijective on objects because both rigidifications have object set $\{0,\ldots,n\}$.  
  The natural isomorphism of \cref{dj.lem.lift} identifies $\frakC_{\dag}\FdagSSet(i) \cong \FdagSCat\frakC(i)$.

  The second assertion of \cref{dj.lem.F-preserves-cof} therefore makes this a dagger Bergner trivial cofibration.  
  On the other hand, $\FdagSSet(i)$ is a free cofibration by \cref{dj.cor.Fdag-free}.  
  Its image under $\frakC_{\dag}$ lies in $W_{\dag}$, so $\FdagSSet(i)\in W^{\sSet}_{\dag}$ by definition.  
  Thus $\FdagSSet(i)$ is a trivial cofibration in $\sSet^{\dag}_{\Joyal}$.

  A fibrant $X$ has the right lifting property against every $\FdagSSet(i)$.  
  By $\FdagSSet\dashv \UdagSSet$, these lifting problems are precisely the ordinary inner-horn lifting problems for $\UdagSSet(X)$.  
  Hence $\UdagSSet(X)$ is a quasi-category.

  If $\calC$ is dagger Bergner fibrant, then $N_{\dag}\calC$ is dagger Joyal fibrant because $N_{\dag}$ is right Quillen by \cref{dj.thm.equivalence}; 
  applying the first part shows that its underlying simplicial set is a quasi-category.  
  Equivalently, $\UdagSSet N_{\dag}\calC=N(\UdagSCat\calC)$ is the usual homotopy coherent nerve of the locally Kan simplicial category $\UdagSCat\calC$.
\end{proof}

\begin{remark}\label{dj.rem.fibrant-open}
  We do \emph{not} assert that every dagger simplicial set whose underlying simplicial set is a quasi-category is fibrant in $\sSet^{\dag}_{\Joyal}$.
  In the present setting, the question is precisely whether self-conjugate lifting problems against trivial cofibrations can always be solved, and \cref{dj.prop.counterexample} shows that such equivariance questions are delicate.
  The characterization of the fibrant objects is left open.
\end{remark}

\section{The unitary core of dagger quasi-categories}\label{dj.core-section}

In this section, we identify coherent unitarity with paths in the unitary core.
The key step is a Real edgewise-subdivision comparison between the covariant involution on a mapping space and the reversal action on geometric realization.
This gives an intrinsic characterization for dagger Joyal equivalences (\cref{dj.thm.intrinsic-equivalence}).

\subsection{The unitary core}

The dagger induces a $C_2$-action on geometric realization, so homotopy fixed points provide a natural receptacle for coherent unitary data.
Applying this construction to the maximal Kan subcomplex and retaining the canonical points associated to objects defines the unitary core functorially in dagger quasi-categories (\cref{prop.unitary_core_functor}).

\begin{construction}
  Let $K$ be a simplicial set.
  The \emph{geometric realization} $|K|$ of $K$ is 
  \begin{align}
    |K| := (\coprod_{n\geq0} K_n \times \Delta^n_{\rmtop}) / \sim,
  \end{align}
  where the equivalence relation is generated by $(K(\alpha)(\sigma),t) \sim (\sigma,|\alpha|(t))$ for every morphism $\alpha : [m] \to [n]$ in $\prism$, every $n$-simplex $\sigma \in K$, and every point $t \in \Delta^m_{\rmtop}$.
  We denote the equivalence class of $(\sigma,t)$ by $[\sigma;t]$.
\end{construction}

\begin{lemma} \label{lem.coordinate_reversal_naturality}
  For every $n \geq 0$, we define a continuous map $R_{n} : \Delta^{n}_{\rmtop} \to \Delta^{n}_{\rmtop}$ by reversing barycentric coordinates: 
  \begin{align}
    R_{n} : 
    \Delta^{n}_{\rmtop} \to \Delta^{n}_{\rmtop}
    : (t_0,\cdots,t_n) \mapsto (t_n,\cdots,t_0).
  \end{align}
  Then, for every morphism $\alpha : [m] \to [n]$ in $\prism$, $R_{n}|\alpha| = |\alpha^{\rev}|R_{m}$.
\end{lemma}

\begin{proof}
  For $t=(t_0,\ldots,t_m)$ and $0\leq j\leq n$, the $j$-th coordinates of the two composites are
  \begin{align}
    (R_n|\alpha|(t))_j
    =
    \sum_{\alpha(k)=n-j}t_k
    =
    \sum_{\alpha(m-i)=n-j}t_{m-i}
    =
    (|\alpha^{\rev}|R_m(t))_j.
  \end{align}
  Hence $R_n|\alpha|=|\alpha^{\rev}|R_m$.
\end{proof}

\begin{lemma} \label{prop.realization_of_opposite}
  Let $K$ be a simplicial set.
  We define a map $\rho_{K} : |K^{\myop}| \to |K|$ by 
  \begin{align}
    \rho_{K} : |K^{\myop}| \to |K| : [\sigma;(t_0,\cdots,t_n)] \mapsto [\sigma;(t_n,\cdots,t_0)].
  \end{align}
  Then it is a homeomorphism, and for every morphism of simplicial sets $f : K \to L$, it satisfies $|f|\rho_{K} = \rho_{L}|f^{\myop}|$.
\end{lemma}

\begin{proof}
  This is the coordinate-reversal comparison used in the realization of Real simplicial sets (see \cite[\S 3.1]{Dotto16}).
  For a generating relation in $|K^{\myop}|$, the two representatives are $[K(\alpha^{\rev})(\sigma);t]$ and $[\sigma;|\alpha|(t)]$.
  Their images agree in $|K|$ because $R_n|\alpha|=|\alpha^{\rev}|R_m$ (by \cref{lem.coordinate_reversal_naturality}).
  Thus the formula descends continuously through the quotient.
  Since $R_n^2=\id$, the corresponding map for $K^{\myop}$ is its inverse.
  Finally, we prove naturality:
  \begin{align}
    |f|\rho_K[\sigma;t] = [f(\sigma);R_n(t)] = \rho_L|f^{\myop}|[\sigma;t].
  \end{align}
\end{proof}

\begin{lemma} \label{prop.dagger_realization_action}
  Let $K$ be a dagger simplicial set.
  Then a continuous map $\tau_{K}:= \rho_K \circ |\dag_K| : |K| \to |K|$ satisfies:
  \begin{enumerate}
    \item For every $n$-simplex $\sigma\in K$, $\tau_K([\sigma;(t_0,\cdots,t_n)])=[\sigma^\dagger;(t_n,\cdots,t_0)]$.
    \item The map $\tau_K$ is involutive: $\tau_K^2=\id_{|K|}$ and $\tau_{K}|_{0}=\id_{K_{0}}$.
    Therefore $|K|$ is a $C_{2}$-topological space.
    \item For every dagger morphism $f : K \to L$, the continuous map $|f| : |K| \to |L|$ is $C_2$-equivariant.
  \end{enumerate}
\end{lemma}

\begin{proof}
  (1) 
  Since $ \dag_K:K \to K^{\myop}$ is degreewise given by $\sigma\mapsto \sigma^{\dag}$, we have $|\dag_K|([\sigma;t]) = [\sigma^{\dag};t]$ as a point of $|K^{\myop}|$.  
  Applying $\rho_K$, we obtain
  \begin{align*}
    \tau_K([\sigma;t])
    =
    \rho_K|\dag_K|([\sigma;t]) 
    =
    \rho_K([\sigma^{\dag};t]) 
    =
    [\sigma^{\dag};R_n(t)].
  \end{align*}

  (2)
  By $(1)$, we have
  \begin{align*}
    \tau_K^2([\sigma;t])
    =
    \tau_K([\sigma^{\dag};R_n(t)]) 
    =
    [(\sigma^{\dag})^{\dag};R_n^2(t)] 
    =
    [\sigma;t].
  \end{align*}
  Now let $x\in K_0$ be a vertex.  
  Since $\Delta^0_{\rmtop}$ has one
  point, write this point as $1$.  Then 
  \begin{align}
    \tau_K([x;1]) = [x^{\dag};R_0(1)] = [x;1].
  \end{align}

  (3)
  By definition, this means $\dag_L\circ f = f^{\myop}\circ \dag_K$.
  Taking geometric realizations gives $|\dag_L|\circ |f| = |f^{\myop}|\circ |\dag_K|$.
  By the naturality of $\rho$, we also have $|f|\circ \rho_K = \rho_L\circ |f^{\myop}|$.
  Therefore
  \begin{align*}
    \tau_L\circ |f|
    =
    \rho_L\circ |\dag_L|\circ |f| 
    =
    \rho_L\circ |f^{\myop}|\circ |\dag_K| 
    =
    |f|\circ \rho_K\circ |\dag_K| 
    =
    |f|\circ \tau_K.
  \end{align*}
\end{proof}

\begin{construction}
  We fix some constructions:
  \begin{itemize}
    \item we define the groupoid $\widetilde{C_2}$ as follows:
    Its set of objects is $C_2=\{1,\omega\}$, and for every pair of objects $g,h\in C_2$ there is exactly one morphism $e_{g,h}:g \to h$.
    It is a $C_2$-category.  
    For $k\in C_2$, the action functor $L_k:\widetilde{C_2} \to \widetilde{C_2}$ is defined by $ L_k(g) := kg$ and $L_k(e_{g,h}) := e_{kg,kh}$;
    \item we define a simplicial set $E_{\bullet}C_2 := N(\widetilde{C_2})$.
    Since there is a unique morphism between any two objects of $\widetilde{C_2}$, we have a natural bijection $(E_{\bullet}C_2)_n = N(\widetilde{C_2})_n \simeq C_2^{n+1}$;
    \item under this identification, the $C_2$-action is given by $h\cdot(g_0,\cdots,g_n) := (hg_0,\cdots,hg_n)$.
    This action commutes with the face and degeneracy maps;
    \item we define the $C_2$-topological space $EC_{2} := |E_{\bullet}C_2|$.
  \end{itemize}
\end{construction}

\begin{lemma} \label{lem.EC2_free_contractible}
  The topological space $EC_2$ is a free contractible $C_2$-CW complex.
\end{lemma}

\begin{proof}
  Since the category $\widetilde{C_2}$ admits a terminal object $1$, the nerve $E_{\bullet}C_2$ is contractible.
  Therefore its geometric realization $EC_2$ is contractible as an underlying space.

  An $n$-simplex $(g_0,\ldots,g_n)$ is nondegenerate if and only if $g_i\neq g_{i+1}$ for every $0\leq i<n$.
  Since $C_2=\{1,\omega\}$, there are exactly two nondegenerate $n$-simplices $(1,\omega,1,\ldots)$ and $(\omega,1,\omega,\ldots)$.
  They are exchanged by the $C_2$-action.
  Hence the equivariant skeletal filtration of the realization is obtained, in each dimension $n$, from a pushout
  \begin{align}
    \begin{matrix}
      C_2\times|\partial\Delta^n|
      &\to&
      |\sk_{n-1}E_{\bullet}C_2|\\
      \downarrow&&\downarrow\\
      C_2\times|\Delta^n|
      &\to&
      |\sk_nE_{\bullet}C_2|.
    \end{matrix}
  \end{align}
  Here $\sk_{-1}E_{\bullet}C_2=\varnothing$, and $C_2$ acts by left multiplication on the first factor.
  Thus $EC_2$ has one free equivariant cell $C_2/\{1\}\times D^n$ in each dimension.
  It is a free $C_2$-CW complex.
\end{proof}

\begin{remark}
  We use the standard $q$-model structure on compactly generated weak Hausdorff spaces, in which every object is fibrant, and equip $C_2$-spaces with the Borel model structure created by the forgetful functor.
  Equivariant mapping spaces carry the compactly generated compact-open topology.
  The map $EC_2 \to *$ is an underlying weak equivalence because $EC_2$ is contractible and $EC_2$ is cofibrant as a $C_2$-topological space because \cref{lem.EC2_free_contractible} identifies it as a free $C_2$-CW complex.
  Thus $EC_2$ is a cofibrant replacement of the point in $C_2$-topological spaces.
  Consequently, for a $C_2$-topological space $X$, the space $\Map_{C_2}(EC_2,X)$ computes the derived fixed point space $\RMap_{C_2}(*,X)$.
\end{remark}

\begin{definition} \label{def.topological_homotopy_fixed_points}
  Let $X$ be a strict $C_2$-space.  
  We define the \emph{topological homotopy $C_{2}$-fixed point space} $X^{hC_2}_{\rmtop}$ of $X$ by 
  \begin{align}
    X^{hC_2}_{\rmtop} := \Map_{C_2}(EC_2,X).
  \end{align}
\end{definition}

\begin{definition}
  Let $K$ be a dagger simplicial set.  
  We define the \emph{homotopy $C_{2}$-fixed point Kan complex} $K^{hC_2}$ of $K$ by 
  \begin{align}
    K^{hC_2} := \Sing \Map_{C_2}(EC_2,|K|).
  \end{align}
\end{definition}

\begin{remark}
  Let $f : EC_{2} \to |K|$ be a $C_{2}$-equivariant map.
  We set a point $x := f(1)$ in $|K|$ for the vertex $1 \in EC_{2}$.
  Since $f$ is equivariant, $f(\omega) = \omega x$.

  An edge $(1,\omega) \in EC_{2}$ gives a path $\gamma : x \to \omega x$.
  Similarly, an edge $(\omega,1) \in EC_{2}$ gives a path $\omega\gamma : \omega x \to x$ in $|K|$.
  A $2$-simplex $(1,\omega,1)$ gives a homotopy $\omega\gamma \circ \gamma \simeq \id_{x}$.
  Higher simplices give coherence data of these homotopies.
  Thus a homotopy $C_{2}$-fixed point is a homotopy-coherent version of a fixed point $x = \omega x$.
\end{remark}

\begin{lemma}
  \label{lem.strict_to_homotopy_fixed_points}
  Let $X$ be a $C_2$-topological space.  
  There exists a natural continuous map 
  \begin{align}
    \kappa_{X} : 
    X^{C_2} \to X^{hC_2}_{\rmtop} 
    : x \mapsto c_{x},
  \end{align}
  where $c_{x} : EC_{2} \to X : e \mapsto x$ is the constant map.
\end{lemma}

\begin{proof}
  Let $p:EC_2\to *$ be the unique equivariant map.
  In compactly generated spaces, equivariant mapping spaces form a topological enrichment, so precomposition with $p$ is continuous:
  \begin{align}
    p^*:\Map_{C_2}(*,X)\to\Map_{C_2}(EC_2,X).
  \end{align}
  Under the natural homeomorphism $\Map_{C_2}(*,X)\cong X^{C_2}$, this map is precisely $x\mapsto c_x$.
  This proves continuity.
  If $f:X\to Y$ is an equivariant continuous map, then
  \begin{align}
    \Map_{C_2}(EC_2,f)(c_x)
    =
    f\circ c_x
    =
    c_{f(x)}
    =
    \kappa_Y(f^{C_2}(x)).
  \end{align}
  Hence the maps $\kappa_X$ are natural.
\end{proof}

\begin{remark}
  The map $\kappa_X$ need not be a weak equivalence.
  For example, if $X=EC_2$, then $X^{C_2}=\varnothing$, whereas $\Map_{C_2}(EC_2,EC_2)$ contains the identity map.
\end{remark}

\begin{definition} \label{def.canonical_homotopy_fixed_point}
  Let \(K\) be a dagger simplicial set and let \(x\in K_0\).
  By \cref{prop.dagger_realization_action}, the point \(x\in|K|\) is strictly fixed by \(\tau_K\).  
  Therefore the constant map $c_{x} : EC_2\to|K| : e \mapsto x$ is \(C_2\)-equivariant.
  We denote the corresponding vertex of $K^{hC_{2}}$ by $\widehat{x}$.
  We will refer to \(\widehat{x}\) as the \emph{canonical homotopy $C_{2}$-fixed point associated to \(x\)}.
\end{definition}

\begin{remark}
  Let $\calC$ be a dagger quasi-category.
  We write $\calC^{\simeq}$ for the maximal Kan subcomplex of $\calC$, also called its groupoid core.
  The dagger preserves equivalences, so it restricts to a dagger involution on $\calC^{\simeq}$.
\end{remark}

\begin{definition} \label{def.unitary_core}
  Let $\calC$ be a dagger quasi-category.
  We let $\calU(\calC) \subseteq (\calC^{\simeq})^{hC_{2}}$ denote the full sub Kan complex spanned by the canonical homotopy fixed points associated to objects of $\calC$.
  We will refer to $\calU(\calC)$ as the \emph{unitary core} of $\calC$.
\end{definition}

\begin{lemma} \label{prop.unitary_core_functor}
  The construction $\calC \mapsto \calU(\calC)$ defines a functor $\calU:\qCat^\dag\to\Kan$.
\end{lemma}

\begin{proof}
  Let $f : \calC \to \calD$ be a dagger functor. 
  Then it restricts to a dagger $f^{\simeq} : \calC^{\simeq} \to \calD^{\simeq}$.
  By \cref{prop.dagger_realization_action}, $|f^{\simeq}| : |\calC^{\simeq}| \to |\calD^{\simeq}|$ is \(C_2\)-equivariant.  
  Therefore postcomposition induces $(f^\simeq)^{hC_2} : (\calC^{\simeq})^{hC_2} \to (\calD^{\simeq})^{hC_2}$.
  For every object \(x\in\calC_0\), we have $|f^\simeq|\circ c_x = c_{f(x)}$.
  Hence $(f^\simeq)^{hC_2}(\widehat{x}) = \widehat{f(x)}$.
  Thus it restricts uniquely to $\calU(f) : \calU(\calC) \to \calU(\calD)$.
\end{proof}

\subsection{The characterization of dagger Joyal equivalences}

For a functor between quasi-categories, being a Joyal equivalence is equivalent both to being fully faithful and essentially surjective.
The dagger-specific issue is that ordinary essential surjectivity must be replaced by essential surjectivity through coherently unitary equivalences.
We compare this condition with surjectivity on the components of the unitary core (\cref{dj.thm.intrinsic-equivalence}).

\begin{remark}\label{dj.prop.ordinary-recognition}
  Let $f:X\to Y$ be a functor between quasi-categories.
  Then $f$ is a Joyal equivalence if and only if $f$ is fully faithful and essentially surjective.
  In this case, the induced map of groupoid cores $f^{\simeq}:X^{\simeq}\to Y^{\simeq}$ is a weak homotopy equivalence.
\end{remark}

\begin{notation}\label{dj.not.intrinsic-mapping}
  We fix some constructions:
  \begin{itemize}
    \item for Kan complex $K$ and vertices $a,b\in K_0$, we let $\Path_K(a,b)$ denote the homotopy fiber over $(a,b)$ of the endpoint map $K^{\Delta^1}\to K\times K$;
    \item for a compactly generated space $T$ and points $a,b\in T$, let
    \begin{align}
      P_{a,b}T
      :=
      \{\gamma:[0,1]\to T\mid\gamma(0)=a,\ \gamma(1)=b\}
    \end{align}
    with the compact-open topology followed by compact generation.
    Since $\{0,1\}\to[0,1]$ is a Hurewicz cofibration, endpoint evaluation $T^{[0,1]}\to T\times T$ is a Hurewicz fibration in compactly generated spaces.
    Thus $P_{a,b}T$ represents its derived fiber over $(a,b)$.
  \end{itemize}
\end{notation}

\begin{lemma}\label{dj.lem.full-subkan}
  Let $K$ be a Kan complex and let $A\subseteq K$ be the full simplicial subset spanned by a set of vertices.
  Then $A$ is a Kan complex.
  For every $a,b\in A_0$, the inclusion induces a weak equivalence $\Path_A(a,b)\to\Path_K(a,b)$.
\end{lemma}

\begin{proof}
  Let $\Lambda_i^n\to A$ be a horn.
  If $n\geq2$, every vertex of $\Delta^n$ occurs in $\Lambda_i^n$.
  A filler in $K$ therefore has all its vertices in $A_0$, and fullness implies that it factors through $A$.
  If $n=1$, the degenerate edge at the unique vertex of the horn is a filler in $A$.
  Thus $A$ is a Kan complex.

  The endpoint maps $A^{\Delta^1}\to A^{\partial\Delta^1}$ and $K^{\Delta^1}\to K^{\partial\Delta^1}$ are Kan fibrations.
  An $n$-simplex in the strict fiber of the second map over $(a,b)$ is a map $\Delta^n\times\Delta^1\to K$ which is constant at $a$ and $b$ on the two ends.
  Every vertex of its domain is therefore sent to either $a$ or $b$.
  Fullness shows that the map factors through $A$.
  Hence the two strict fibers are isomorphic, and they compute the corresponding homotopy fibers.
\end{proof}

\begin{construction}
  For a simplicial category $\calC$, write $\calC[\Mor^{-1}]$ for its degreewise groupoid completion:
  \begin{align}
    (\calC[\Mor^{-1}])_n
    :=
    \calC_n[\Mor(\calC_n)^{-1}].
  \end{align}
  Here $\calC_n$ denotes the ordinary category with the same objects as $\calC$ and with $\Hom_{\calC_n}(x,y)=\calC(x,y)_n$;
  its composition is induced by the simplicial composition of $\calC$.
  
  If $\calC$ is a dagger simplicial category, then $\calC[\Mor^{-1}]$ inherits a dagger because groupoid completion commutes with opposites.
\end{construction}

\begin{lemma}\label{dj.lem.unitary-derived-fixed}
  Let $\bbG$ be a strict dagger simplicial groupoid on the object set $\{0,1\}$.
  The formula $\sigma(g):=(g^{\dag})^{-1}$ defines a $C_2$-action on $\bbG(0,1)$.
  Then there exists a natural weak equivalence
  \begin{align}
    \RMap_{\sCat^{\dag}_{\{0,1\}}}(\bbI_{\dag},\bbG)
    \simeq
    \underline{\Map}_{C_2} (E_{\bullet}C_2,\bbG(0,1)_{\sigma}).
  \end{align}
\end{lemma}

\begin{proof}
  Inversion is simplicial in a strict simplicial groupoid.
  Since the dagger is involutive and commutes with inversion, $\sigma$ is a covariant simplicial involution.

  For a set $O$, the category $\sGpd_O$ of simplicial groupoids with object set $O$ admits a cofibrantly generated simplicial model structure in which weak equivalences and fibrations are detected on all mapping simplicial sets \cite[\S 2]{Ber08Inv}.
  Since this model category is cofibrantly generated and simplicial, the projective model structure $\Fun(BC_{2},\sGpd_O)_{\proj}$ exists and is simplicial by \cite[Theorems 11.6.1 and 11.7.3]{Hir02}.

  For a strict dagger simplicial groupoid $\bbA$, the formula
  $\sigma(g):=(g^{\dag})^{-1}$ defines a covariant simplicial involution fixing the objects.
  Conversely, such an involution defines a dagger by $g^{\dag}:=(\sigma g)^{-1}$.
  These constructions are mutually inverse and identify dagger functors with $C_2$-equivariant simplicial functors.
  Thus there exists an isomorphism of simplicial categories
  \begin{align}
    \sGpd^{\dag}_O
    \cong
    \Fun(BC_{2},\sGpd_O).
  \end{align}
  We transport the projective model structure across this isomorphism.
  Consequently, weak equivalences and fibrations in $\sGpd^{\dag}_O$ are detected on all underlying mapping simplicial sets.

  Degreewise groupoid completion and inclusion give an adjunction
  \begin{align}
    (-)[\Mor^{-1}]:
    \sCat^{\dag}_{O,\Bergner}
    \rightleftarrows
    \sGpd^{\dag}_{O,\Bergner}
    :\iota .
  \end{align}
  This adjunction is simplicially enriched.
  Indeed, cotensors of simplicial groupoids are computed as in simplicial categories and remain simplicial groupoids, so $(\iota\bbA)^{\Delta^n} = \iota(\bbA^{\Delta^n})$.
  Applying the ordinary completion--inclusion adjunction to these cotensors gives the required isomorphism of simplicial mapping objects.
  Finally, $\iota$ preserves fibrations and trivial fibrations because both model structures detect them on the same mapping simplicial sets.
  Hence this is a simplicial Quillen adjunction.

  From now on, take $O=\{0,1\}$.
  Choose a natural dagger interval $\bbH\to\bbI_{\dag}$.
  The object $\bbH$ is cofibrant in $\sCat^{\dag}_{\{0,1\},\Bergner}$, and $\UdagSCat\bbH$ is cofibrant in $\sCat_{\{0,1\},\Bergner}$ by \cref{bg.lem.reduction} and a relative-cell induction.
  Moreover, $\pi_0\bbH$ is a groupoid by \cref{bg.lem.interval-basic}.
  Therefore \cite[Proposition~9.5]{DK80} gives a local weak equivalence $\bbH\to\bbH[\Mor^{-1}]$.
  Two-out-of-three shows that $\bbH[\Mor^{-1}]\to\bbI_{\dag}$ is a local weak equivalence.
  Since groupoid completion is left Quillen, $\bbH[\Mor^{-1}]$ is cofibrant.
  The enriched adjunction consequently identifies the derived mapping space in the statement with the mapping space out of any cofibrant replacement of $\bbI_{\dag}$ in $\sGpd^{\dag}_{\{0,1\},\Bergner}$.

  We now give another replacement.
  For a $C_2$-simplicial set $A$, let $\bbF_{01}(A)$ be the free simplicial groupoid on generators $[a]:0\to1$, for $a\in A$, with dagger determined by $[a]^{\dag}:=[\omega a]^{-1}$.
  A dagger functor $\bbF_{01}(A)\to\bbG$ is uniquely determined by a $C_2$-equivariant map $A\to\bbG(0,1)_{\sigma}$.
  Hence there exists a natural enriched adjunction
  \begin{align}
    \underline{\Map}_{\sGpd^{\dag}_{\{0,1\},\Bergner}}(\bbF_{01}(A), \bbG)
    \cong
    \underline{\Map}_{C_2}(A,\bbG(0,1)_{\sigma}).
  \end{align}
  Let $\ev_{01}^{\sigma}$ denote the right adjoint given by $\ev_{01}^{\sigma}(\bbG):=\bbG(0,1)_{\sigma}$.
  In this model category, fibrations and trivial fibrations are detected on all four underlying mapping simplicial sets.
  Consequently, $\ev_{01}^{\sigma}$ preserves fibrations and trivial fibrations.
  Hence
  \begin{align}
    \bbF_{01}:
    \Fun(BC_{2},\sSet)_{\proj}
    \rightleftarrows
    \sGpd^{\dag}_{\{0,1\},\Bergner}
    :\ev_{01}^{\sigma}
  \end{align}
  is a simplicial Quillen adjunction.
  The action on $E_{\bullet}C_2$ is free in every simplicial degree, so $E_{\bullet}C_2$ is cofibrant in $\Fun(BC_{2},\sSet)_{\proj}$.
  Indeed, its nondegenerate simplices form free $C_2$-orbits, giving the usual projective cell decomposition.
  Thus $\bbF_{01}(E_{\bullet}C_2)$ is cofibrant in $\sGpd^{\dag}_{\{0,1\},\Bergner}$.
  There exist canonical identifications $\bbF_{01}(*)=\bbI_{\dag}$ and an augmentation $\bbF_{01}(E_{\bullet}C_2)\to\bbI_{\dag}$.
  After forgetting the $C_2$-action, this is the image of the weak equivalence $E_{\bullet}C_2\to*$ under the ordinary fixed-object free-groupoid functor.
  The ordinary functor is left Quillen from $\sSet_{\Kan}$ to $\sGpd_{\{0,1\},\Bergner}$ because its right adjoint, evaluation at $(0,1)$, preserves fibrations and trivial fibrations.
  Since every object of $\sSet_{\Kan}$ is cofibrant, it carries $E_{\bullet}C_2\to*$ to a local weak equivalence.
  Hence $\bbF_{01}(E_{\bullet}C_2)$ is another cofibrant replacement of $\bbI_{\dag}$.

  Every strict simplicial groupoid has Kan mapping simplicial sets.
  Thus $\bbG$ is fibrant in $\sGpd^{\dag}_{\{0,1\},\Bergner}$ and in $\sCat^{\dag}_{\{0,1\},\Bergner}$.
  The enriched mapping space out of $\bbF_{01}(E_{\bullet}C_2)$ therefore computes the derived mapping space from $\bbI_{\dag}$ to $\bbG$.
  Thus we have 
  \begin{align}
    \RMap_{\sCat^{\dag}_{\{0,1\},\Bergner}}(\bbI_{\dag},\bbG)
    &\simeq \underline{\Map}_{\sCat^{\dag}_{\{0,1\},\Bergner}}(\bbH,\iota\bbG)\\
    &\cong \underline{\Map}_{\sGpd^{\dag}_{\{0,1\},\Bergner}}(\bbH[\Mor^{-1}],\bbG)\\
    &\simeq \underline{\Map}_{\sGpd^{\dag}_{\{0,1\},\Bergner}}(\bbF_{01}(E_\bullet C_2),\bbG)\\
    &\cong \underline{\Map}_{C_2}(E_\bullet C_2,\bbG(0,1)_\sigma).
  \end{align}
  The middle weak equivalence is the comparison of the derived mapping spaces obtained from the two cofibrant replacements
  $\bbH[\Mor^{-1}]\to\bbI_{\dag}$ and
  $\bbF_{01}(E_{\bullet}C_2)\to\bbI_{\dag}$ against the fibrant object $\bbG$.
  This also proves naturality in $\bbG$.
\end{proof}

Recall from \cref{prop.dagger_realization_action} that, for a dagger simplicial set $K$, the Real realization involution is
\begin{align}
  \tau_K
  :=
  \rho_K\circ|\dag_K|
  : |K|\to|K|: 
  [\sigma;(t_0,\ldots,t_n)]
  \mapsto 
  [\sigma^\dagger;(t_n,\ldots,t_0)].
\end{align}
It defines the action of the non-identity element of $C_2$ on $|K|$.

\begin{construction}\label{dj.con.real-edgewise}
  Let $K$ be a simplicial set.
  Use the convention $\Tw(K)_n:=K_{2n+1}$ induced by $\epsilon([n]):=[n]^{\myop}\star[n]\cong[2n+1]$.
  Write $\rho_q:[q]\to[q]$ for the order-reversing bijection defined by $\rho_q(r)=q-r$.
  For an order-preserving map $\alpha:[m]\to[n]$, the map $\epsilon(\alpha):[2m+1]\to[2n+1]$ is
  \begin{align}
    \epsilon(\alpha)(r)
    =
    \begin{cases}
      n-\alpha(m-r),
      &0\leq r\leq m,\\
      n+1+\alpha(r-m-1),
      &m+1\leq r\leq2m+1.
    \end{cases}
  \end{align}
  These formulas give an relation 
  \begin{align}
    \rho_{2n+1}\epsilon(\alpha)\rho_{2m+1} = \epsilon(\alpha).
  \end{align}
\end{construction}

\begin{lemma}\label{dj.lem.real-edgewise-homeomorphism}
  Let $K$ be a simplicial set.
  Then we have:
  \begin{enumerate}
    \item there exists a natural edgewise-subdivision homeomorphism 
    \begin{align}
      h_K:(|\Tw(K)|,|D_{\Tw}|)\to(|K|,\tau_K).
    \end{align}
    For every $x\in K_0$, it carries the vertex $s_0x\in\Tw(K)_0$ to $x\in K_0$;
    \item if $K$ is a dagger simplicial set with dagger $d:K\to K^{\myop}$, then the maps $d_{2n+1}$ define a strict simplicial involution $D_{\Tw}$ on $\Tw(K)$;
    \item under the hypothesis of (2), $h_K$ is a strict $C_2$-equivariant homeomorphism.
  \end{enumerate}
\end{lemma}

\begin{proof}
  (1)
  The standard homeomorphism $h_K:|\Tw(K)|\to|K|$ is induced in degree $n$ by
  \begin{align}
    e_n(t_0,\ldots,t_n)
    :=
    \frac12(t_n,\ldots,t_0,t_0,\ldots,t_n)
    \quad \text{and} \quad
    h_K([z;t])
    :=
    [z;e_n(t)].
  \end{align}
  For $0\leq j\leq n$, the coordinates in positions $n-j$ and $n+1+j$ of both $|\epsilon(\alpha)|e_m(t)$ and $e_n|\alpha|(t)$ are $\frac12\sum_{\alpha(i)=j}t_i$.
  Thus the maps $e_n$ induce $h_K$, and the edgewise-subdivision theorem identifies it as a natural homeomorphism.
  This is the coordinate-reversed form of \cite[Remark 3.1.1]{Dotto16}.
  \begin{align}
    h_K[s_0x;1]
    =
    [s_0x;(1/2,1/2)]
    =
    [x;1].
  \end{align}

  (2)
  Suppose now that $K$ is a dagger simplicial set.
  The dagger simplicial identity, together with the last equality of \cref{dj.con.real-edgewise}, gives 
  \begin{align}
    d_{2m+1}K(\epsilon(\alpha)) =K(\epsilon(\alpha))d_{2n+1}.
  \end{align}
  Since $d$ is involutive, the maps $d_{2n+1}$ form a strict simplicial involution $D_{\Tw}$ on $\Tw(K)$.
  
  (3)
  Since $R_{2n+1}e_n=e_n$, \cref{prop.dagger_realization_action} gives 
  \begin{align}
    h_K[d_{2n+1}z;t]=\tau_Kh_K[z;t].
  \end{align}
  Thus $h_K$ is strictly $C_2$-equivariant.
\end{proof}

\begin{lemma}\label{dj.lem.twisted-target-equivalence}
  Let $K$ be a simplicial set, and use the edgewise-subdivision convention of \cref{dj.con.real-edgewise}.
  The maps $\nu_n:[n]\to[2n+1]$ defined by $\nu_n(i):=n+1+i$ induce a natural simplicial map $t_K:\Tw(K)\to K$.
\end{lemma}

\begin{proof}
  For every order-preserving map $\alpha:[m]\to[n]$, we have $\epsilon(\alpha)\nu_m = \nu_n\alpha$.
  Consequently, the maps
  \begin{align}
    (t_K)_n
    :=
    K(\nu_n):
    \Tw(K)_n=K_{2n+1}
    \to
    K_n
  \end{align}
  define a natural simplicial map $t_K:\Tw(K)\to K$.
  Naturality in $K$ follows from $f_nK(\nu_n)=L(\nu_n)f_{2n+1}$ for every simplicial map $f:K\to L$.
\end{proof}

\begin{proposition}\label{dj.prop.real-identity-section}
  Let $\bbG$ be a strict dagger simplicial groupoid.
  Write $D:\bbG\to\bbG^{\myop}$ for its dagger, let $I:\bbG^{\myop}\to\bbG$ be inversion, put $\sigma:=ID$, and put $X:=N(\UdagSCat\bbG)$.
  Then we have:
  \begin{enumerate}
    \item $\sigma$ is a covariant involution fixing the objects and $X$ is a Kan complex;
    \item for the involution $D_{\Tw}$ of \cref{dj.lem.real-edgewise-homeomorphism}, there exists a morphism 
    \begin{align}
      j:(X,N\sigma) \to (\Tw(X),D_{\Tw})
    \end{align}
    in $\Fun(BC_{2},\sSet)_{X_0/}$, natural in strict dagger simplicial functors.
    It satisfies $j(x)=\id_x$, and its value on $h:x\to y$ is represented by $(h^{-1},h)$;
    \item after forgetting the $C_2$-actions, the map $t_X$ of \cref{dj.lem.twisted-target-equivalence} satisfies $t_Xj_{\und}\simeq\id_X$.
    Moreover, $j$ is an equivalence.
  \end{enumerate}
\end{proposition}

\begin{proof}
  (1)
  For a morphism $g$ of $\bbG$, one has $\sigma(g)=(Dg)^{-1}$.
  If $f$ and $g$ are composable, then
  \begin{align}
    \sigma(gf)
    =
    (D(gf))^{-1}
    =
    (DfDg)^{-1}
    =
    (Dg)^{-1}(Df)^{-1}
    =
    \sigma(g)\sigma(f).
  \end{align}
  Thus $\sigma$ is covariant.
  Since the dagger commutes with inversion, one has $\sigma^2=\id_{\bbG}$.
  Both the dagger and inversion fix the objects, so $\sigma$ fixes them as well.

  Every mapping simplicial set of $\UdagSCat\bbG$ is a Kan complex, and its homotopy category is a groupoid.
  Hence $X$ is a Kan complex.

  (2)
  Put $\rho:=N\sigma$.
  Let $d:X\to X^{\myop}$ be the duality induced by $D$, and let $\iota:X\to X^{\myop}$ be the equivalence induced by inversion.
  Since $I^{\myop}I=\id_{\bbG^{\myop}}$, the naturality and strict coherence of the nerve--opposite comparison in \cref{dj.lem.lift} give $\iota\rho = d$.
  Applying \cite[Lemma~3.9 and proof of Lemma~3.14]{HLS25} to $d$, and passing from their convention to that of \cref{dj.con.real-edgewise}, gives a natural equivalence
  \begin{align}
    r:
    (\Tw(X),D_{\Tw})
    \xrightarrow{\simeq}
    (X,\rho)
  \end{align}
  in $\Fun(BC_{2},\sSet)$.
  Here the transported involution is $D_{\Tw}$ because both involutions act in simplicial degree $n$ by $d_{2n+1}$.
  Let $j:(X,\rho) \to (\Tw(X),D_{\Tw})$ be the constant-diagram inverse of $r$ furnished by the cited proof.
  Its constant-diagram representative satisfies $j(x)=\id_x$ and $j(h)=(h^{-1},h)$.
  Since $\sigma$ fixes the objects and $D_{\Tw}(\id_x)=\id_x$, this defines $j$ under $X_0$.
  Naturality of the cited equivalence and of the convention change proves the asserted naturality in strict dagger simplicial functors.

  (3)
  We work after forgetting the $C_2$-actions.
  By \cite[proof of Lemma~3.14]{HLS25}, the endpoint map $\Tw(X)\to X^{\myop}\times X$ factors as
  \begin{align}
    \Tw(X)
    \xrightarrow{r_{\und}}
    X
    \xrightarrow{(\iota,\id_X)}
    X^{\myop}\times X.
  \end{align}
  Taking the second component gives $t_X\simeq r_{\und}$.
  Since $j$ was chosen as an inverse of $r$, we obtain
  \begin{align}
    t_Xj_{\und}
    \simeq
    r_{\und}j_{\und}
    \simeq
    \id_X.
  \end{align}

  Finally, $j$ is an equivalence because it was chosen as an inverse of $r$ in (2).
\end{proof}

\begin{lemma}\label{dj.prop.strict-borel-presentation}
  Let $O$ be a $C_2$-simplicial set, and let $A$ and $B$ be objects of $\Fun(BC_{2},\sSet)_{O/}$.
  Assume that the structure map $O\to A$ is a monomorphism.
  For every equivalence $e:A\to B$ in $\Fun(BC_{2},\sSet)_{O/}$, the induced isomorphism in the homotopy category is represented by a zigzag between $A$ and $B$ of strict $C_2$-equivariant Borel weak equivalences under $O$.
\end{lemma}

\begin{proof}
  The assumption that $O \to A$ is a monomorphism says that $A$ is cofibrant in $\Fun(BC_{2},\sSet)_{\inj})_{O/}$.
  Choose a fibrant replacement $r:B \xrightarrow{\sim} B^{\fib}$ under $O$.
  Let $[e]:A\to B$ denote the isomorphism in $\Ho(\left(\Fun(BC_{2},\sSet)_{\inj}\right)_{O/})$ determined by the given equivalence.
  Since $A$ is cofibrant and $B^{\fib}$ is fibrant, $[r]\circ[e]:A\to B^{\fib}$ has a strict representative $\widetilde e: A \to B^{\fib}$ under $O$.
  Its homotopy class is an isomorphism, so $\widetilde e$ is a weak equivalence.
  Thus $A \xrightarrow{\widetilde e} B^{\fib} \xleftarrow{r} B$ is a zigzag of strict $C_2$-equivariant maps under $O$.
  The injective and projective model structures have the same objectwise weak equivalences.
  Therefore it is a zigzag of Borel weak equivalences.
\end{proof}

\begin{proposition}\label{dj.lem.real-edgewise}
  Let $\bbG$ be a strict dagger simplicial groupoid.
  Write $D:\bbG\to\bbG^{\myop}$ for its dagger and let $I:\bbG^{\myop}\to\bbG$ be inversion.
  Then $\sigma:=ID$ is a covariant involution fixing the objects.
  If $X:=N(\UdagSCat\bbG)$, then there exists a natural vertex-preserving Borel equivalence of $C_2$-spaces 
  \begin{align}
    (|X|,|N\sigma|) \simeq (|X|,\tau_{N_{\dag}\bbG}).
  \end{align}
\end{proposition}

\begin{proof}
  By \cref{dj.prop.real-identity-section}, $\sigma$ is a covariant involution fixing the objects, and there exists a natural equivalence 
  \begin{align}
    j:(X,N\sigma) \simeq (\Tw(X),D_{\Tw})
  \end{align}
  in $\Fun(BC_{2},\sSet)_{X_0/}$.
  The structure map $X_0\to X$ is a monomorphism.
  By \cref{dj.prop.strict-borel-presentation}, the isomorphism induced by $j$ in the homotopy category is represented by a zigzag of strict $C_2$-equivariant Borel weak equivalences under $X_0$.
  Geometric realization gives a vertex-preserving Borel equivalence 
  \begin{align}
    |j| : (|X|,|N\sigma|) \simeq (|\Tw(X)|,|D_{\Tw}|).
  \end{align}

  By \cref{dj.lem.real-edgewise-homeomorphism}, the natural map $h_X: (|\Tw(X)|,|D_{\Tw}|) \to (|X|,\tau_{N_{\dag}\bbG})$ is a strict $C_2$-equivariant homeomorphism.
  Since $\id_x=s_0x$, it is defined under the full vertex space.
  Composing these equivalences gives the required vertex-preserving Borel equivalence.

  The equivalence $j$ and the homeomorphism $h_X$ are natural in strict dagger simplicial functors.
  The strict zigzag need not be chosen naturally, but it represents the image of the natural equivalence determined by $j$ in the Borel homotopy category.
  Hence the resulting equivalence class is natural.
\end{proof}

\begin{lemma}\label{dj.lem.equivariant-path}
  Let $K$ be a dagger Kan complex and let $x,y\in K_0$.
  Give $((\frakC_{\dag}K)[\Mor^{-1}])(x,y)$ the involution $g\mapsto(g^{\dag})^{-1}$.
  Give the ordinary path space $P_{x,y}|K|$ the pointwise involution $\gamma(t)\mapsto\tau_K(\gamma(t))$.
  Then there exists a zigzag of $C_2$-equivariant weak equivalences, through projectively fibrant $C_2$-simplicial sets,
  \begin{align}
    ((\frakC_{\dag}K)[\Mor^{-1}])(x,y)
    \simeq
    \Sing P_{x,y}|K|.
  \end{align}
\end{lemma}

\begin{proof}
  Put $\bbG:=(\frakC_{\dag}K)[\Mor^{-1}]$.
  Write $\sigma(g):=(g^{\dag})^{-1}$.
  This is a strict covariant involution of $\bbG$ fixing its objects.

  The monomorphism $\sk_0K\to K$ is a Joyal cofibration.
  Ordinary rigidification is left Quillen by \cite[Theorem 2.2.5.1]{HTT}.
  The simplicial category $\frakC(\sk_0K)$ is the discrete simplicial category on $K_0$:
  Since $\sk_0K\to K$ is the identity on vertices, the induced Bergner cofibration $\frakC(\sk_0K)\to\frakC K$ is the identity on objects and is therefore a cofibration in $\sCat_{K_0}$.
  It follows that $\frakC K$ is cofibrant in the ordinary fixed-object Bergner model structure.
  Its homotopy category is canonically $\h K$, which is a groupoid because $K$ is a Kan complex.
  Therefore \cite[Proposition 9.5]{DK80} gives a local weak equivalence $\frakC K\to(\frakC K)[\Mor^{-1}]$, which is the identity on objects.
  Since degreewise groupoid completion commutes with opposites, the ordinary localization map admits a dagger refinement $q_{\dag}:\frakC_{\dag}K\to\bbG$.
  By \cref{dj.lem.lift}, it has an adjoint $\eta_K:K\to N_{\dag}\bbG$.
  The adjunction bijection of \cref{dj.lem.lift} is the restriction of the ordinary adjunction bijection for $\frakC\dashv N$.
  Consequently, the underlying map $\UdagSSet(\eta_K)$ is precisely the ordinary adjoint of $\UdagSCat(q_{\dag})$.

  The functor $\UdagSCat(q_{\dag})$ is the preceding local weak equivalence and is the identity on objects.
  Hence it is a Bergner weak equivalence.
  Since the adjunction $\frakC \dashv N$ is a Quillen equivalence, $\UdagSSet(\eta_K)$ is a Joyal equivalence.

  The homotopy category of $\UdagSCat\bbG$ is a groupoid.
  Hence $N(\UdagSCat\bbG)$ is a Kan complex.
  Since $\UdagSSet K$ is also a Kan complex, the Joyal equivalence $\UdagSSet(\eta_K)$ is a weak homotopy equivalence.

  We moreover verify that $\eta_K$ is a dagger Joyal equivalence.
  The functor $q_{\dag}$ is a local weak equivalence and is the identity on objects.
  Identities are coherently unitary, so $q_{\dag}$ is a dagger DK-equivalence by \cref{bg.prop.cu-groupoid} (1).
  The simplicial category $\bbG$ is dagger Bergner fibrant.
  By \cref{dj.lem.lift}, the underlying functor of the counit $\varepsilon_{\bbG}:\frakC_{\dag}N_{\dag}\bbG\to\bbG$ is the ordinary rigidification--nerve counit.
  Since $\UdagSCat\bbG$ is Bergner fibrant, the ordinary Joyal--Bergner Quillen equivalence makes this underlying counit a Bergner weak equivalence.
  The counit is a dagger functor and is the identity on objects, so it is a dagger DK-equivalence.
  The adjunction identity gives $q_{\dag}=\varepsilon_{\bbG}\circ\frakC_{\dag}(\eta_K)$.
  Two-out-of-three therefore shows that $\frakC_{\dag}(\eta_K)$ is a dagger DK-equivalence.
  Thus $\eta_K$ is a dagger Joyal equivalence.

  Since $q_{\dag}$ is the identity on objects, $\eta_K$ is the identity on vertices.
  Consequently, $|\eta_K|: (|K|,\tau_K) \to (|N(\UdagSCat\bbG)|,\tau_{N_{\dag}\bbG})$ is a vertex-preserving $C_2$-equivariant weak equivalence by \cref{prop.dagger_realization_action}.

  The natural mapping-space comparison for the homotopy coherent nerve \cite[Corollary 5.3]{DS11}, together with the rigidification--nerve counit \cite[Proposition 5.9]{DS11}, gives a natural zigzag
  \begin{align}
    \bbG(x,y)_{\sigma}
    \xleftarrow{\simeq}
    \frakC(N(\UdagSCat\bbG))(x,y)
    \simeq
    \Map_{N(\UdagSCat\bbG)}(x,y)_{N\sigma}.
  \end{align}
  Naturality of the Dugger--Spivak comparisons with respect to the involution $N\sigma$ makes this a zigzag of $C_2$-equivariant weak equivalences.

  Let $L$ be a Kan complex and let $a,b\in L_0$.
  By \cite[Corollary 5.3]{DS11}, $\Map_L(a,b)$ and $\frakC L(a,b)$ are connected by a natural weak-equivalence zigzag.
  As in the first paragraph of this proof, $\frakC L$ is cofibrant in the fixed-object Bergner model structure and its homotopy category is $\h L$, which is a groupoid.
  Hence \cite[Proposition 9.5]{DK80} gives a natural local weak equivalence $\frakC L\to(\frakC L)[\Mor^{-1}]$.
  The path-category comparison of \cite[Proposition 4.9 and Theorem 4.10]{MRZ23} identifies this localized rigidification, by a natural weak-equivalence zigzag, with the singular Moore path category.
  Finally, inclusion as paths of length one and rescaling of lengths give a natural homotopy equivalence from the Moore path space to the ordinary path space.
  Combining these comparisons gives a natural zigzag
  \begin{align}
    \Map_L(a,b)
    \simeq
    \frakC L(a,b)
    \simeq
    ((\frakC L)[\Mor^{-1}])(a,b)
    \simeq
    \Sing P^{\rmM}_{a,b}|L|
    \simeq
    \Sing P_{a,b}|L|,
  \end{align}
  where $P^{\rmM}_{a,b}|L|$ is the Moore path space between the realized vertices.
  Every comparison is natural in bipointed maps.
  Hence, if $L$ carries an involution fixing $a$ and $b$, the zigzag is $C_2$-equivariant.
  Applying it to $(N(\UdagSCat\bbG),x,y)$ with the $N\sigma$-action gives the required comparison between
  $\Map_{N(\UdagSCat\bbG)}(x,y)_{N\sigma}$ and the path space with the pointwise $|N\sigma|$-action.

  Regard $P_{x,y}|N(\UdagSCat\bbG)|$ first with the pointwise action induced by $|N\sigma|$ and then with the pointwise action induced by $\tau_{N_{\dag}\bbG}$.
  For a bipointed $C_2$-space $(T,a,b)$, the space $P_{a,b}T$ represents the homotopy fiber over $(a,b)$ of endpoint evaluation by \cref{dj.not.intrinsic-mapping}.
  It therefore depends functorially on $(T,a,b)$ in the $\infty$-category of bipointed $C_2$-spaces and carries equivalences under the two endpoints to equivalences.
  The equivalence of \cref{dj.lem.real-edgewise} is represented by a zigzag under the full vertex space.
  Applying the endpoint homotopy-fiber functor at $x$ and $y$ gives a $C_2$-equivariant weak-equivalence zigzag between the two actions on $P_{x,y}|N(\UdagSCat\bbG)|$.
  Applying the same argument to the vertex-preserving equivariant weak equivalence $|\eta_K|$ gives
  \begin{align}
    P_{x,y}|K|
    \simeq
    P_{x,y}|N(\UdagSCat\bbG)|
  \end{align}
  for the corresponding pointwise actions.
  Thus we obtain 
  \begin{align}
    \bbG(x,y)_{\sigma}
    &\simeq
    \Map_{N(\UdagSCat\bbG)}(x,y)_{N\sigma} \\
    &\simeq
    \Sing P_{x,y}(|N(\UdagSCat\bbG)|,|N\sigma|) \\
    &\simeq
    \Sing P_{x,y}(|N(\UdagSCat\bbG)|,\tau_{N_{\dag}\bbG}) \\
    &\simeq
    \Sing P_{x,y}(|K|,\tau_K).
  \end{align}
  All the preceding comparisons are natural in bipointed dagger maps after passage to the projective homotopy category.

  Finally, apply a functorial fibrant replacement in
  $\Fun(BC_{2},\sSet)_{\proj}$ to every intermediate object in the preceding natural zigzags.
  Two-out-of-three shows that all induced arrows remain weak equivalences.
  The two displayed endpoints are already projectively fibrant.
  Adjoining their replacement maps therefore gives the asserted zigzag entirely through projectively fibrant $C_2$-simplicial sets.
\end{proof}

\begin{proposition}\label{dj.lem.core-interval-mapping}
  Let $X$ be a dagger quasi-category, let $x,y\in X_0$, and let $\bbH$ be a natural dagger interval.
  There exists a zigzag of weak equivalences
  \begin{align}
    \Path_{\calU(X)}(\widehat x,\widehat y)
    \simeq
    \RMap_{\sCat^{\dag}_{\{0,1\}}}(\bbI_{\dag},\frakC_{\dag}X\langle x,y\rangle)
    \simeq
    \underline{\Map}_{\sCat^{\dag}_{\{0,1\}}}(\bbH,R(\frakC_{\dag}X\langle x,y\rangle)).
  \end{align}
  This zigzag is natural in bipointed dagger maps in the homotopy category.
  Here $R$ is the fixed-object fibrant replacement fixed before \cref{bg.def.cu}.
\end{proposition}

\begin{proof}
  Let $(\frakC_{\dag}X)^{\simeq}$ denote the dagger simplicial subcategory whose mapping spaces are the unions of the components representing isomorphisms in $\h\frakC_{\dag}X$.
  The natural maps give a zigzag, identical on objects,
  \begin{align}
    (\frakC_{\dag}X)^{\simeq}
    \leftarrow
    \frakC_{\dag}(X^{\simeq})
    \to
    \frakC_{\dag}(X^{\simeq})[\Mor^{-1}] .
  \end{align}
  For every pair of vertices $a,b$, the Dugger--Spivak comparison
  \cite[Corollary 5.3]{DS11} identifies the mapping space
  $\frakC(X^{\simeq})(a,b)$ with the union of those components of
  $\Map_X(a,b)$ which represent equivalences.
  The same union of components is weakly equivalent to
  $(\frakC X)^{\simeq}(a,b)$.
  Naturality of these comparisons shows that the first arrow is a local weak equivalence.

  The monomorphism $\operatorname{sk}_0X^{\simeq}\to X^{\simeq}$ is a cofibration.
  Since rigidification is left Quillen, $\frakC(X^{\simeq})$ is cofibrant with its object set fixed.
  Its homotopy category is a groupoid because every edge of $X^{\simeq}$ is an equivalence.
  Therefore \cite[Proposition~9.5]{DK80} applies and shows that the second arrow is a local weak equivalence.

  The displayed maps themselves are obtained from the inclusion $X^{\simeq}\to X$ by rigidification and from degreewise groupoid completion.
  Both constructions commute with opposites, so these are dagger functors.

  Put $\bbG := \frakC_{\dag}(X^{\simeq})[\Mor^{-1}]\langle x,y\rangle$.
  Let $(\frakC_{\dag}X\langle x,y\rangle)^{\simeq}\subseteq\frakC_{\dag}X\langle x,y\rangle$ be defined by the same union-of-components construction.
  Let $(R(\frakC_{\dag}X\langle x,y\rangle))^{\simeq}\subseteq R(\frakC_{\dag}X\langle x,y\rangle)$ be the union of the mapping-space components which represent isomorphisms in the homotopy category.
  This is closed under composition, identities, and the dagger.
  It is fibrant because each mapping space is a union of components of a Kan complex.
  Its homotopy category is a groupoid, although
  $(R(\frakC_{\dag}X\langle x,y\rangle))^{\simeq}$ need not itself be a strict simplicial groupoid.
  
  Every dagger functor from a natural dagger interval into the fixed-object fibrant replacement of $\frakC_{\dag}X\langle x,y\rangle$ factors through its groupoidal part by \cref{bg.lem.interval-basic}.
  This also identifies the enriched mapping spaces.
  An $n$-simplex of the enriched mapping space is given on each ordered pair by a map from $\bbH(i,j)\times\Delta^n$.
  The mapping space $\bbH(i,j)$ is weakly contractible and therefore connected, while $\Delta^n$ is connected.
  Restriction to any vertex of $\Delta^n$ gives a dagger functor into this fibrant replacement.
  Every morphism in $\h\bbH$ is invertible by \cref{bg.lem.interval-basic}, so the image of every vertex of $\bbH(i,j)\times\Delta^n$ represents an isomorphism in the homotopy category of this fibrant replacement.
  Since this domain is connected, its entire image lies in one of the selected components.

  The initial fixed-object dagger category embeds cofibrantly into $\bbJ_{\dag}$.
  Moreover, $j_{\bbH}:\bbJ_{\dag}\to\bbH$ is a fixed-object cofibration by definition.
  Hence $\bbH$ is cofibrant.
  Since $p_{\bbH}:\bbH\to\bbI_{\dag}$ is a local weak equivalence, it is a cofibrant approximation of $\bbI_{\dag}$ in the fixed-object model structure.
  Since $\bbH\to\bbI_{\dag}$ is a cofibrant approximation and the chosen replacement is fibrant, the derived mapping space from $\bbI_{\dag}$ is computed by the enriched mapping space from $\bbH$ into this target.
  The preceding factorization identifies that enriched mapping space with the one obtained after restricting the target to its groupoidal part.

  Write $r:\frakC_{\dag}X\langle x,y\rangle \to R(\frakC_{\dag}X\langle x,y\rangle)$ for the fixed-object fibrant replacement.
  It restricts to a local weak equivalence $r^{\simeq}: (\frakC_{\dag}X\langle x,y\rangle)^{\simeq} \to (R(\frakC_{\dag}X\langle x,y\rangle))^{\simeq}$
  Indeed, it is a local weak equivalence and hence induces bijections on the components of all mapping spaces.
  Since it is the identity on objects, the induced functor on homotopy categories is an isomorphism and therefore preserves and reflects which components represent isomorphisms.
  Restricting to the unions of these corresponding components therefore remains a weak equivalence on every mapping space.

  Restricting the preceding rigidification and completion maps to the two selected objects gives the explicit fixed-object zigzag
  \begin{align}
    (R(\frakC_{\dag}X\langle x,y\rangle))^{\simeq}
    \xleftarrow{\simeq}
    (\frakC_{\dag}X\langle x,y\rangle)^{\simeq}
    \xleftarrow{\simeq}
    \frakC_{\dag}(X^{\simeq})\langle x,y\rangle
    \xrightarrow{\simeq}
    \bbG .
  \end{align}
  Thus $(R(\frakC_{\dag}X\langle x,y\rangle))^{\simeq}$ and the strict simplicial groupoid $\bbG$ are isomorphic in the fixed-object homotopy category.
  The derived comparison can now be written without applying the strict-groupoid calculation to $(R(\frakC_{\dag}X\langle x,y\rangle))^{\simeq}$:
  \begin{align}\label{dj.eq.core-strictification}
    \RMap_{\sCat^{\dag}_{\{0,1\}}}
    (\bbI_{\dag},\frakC_{\dag}X\langle x,y\rangle)
    &\simeq
    \underline{\Map}_{\sCat^{\dag}_{\{0,1\}}}
    (\bbH,R(\frakC_{\dag}X\langle x,y\rangle))\\
    &\cong
    \underline{\Map}_{\sCat^{\dag}_{\{0,1\}}}
    (\bbH,(R(\frakC_{\dag}X\langle x,y\rangle))^{\simeq})\\
    &\simeq
    \RMap_{\sCat^{\dag}_{\{0,1\}}}
    (\bbI_{\dag},(R(\frakC_{\dag}X\langle x,y\rangle))^{\simeq})\\
    &\simeq
    \RMap_{\sCat^{\dag}_{\{0,1\}}}(\bbI_{\dag},\bbG)\\
    &\simeq 
    \underline{\Map}_{C_2}(E_{\bullet}C_2,\bbG(0,1)_{\sigma}).
  \end{align}
  The first and third weak equivalences use the cofibrancy of $\bbH$ and the fibrancy of the indicated targets.
  The last weak equivalence is the strict simplicial groupoid calculation of \cref{dj.lem.unitary-derived-fixed}.
  Both $(\frakC_{\dag}(X^{\simeq})[\Mor^{-1}])(x,y)_{\sigma}$ and $\Sing P_{x,y}|X^{\simeq}|$ are fibrant in the projective model structure on $C_2$-simplicial sets.
  The first is a Kan complex because it is a mapping object of a strict simplicial groupoid, and the second is a singular complex.
  Since $E_{\bullet}C_2$ is projectively cofibrant, the simplicial right Quillen functor $\underline{\Map}_{C_2}(E_{\bullet}C_2,-)$ carries the zigzag of \cref{dj.lem.equivariant-path} to a zigzag of weak equivalences.
  We therefore obtain
  \begin{align}
    \underline{\Map}_{C_2}
    (
      E_{\bullet}C_2,
      (\frakC_{\dag}(X^{\simeq})[\Mor^{-1}])(x,y)_{\sigma}
    )
    \simeq
    \underline{\Map}_{C_2}
    (
      E_{\bullet}C_2,
      \Sing P_{x,y}|X^{\simeq}|
    ).
  \end{align}
  The realization--singular adjunction and exponential adjunction now give a natural zigzag
  \begin{align}
    \underline{\Map}_{C_2}(E_{\bullet}C_2,\Sing P_{x,y}|X^{\simeq}|)
    &\simeq
    \Sing\Map_{C_2}(EC_2,P_{x,y}|X^{\simeq}|)\\
    &\cong
    \Sing P_{c_x,c_y}\Map_{C_2}(EC_2,|X^{\simeq}|).
  \end{align}
  The enriched realization--singular adjunction identifies
  $\Sing P_{c_x,c_y}\Map_{C_2}(EC_2,|X^{\simeq}|)$ with the strict fiber over $(c_x,c_y)$ of
  \begin{align}
    (\Sing \Map_{C_2}(EC_2,|X^{\simeq}|))^{\Delta^1}\to(\Sing \Map_{C_2}(EC_2,|X^{\simeq}|))^{\partial\Delta^1}.
  \end{align}
  This endpoint map is a Kan fibration.
  Hence the strict fiber computes the homotopy fiber and is naturally weakly equivalent to the path space from $\widehat x$ to $\widehat y$ in $(X^{\simeq})^{hC_2}$.
  Since $\calU(X)$ is the full Kan subcomplex spanned by the canonical vertices, \cref{dj.lem.full-subkan} gives
  \begin{align}
    \Sing P_{c_x,c_y}\Map_{C_2}(EC_2,|X^{\simeq}|)
    \simeq
    \Path_{\calU(X)}(\widehat x,\widehat y)
  \end{align}
  Together with \eqref{dj.eq.core-strictification}, this proves the first weak equivalence in the statement.

  The target $R(\frakC_{\dag}X\langle x,y\rangle)$ is fibrant by construction.
  Since the fixed-object model structure is simplicial by \cref{bg.thm.fixed-object}, the enriched mapping space from $\bbH$ into this target is a Kan complex and computes the second derived mapping space in the statement.

  All replacements used above may be chosen functorially.
  The maximal Kan subcomplex, rigidification, groupoid completion, the Dugger--Spivak comparison, the equivariant path comparison, and the enriched adjunctions then define natural transformations after passage to the homotopy category.
  Hence the displayed zigzag is natural in bipointed dagger maps.
\end{proof}

\begin{corollary}\label{dj.lem.core-interval}
  Let $X$ be a dagger quasi-category and let $x,y\in X_0$.
  Then $\widehat x$ and $\widehat y$ lie in the same component of $\calU(X)$ if and only if there exists a coherently unitary equivalence from $x$ to $y$ in $\h\frakC_{\dag}X$.
\end{corollary}

\begin{proof}
  Fix a natural dagger interval $\bbH$.
  By \cref{dj.lem.core-interval-mapping}, $\widehat x$ and $\widehat y$ lie in the same component of $\calU(X)$ if and only if $\underline{\Map}_{\sCat^{\dag}_{\{0,1\}}}(\bbH,R(\frakC_{\dag}X\langle x,y\rangle))$ is nonempty.
  The last mapping space is nonempty precisely when there exists a dagger functor $\bbH\to R(\frakC_{\dag}X\langle x,y\rangle)$.
  The local weak equivalence $\frakC_{\dag}X\langle x,y\rangle\to R(\frakC_{\dag}X\langle x,y\rangle)$ is surjective on $\pi_0$ of every mapping space.

  Hence, if $\Phi$ is such a functor, there exists a vertex $u\in\frakC_{\dag}X(x,y)$ satisfying $[r(u)]=[\Phi(h_{\bbH})]$.
  The pair $(\bbH,\Phi)$ is therefore a witness for the coherent unitarity of $[u]$ in the sense of \cref{bg.def.cu}.

  Conversely, every natural dagger interval is a cofibrant approximation of $\bbI_{\dag}$.
  A witness using any such interval makes the derived mapping space nonempty, and weak equivalences of the resulting Kan mapping spaces preserve nonemptiness.
  Hence there exists a witness using the fixed $\bbH$.
  By \cref{dj.lem.core-interval-mapping}, this is equivalent to the existence of a path from $\widehat x$ to $\widehat y$ in $\calU(X)$.
\end{proof}

\begin{lemma}\label{dj.lem.intrinsic-local}
  Let $f:X\to Y$ be a dagger functor between dagger quasi-categories.
  Then $\frakC_{\dag}(f)$ is a local weak equivalence if and only if $f$ is fully faithful.
\end{lemma}

\begin{proof}
  By \cref{dj.lem.lift}, the underlying map of mapping simplicial sets is the corresponding component of $\frakC(\UdagSSet f) : \frakC(\UdagSSet X)\to\frakC(\UdagSSet Y)$.
  The natural Dugger--Spivak comparison \cite[Corollary 5.3]{DS11} gives, for every $x,x'\in X_0$, a commutative zigzag 
  \begin{align}
    \frakC(\UdagSSet X)(x,x') \simeq \Map_X(x,x')
    \quad \text{and} \quad
    \frakC(\UdagSSet Y)(fx,fx') \simeq \Map_Y(fx,fx').
  \end{align}
  Two-out-of-three proves the assertion.
\end{proof}

\begin{theorem}[Intrinsic recognition of dagger Joyal equivalences]\label{dj.thm.intrinsic-equivalence}
  Let $f:X\to Y$ be a dagger functor between dagger quasi-categories.
  The following conditions are equivalent:
  \begin{enumerate}
    \item $f$ is a dagger Joyal equivalence, equivalently a weak equivalence in $\sSet^{\dag}_{\Joyal}$;
    \item $f$ is fully faithful and the map $\pi_0\calU(f):\pi_0\calU(X)\to\pi_0\calU(Y)$ is surjective;
    \item $f$ is fully faithful and $\calU(f):\calU(X)\to \calU(Y)$ is a weak homotopy equivalence.
  \end{enumerate}
\end{theorem}

\begin{proof}
  (1) $\Rightarrow$ (2):
  Suppose that $f$ is a dagger Joyal equivalence.
  By \cref{dj.def.W}, the dagger simplicial functor
  $\frakC_{\dag}(f)$ is a local weak equivalence and is coherently unitarily essentially surjective.
  By \cref{dj.lem.intrinsic-local}, the local weak-equivalence condition is equivalent to full faithfulness of $f$.
  It remains to prove that $\pi_0\calU(f): \pi_0\calU(X)\to\pi_0\calU(Y)$ is surjective.

  Let $y\in Y_0$.
  Coherent-unitary essential surjectivity gives an object $x\in X_0$ and a coherently unitary equivalence $[v]: f(x) \to y$ in $\h\frakC_{\dag}Y$.
  By \cref{dj.lem.core-interval}, this is equivalent to the assertion that the canonical vertices $\widehat{f(x)}$ and $\widehat y$ lie in the same connected component of $\calU(Y)$.
  Naturality of the unitary core gives $\calU(f)(\widehat x)=\widehat{f(x)}$.
  Hence $[\widehat y]=\pi_0\calU(f)([\widehat x])$ in $\pi_0\calU(Y)$.

  By definition, $\calU(Y)$ is the full Kan subcomplex spanned by the canonical vertices $\widehat y$ for $y\in Y_0$.
  Thus every connected component of $\calU(Y)$ contains a canonical vertex.
  It follows that $\pi_0\calU(f)$ is surjective.

  (2) $\Rightarrow$ (1):
  Suppose that $f$ is fully faithful and that $\pi_0\calU(f)$ is surjective.
  Let $y\in Y_0$.
  By surjectivity, the component $[\widehat y]\in\pi_0\calU(Y)$ has a preimage in $\pi_0\calU(X)$.
  Every connected component of $\calU(X)$ contains a vertex, and every vertex of $\calU(X)$ is a canonical vertex.
  Hence this preimage can be represented by $\widehat x$ for some $x\in X_0$.
  We therefore have 
  \begin{align}
    [\widehat y] = \pi_0\calU(f)([\widehat x]) = [\widehat{f(x)}].
  \end{align}
  Thus $\widehat{f(x)}$ and $\widehat y$ lie in the same connected component of $\calU(Y)$.
  By \cref{dj.lem.core-interval}, there exists a coherently unitary equivalence $[v]: f(x) \to y$ in $\h\frakC_{\dag}Y$.
  Since this holds for every $y\in Y_0$, the functor $\frakC_{\dag}(f)$ is coherently unitarily essentially surjective.

  Moreover, full faithfulness of $f$ implies that $\frakC_{\dag}(f)$ is a local weak equivalence (by \cref{dj.lem.intrinsic-local}).
  Consequently, $\frakC_{\dag}(f)$ is a dagger DK-equivalence.
  By \cref{dj.def.W}, $f$ is a dagger Joyal equivalence.

  (2) $\Rightarrow$ (3):
  Choose the vertex $1\in EC_2$.
  Evaluation at this vertex gives a natural map $\ev_1:(Y^{\simeq})^{hC_2}\to\Sing|Y^{\simeq}|$, which sends $\widehat y$ to $y$.
  Surjectivity of $\pi_0\calU(f)$ therefore implies that, for every $y\in Y_0$, the vertices $f(x)$ and $y$ lie in the same component of $Y^{\simeq} \simeq \Sing|Y^{\simeq}|$ for some $x\in X_0$.
  Thus the underlying functor of quasi-categories is ordinarily essentially surjective.
  It is fully faithful by assumption, so \cref{dj.prop.ordinary-recognition} shows that the underlying map $f:X\to Y$ is a Joyal equivalence.

  By the last assertion of \cref{dj.prop.ordinary-recognition}, the dagger map $f^{\simeq}:X^{\simeq}\to Y^{\simeq}$ is an underlying weak homotopy equivalence of dagger Kan complexes.
  By \cref{prop.dagger_realization_action}, its realization $|f^{\simeq}|$ is $C_2$-equivariant, and it is therefore a weak equivalence in the Borel model structure on $C_2$-spaces.
  The space $EC_2$ is cofibrant in this model structure (by \cref{lem.EC2_free_contractible}).
  Hence the right Quillen functor $\Map_{C_2}(EC_2,-)$, followed by $\Sing$, carries this map to a weak equivalence $(X^{\simeq})^{hC_2}\to(Y^{\simeq})^{hC_2}$.

  The unitary cores are the full Kan subcomplexes spanned by the canonical vertices.
  By \cref{dj.lem.full-subkan}, their path spaces between canonical vertices agree with the corresponding path spaces in the ambient homotopy fixed point Kan complexes.
  The preceding weak equivalence therefore makes $\calU(f)$ fully faithful as a map of Kan complexes.
  Condition (2) makes it essentially surjective.
  Therefore $\calU(f)$ is a weak homotopy equivalence.

  (3) $\Rightarrow$ (2):
  The implication from (3) to (2) is immediate.
\end{proof}

\begin{corollary}\label{dj.cor.intrinsic-arbitrary}
  Let $g:A\to B$ be any morphism of dagger simplicial sets.
  Choose a functorial dagger Joyal fibrant replacement and let $g^{\fib}:A^{\fib}\to B^{\fib}$ be the induced map.
  Then $g$ is a dagger Joyal equivalence if and only if $g^{\fib}$ is fully faithful and $\pi_0\calU(g^{\fib})$ is surjective.
  Equivalently, $g^{\fib}$ is fully faithful and $\calU(g^{\fib})$ is a weak homotopy equivalence.
\end{corollary}

\begin{proof}
  Both fibrant-replacement maps are weak equivalences.
  Two-out-of-three therefore reduces the assertion to $g^{\fib}$.
  By \cref{dj.cor.fibrant-underlying}, its source and target are dagger quasi-categories.
  Apply \cref{dj.thm.intrinsic-equivalence}.
\end{proof}

\begin{remark}\label{dj.rem.intrinsic-sanity}
  For a dagger $1$-category $\calA$, the counit $\frakC_{\dag}N_{\dag}\calA\to\calA$ is an identity-on-objects local weak equivalence by the proof of \cref{dj.thm.equivalence}.
  Coherent unitarity is preserved and reflected by this counit by \cref{bg.prop.cu-groupoid} (5), and it agrees with strict unitarity in $\calA$ by \cref{bg.lem.discrete}.
  Hence \cref{dj.lem.core-interval} shows that two canonical vertices of $\calU(N_{\dag}\calA)$ lie in the same component precisely when the corresponding objects are joined by a unitary isomorphism.
  Thus \cref{dj.thm.intrinsic-equivalence} reduces to the usual criterion:
  a dagger functor is an equivalence precisely when it is fully faithful and every target object is unitarily isomorphic to an object in its image.
\end{remark}

\section{Comparisons with anti-involutive models}\label{pointset.ainv-section}

We compare the dagger simplicial sets and model structures constructed above with the anti-involutive point-set models of \cite{DCH19,DH25}.
We first compare the underlying indexing categories and identify dagger simplicial sets as the strict fixed-vertex anti-involutive simplicial sets.
We then compare weak equivalences, cofibrations, fibrant objects, and generators in the two model structures, culminating in \cref{pointset.cor.no-naive-quillen}.

\subsection{Anti-involutive simplicial sets}

We begin by separating the vertex-level involution from the reversal already present in the dagger indexing category.
The resulting tagged indexing category admits a normal-form description, and its fixed-vertex quotient recovers the indexing category for dagger simplicial sets.

Recall the involutive automorphism $\rev:\prism\to\prism$ and the notation $\alpha^{\rev}$ in \cref{constr.alpha_rev}.

\begin{definition}\label{pointset.def.tagged}
  We let $\widetilde\prism_{\rev}$ denote the category generated by $\prism$ together with symbols $\widetilde\rho_n:[n]\to[n]$ for $n\geq0$, subject to the relations $\widetilde\rho_n^2=\id_{[n]}$ and $\widetilde\rho_n\alpha=\alpha^{\rev}\widetilde\rho_m$ for every order-preserving map $\alpha:[m]\to[n]$.
\end{definition}

\begin{lemma}\label{pointset.lem.normal-forms}
  We have:
  \begin{enumerate}
    \item every morphism $[m]\to[n]$ of $\widetilde\prism_{\rev}$ has a normal form $\alpha$ or $\widetilde\rho_n\alpha$, with $\alpha:[m]\to[n]$ order preserving;
    \item two normal forms of the same orientation agree if and only if their order-preserving parts agree;
    \item normal forms of opposite orientations have the same underlying set map if and only if that set map is constant;
    \item the quotient obtained by imposing $\widetilde\rho_0=\id_{[0]}$ is canonically isomorphic to $\prism_{\rev}$.
  \end{enumerate}
\end{lemma}

\begin{proof}
  (1)
  Using the crossed relation
  $\widetilde\rho_n\alpha=\alpha^{\rev}\widetilde\rho_m$,
  every occurrence of a tagged reversal can be moved to the left of every order-preserving map.
  Since $\widetilde\rho_n^2=\id_{[n]}$, pairs of tagged reversals cancel.
  Hence every morphism $[m]\to[n]$ has the form
  $\alpha$ or $\widetilde\rho_n\alpha$ for an order-preserving map
  $\alpha:[m]\to[n]$.

  (2)
  Let $\calT$ have the same objects as $\prism$ and the following hom-sets:
  \begin{align}
    \Hom_{\calT}([m],[n])
    :=
    \Hom_{\prism}([m],[n])\times C_2.
  \end{align}
  Put $\rev^1:=\id$ and $\rev^\omega:=\rev$;
  similarly, put $\widetilde\rho_n^1:=\id_{[n]}$ and $\widetilde\rho_n^\omega:=\widetilde\rho_n$.
  Define composition by
  \begin{align}
    (\beta,h)\circ(\alpha,g)
    :=
    (\rev^g(\beta)\circ\alpha,hg).
  \end{align}
  Since $\rev$ is an involutive automorphism of $\prism$, this composition is associative and has identities $(\id_{[n]},1)$.

  The assignments $\alpha\mapsto(\alpha,1)$ and $\widetilde\rho_n\mapsto(\id_{[n]},\omega)$ satisfy the defining relations and induce a functor $\widetilde\prism_{\rev}\to\calT$.
  Conversely, define a functor $\calT\to\widetilde\prism_{\rev}$ by $(\alpha,g)\mapsto\widetilde\rho_n^g\alpha$.
  This assignment respects composition because, for every $\beta:[n]\to[p]$, we have
  \begin{align}
    \beta\widetilde\rho_n^g
    =
    \widetilde\rho_p^g\rev^g(\beta).
  \end{align}
  The two functors are inverse to each other.
  Therefore two normal forms with the same orientation agree if and only if their order-preserving parts agree.

  (3)
  The underlying set map of $\alpha$ is order preserving, whereas that of
  $\widetilde\rho_n\beta$ is order reversing.
  If these underlying maps agree, the common map is both order preserving and order reversing, and is therefore constant.
  Conversely, if $c_k:[m]\to[n]$ is the constant map with value $k$, then
  $c_k$ and $\widetilde\rho_nc_{n-k}$ have the same underlying set map.

  (4)
  Let $!:[m]\to[0]$ be the unique map, let
  $v_k:[0]\to[n]$ be given by $v_k(0)=k$, and put $c_k:=v_k!$.
  After imposing $\widetilde\rho_0=\id_{[0]}$, the crossed relation gives:
  \begin{align}
    \widetilde\rho_nc_{n-k}
    =
    \widetilde\rho_nv_{n-k}! 
    =
    v_k\widetilde\rho_0! 
    =
    v_k! 
    =
    c_k.
  \end{align}
  Thus the quotient identifies the two opposite normal forms associated with every constant map.

  The underlying-set-map functor from the quotient to
  $\prism_{\rev}$ is surjective by (1).
  If two representatives of the same orientation have the same image, they agree by (2).
  If representatives of opposite orientations have the same image, their common underlying map is constant by (3), and the preceding calculation shows that they agree in the quotient.
  Hence the functor is also injective on every hom-set and is therefore an isomorphism.
\end{proof}

\begin{notation}
  We write $q_{\rmtag}:\widetilde\prism_{\rev}\to\prism_{\rev}$ for the quotient functor obtained from \cref{pointset.lem.normal-forms}.
\end{notation}

Let $C_2$ act on $\prism$ by $\rev$.
Drummond-Cole and Hackney define the reflexive crossed simplicial category $\nabla:=\prism\rtimes C_2$ in \cite[Definition~4.8]{DCH19}.
The same category is recalled in \cite[Section~5]{DH25}.

\begin{proposition}\label{pointset.prop.reflexive-index}
  There is an isomorphism of categories:
  \begin{align}
    \Phi_{\rmDH}:
    \nabla
    \xrightarrow{\cong}
    \widetilde\prism_{\rev}.
  \end{align}
  Under this isomorphism, the composite $q_{\rmtag}\Phi_{\rmDH}:\nabla\to\prism_{\rev}$ is the quotient generated by $\widetilde\rho_0=\id$.
\end{proposition}

\begin{proof}
  Use the multiplicative notation $C_2=\{1,\omega\}$ for the second coordinate of the semidirect product.
  Thus the hom-sets are
  \begin{align}
    \Hom_{\nabla}([m],[n])
    =
    \Hom_{\prism}([m],[n])\times C_2.
  \end{align}
  We use the DCH composition convention $(\beta,h)\circ(\alpha,g) = (\beta\circ\rev^h(\alpha),hg)$.
  For $(\alpha,g):[m]\to[n]$, define
  \begin{align}
    \Phi_{\rmDH}(\alpha,g)
    :=
    \widetilde\rho_n^g\rev^g(\alpha)
    =
    \alpha\widetilde\rho_m^g,
  \end{align}
  where the second equality follows from the crossed relation.

  The mutually inverse functors constructed in the proof of \cref{pointset.lem.normal-forms} show that every morphism of $\widetilde\prism_{\rev}$ has a unique tagged normal form $\widetilde\rho_n^g\gamma$.
  Define $\Psi(\widetilde\rho_n^g\gamma) := (\rev^g(\gamma),g)$.
  Then we have
  \begin{align}
    \Psi\Phi_{\rmDH}(\alpha,g)
    =
    (\rev^g(\rev^g(\alpha)),g)
    =
    (\alpha,g)
    \quad \text{and} \quad
    \Phi_{\rmDH}\Psi(\widetilde\rho_n^g\gamma)
    =
    \widetilde\rho_n^g\rev^g(\rev^g(\gamma))
    =
    \widetilde\rho_n^g\gamma.
  \end{align}
  Thus $\Psi$ is the inverse of $\Phi_{\rmDH}$.

  In $\nabla$, the elements $(\id_{[0]},1)$ and $(\id_{[0]},\omega)$ are distinct.
  They correspond to $\id_{[0]}$ and $\widetilde\rho_0$.
  By \cref{pointset.lem.normal-forms}, imposing their equality identifies exactly the two orientation tags on constant maps and produces $\prism_{\rev}$.
\end{proof}

\begin{notation}
  We let $\sSet^{\AInv}$ denote the category of simplicial sets equipped with an isomorphism $\tau:X\to X^{\myop}$ satisfying $\tau^{\myop}\tau=\id$, with equivariant simplicial maps as morphisms.
\end{notation}

\begin{proposition}\label{pointset.thm.strict-presheaves}
  There is a natural isomorphism
  \begin{align}
    \Fun(\widetilde\prism_{\rev}^{\myop},\Set)
    \cong
    \sSet^{\AInv}.
  \end{align}
  Under this isomorphism and the equivalence $\Fun(\prism_{\rev}^{\myop},\Set)\simeq\sSet^{\dag}$ of \cref{prop.sSetdag_is_equivalent_to_Fun}, precomposition with $q_{\rmtag}$ is the following fully faithful inclusion
  \begin{align}
    q_{\rmtag}^*:
    \sSet^{\dag}
    \hookrightarrow
    \sSet^{\AInv}.
  \end{align}
\end{proposition}

\begin{proof}
  Let $W:\widetilde\prism_{\rev}^{\myop}\to\Set$ be a presheaf, and let $X$ be its restriction to $\prism$.
  Put $\tau_n:=W(\widetilde\rho_n)$.
  The crossed relations say that $\tau_n^2=\id$ and that, for every $\alpha:[m]\to[n]$, we have
  \begin{align}
    X(\alpha)\tau_n
    =
    \tau_mX(\alpha^{\rev}).
  \end{align}
  This is precisely the naturality condition for an involutive isomorphism $X\to X^{\myop}$.

  Conversely, let $(X,\tau)$ be anti-involutive.
  For $g\in C_2=\{1,\omega\}$, put $\tau_n^1:=\id_{X_n}$ and $\tau_n^\omega:=\tau_n$.
  On a tagged normal form $\widetilde\rho_n^g\alpha:[m]\to[n]$, define $W(\widetilde\rho_n^g\alpha) := X(\alpha)\tau_n^g$.
  The uniqueness of normal forms and the preceding naturality identity show that this defines a presheaf.
  These constructions are inverse on objects and morphisms.

  Such a presheaf factors through $q_{\rmtag}$ if and only if it sends $\widetilde\rho_0=\id$ to an equality.
  Under the displayed isomorphism, this condition is $\tau_0=\id$, which is exactly the fixed-vertex condition for a dagger simplicial set.
\end{proof}

\begin{remark}
  More generally, for every category $\calV$, precomposition defines a fully faithful functor
  \begin{align}
    q_{\rmtag}^*:
    \Fun(\prism_{\rev}^{\myop},\calV)
    \longrightarrow
    \Fun(\widetilde\prism_{\rev}^{\myop},\calV),
  \end{align}
  whose essential image consists of the diagrams sending $\widetilde\rho_0$ to an identity.
\end{remark}

\subsection{Comparisons between the dagger Joyal and DCH model structures}

We compare the two model structures along the adjoint triple $q_{\rmtag,!}\dashv q_{\rmtag}^*\dashv q_{\rmtag,*}$.
The comparison is organized by weak equivalences, cofibrations, fibrant objects, and generators, after which we combine the obstructions to prove \cref{pointset.cor.no-naive-quillen}.

\begin{notation}
  The free--forgetful adjunction of \cite[Section~5]{DH25} is
  \begin{align}
    F_{\AInv}:
    \sSet
    \rightleftarrows
    \sSet^{\AInv}
    :U_{\AInv}
  \end{align}
  We retain the notation $X^{(0)}=\sk_0X$ and the adjunction $\FdagSSet\dashv\UdagSSet$ from \cref{dj.not.adjunctions}.
\end{notation}

\begin{lemma}\label{pointset.prop.free-completions}
  We have:
  \begin{enumerate}
    \item for every simplicial set $X$, there are natural isomorphisms
    \begin{align}
      F_{\AInv}(X)
      \cong
      X\amalg X^{\myop}
      \quad \text{and} \quad
      \FdagSSet(X)
      \cong
      X\amalg_{X^{(0)}}X^{\myop};
    \end{align}
    \item the canonical map $F_{\AInv}(X) \to q_{\rmtag}^*\FdagSSet(X)$ identifies the two copies of every vertex and its degeneracies;
    \item we let $q_{\rmtag,!}$ denotes left Kan extension along $q_{\rmtag}$; 
    then there is a natural isomorphism
    \begin{align}
      q_{\rmtag,!}F_{\AInv}
      \cong
      \FdagSSet.
    \end{align}
  \end{enumerate}
\end{lemma}

\begin{proof}
  (1)
  Let $Y$ be anti-involutive.
  An anti-involutive map $X\amalg X^{\myop}\to Y$ is uniquely determined by its restriction to the first summand, because the involution forces its restriction to the second summand.
  Thus there is a natural isomorphism
  \begin{align}
    \Hom_{\sSet^{\AInv}}(X\amalg X^{\myop},Y)
    \cong
    \Hom_{\sSet}(X,U_{\AInv}Y),
  \end{align}
  proving the first formula.
  The second formula is the free-dagger calculation of \cref{dj.not.adjunctions}.

  (2)
  Its pushout identifies the two copies of every totally degenerate simplex, and the coproduct-to-pushout map is the stated natural quotient.

  (3)
  Let $i_{\rmtag}:\prism\to\widetilde\prism_{\rev}$ and $i_{\rev}:\prism\to\prism_{\rev}$ be the inclusions.
  Since $q_{\rmtag}i_{\rmtag}=i_{\rev}$, their restriction functors satisfy $i_{\rmtag}^*q_{\rmtag}^*=i_{\rev}^*$.
  The functors $q_{\rmtag,!}F_{\AInv}$ and $\FdagSSet$ are both left adjoint to $i_{\rev}^*$.
  Uniqueness of a left adjoint gives the final natural isomorphism.
  Under it, the unit of $q_{\rmtag,!}\dashv q_{\rmtag}^*$ is the quotient just described.
\end{proof}

\subsubsection{Weak equivalences}

The DCH model structure is right-induced from the ordinary Joyal model structure.
Thus a morphism is a DCH weak equivalence exactly when its underlying simplicial map is a Joyal equivalence \cite[Corollary~4.9]{DCH19}.
The dagger Joyal weak equivalences additionally detect the selected coherent-unitary core.

\begin{proposition}\label{pointset.thm.weak-comparison}
  Let $f:X\to Y$ be a functor of dagger quasi-categories.
  If $f$ is a dagger Joyal equivalence, then $q_{\rmtag}^*f$ is a DCH weak equivalence.
\end{proposition}

\begin{proof}
  By \cref{dj.thm.intrinsic-equivalence}, a dagger Joyal equivalence is fully faithful and induces a weak equivalence $\calU(X)\to\calU(Y)$ on unitary cores.
  In particular, every vertex of $Y$ is connected by a coherently unitary equivalence to a vertex in the image of $f$.
  Hence the underlying map of quasi-categories is fully faithful and essentially surjective, and is therefore a Joyal equivalence.
\end{proof}

\begin{remark}\label{pointset.prop.weak-counterexample}
  The converse of \cref{pointset.thm.weak-comparison} does not hold in general:
  there is a fully faithful dagger functor $F:\calC\to\calD$ between dagger groupoids such that:
  \begin{enumerate}
    \item the underlying functor is an equivalence of categories;
    \item $F$ is not unitarily essentially surjective;
    \item $q_{\rmtag}^*N_{\dag}(F)$ is a DCH weak equivalence, while $N_{\dag}(F)$ is not a dagger Joyal equivalence.
  \end{enumerate}

  Let $\calD$ have objects $0$ and $1$, with $\Hom_{\calD}(i,j):=\bbZ$ for every ordered pair $(i,j)$.
  Identities are represented by $0$, and composition is $b\circ a:=a+b$.
  Every morphism $a$ has inverse $-a$, so $\calD$ is a groupoid.

  Put $t_0:=0$ and $t_1:=1$.
  For $a:i\to j$, define $a^{\dag}:=a+t_i-t_j$, regarded as a morphism $j\to i$.
  The dagger axioms follow from
  \begin{align}
    (a^{\dag})^{\dag}
    =
    a+t_i-t_j+t_j-t_i
    =a
    \quad \text{and} \quad
    (b\circ a)^{\dag}
    =
    a+b+t_i-t_k
    =
    a^{\dag}\circ b^{\dag}.
  \end{align}
  Identities are fixed, so this defines a dagger on $\calD$.

  Let $\calC$ be the full dagger subgroupoid on $0$, and let $F:\calC\hookrightarrow\calD$ be the inclusion.
  It is fully faithful.
  The morphism $0\in\Hom_{\calD}(0,1)$ is an isomorphism, so the underlying functor is essentially surjective and hence an equivalence of categories.

  A morphism $a:i\to j$ is unitary exactly when $a^{\dag}=a^{-1}=-a$, or equivalently $2a=t_j-t_i$.
  For $i\neq j$, the right-hand side is $1$ or $-1$, so this equation has no solution in $\bbZ$.
  For $i=j$, it forces $a=0$.
  Hence the only unitary morphisms of $\calD$ are its identities, and no unitary morphism joins $0$ to $1$.
  Therefore $F$ is not unitarily essentially surjective.

  Regard $\calC$ and $\calD$ as dagger simplicial categories with discrete mapping spaces.
  They are fibrant by \cref{bg.cor.fibrant}, and their dagger nerves are fibrant by \cref{dj.cor.fibrant-underlying}.
  Their underlying nerves are Kan complexes, so their images under $q_{\rmtag}^*$ are DCH fibrant.
  Since the underlying functor is an equivalence, $q_{\rmtag}^*N_{\dag}(F)$ is a DCH weak equivalence.
  By \cref{bg.lem.discrete}, coherent unitarity is strict unitarity here.
  The unitary core of $\calC$ therefore has one component, while that of $\calD$ has two.
  By \cref{dj.thm.intrinsic-equivalence}, $N_{\dag}(F)$ is not a dagger Joyal equivalence.
\end{remark}

\subsubsection{Cofibrations}

Recall that a normal monomorphism in the DCH model structure is a monomorphism for which the involution acts freely on the new nondegenerate simplices in every degree.
On the other hand, a free cofibration of the dagger Joyal model structure requires this freeness only in positive dimensions, because every vertex is dagger-fixed.

\begin{proposition}\label{pointset.cor.normal-to-free}
  The left adjoint $q_{\rmtag,!}:\sSet^{\AInv}\to\sSet^{\dag}$ sends every DCH normal monomorphism to a free cofibration.
\end{proposition}

\begin{proof}
  The DCH normal monomorphisms are generated under pushouts, transfinite compositions, and retracts by $F_{\AInv}(\partial\Delta^n) \to  F_{\AInv}(\Delta^n)$.
  By \cref{pointset.prop.free-completions} (3), their images are the generating free cofibrations $\FdagSSet(\partial\Delta^n)\to\FdagSSet(\Delta^n)$.
  The functor $q_{\rmtag,!}$ preserves colimits and retracts.
\end{proof}

\begin{proposition}\label{pointset.thm.cofibrations}
  Let $i:A\to B$ be a monomorphism of dagger simplicial sets.
  The anti-involutive map $q_{\rmtag}^*i$ is a normal monomorphism if and only if:
  \begin{enumerate}
    \item $i$ is bijective on vertices;
    \item $i$ is a free cofibration, equivalently a cofibration of the dagger Joyal model structure.
  \end{enumerate}
\end{proposition}

\begin{proof}
  Suppose first that $q_{\rmtag}^*i$ is normal.
  The involution on $B_0$ is the identity.
  Hence an element of $B_0\setminus A_0$ would be fixed by the involution, contradicting normality.
  Thus $A_0\to B_0$ is bijective.
  In positive dimensions, normality says precisely that no new nondegenerate simplex is fixed by the involution.
  By \cref{dj.def.free-cof}, this is the free-cofibration condition.

  Conversely, suppose that (1) and (2) hold.
  There are no new nondegenerate simplices in degree zero, and in positive dimensions condition (2) makes the involution free on all new nondegenerate simplices.
  The nondegenerate-simplex criterion \cite[Corollary~4.9]{DCH19} shows that $q_{\rmtag}^*i$ is normal.
\end{proof}

\begin{corollary}\label{pointset.cor.cofibrant-obstruction}
  If a dagger simplicial set $X$ has a vertex, then $q_{\rmtag}^*X$ is not cofibrant in the DCH model structure.
  In contrast, the one-vertex dagger simplicial set $\FdagSSet(\Delta^0)$ is cofibrant in the dagger Joyal model structure.
\end{corollary}

\begin{proof}
  If $X_0$ is nonempty, the map $\varnothing\to X$ is not bijective on vertices.
  By \cref{pointset.thm.cofibrations}, $\varnothing\to q_{\rmtag}^*X$ is not a DCH normal monomorphism.
  Hence $q_{\rmtag}^*X$ is not DCH cofibrant.
  On the other hand, $\varnothing\to\FdagSSet(\Delta^0)$ is one of the generating free cofibrations.
\end{proof}

\subsubsection{Fibrant objects}

\begin{proposition}\label{pointset.prop.fibrant-comparison}
  For a dagger simplicial set $X$, the following are equivalent:
  \begin{enumerate}
    \item $q_{\rmtag}^*X$ is fibrant in the DCH anti-involutive Joyal model structure;
    \item $\UdagSSet X$ is a quasi-category.
  \end{enumerate}
\end{proposition}

\begin{proof}
  The DCH model structure is right-induced from the ordinary Joyal model structure, and its fibrant objects are precisely the anti-involutive quasi-categories \cite[Corollary~4.9]{DCH19}.
  The functor $q_{\rmtag}^*$ does not change the underlying simplicial set.
\end{proof}

\begin{corollary}\label{pointset.cor.dagger-fibrant-to-dch}
  Every fibrant object of the dagger Joyal model becomes DCH fibrant after applying $q_{\rmtag}^*$.
\end{corollary}

\begin{proof}
  By \cref{dj.cor.fibrant-underlying}, the underlying simplicial set of every dagger Joyal fibrant object is a quasi-category.
  Apply \cref{pointset.prop.fibrant-comparison}.
\end{proof}

\begin{example}\label{pointset.example.nonfixed-fibrant}
  The discrete simplicial set on two vertices, equipped with the involution which exchanges them, is DCH fibrant.
  It does not belong to the essential image of $q_{\rmtag}^*$, since its degree-zero involution is nontrivial.
  Thus DCH fibrant objects need not have strict fixed vertices.
\end{example}

\subsubsection{Generators}

\begin{notation}
  We fix some notational conventions:
  \begin{itemize}
    \item let $\mathbf J$ be the ordinary walking-isomorphism category, and write $\bbJ:=N(\mathbf J)$;
    \item let $d_{\ai} : F_{\AInv}(\Delta^0) \to F_{\AInv}(\bbJ)$ be the double-$\bbJ$ inclusion of \cite[Example~5.9]{DH25}.
  \end{itemize}
\end{notation}

\begin{proposition}\label{pointset.prop.free-iso-to-unitary}
  There is a canonical degreewise-surjective dagger map $\pi:  \FdagSSet(\bbJ) \to N_{\dag}(\bbI_{\dag})$ which imposes the relation $u^{\dag}=u^{-1}$ on the generating isomorphism.
\end{proposition}

\begin{proof}
  The underlying simplicial set of $N_{\dag}(\bbI_{\dag})$ is $\bbJ$.
  The identity map on $\bbJ$ therefore extends uniquely along the free--forgetful adjunction to $\pi$.
  Under the presentation $\FdagSSet(\bbJ) = \bbJ \amalg_{\sk_0\bbJ} \bbJ^{\myop}$, it is the identity on the first summand and is determined on the second summand by the dagger of $N_{\dag}(\bbI_{\dag})$.
  Thus it identifies the formal dagger of $u$ with $u^{-1}$.
  The first summand already maps onto the target, so $\pi$ is degreewise surjective.
\end{proof}

\begin{proposition}\label{pointset.thm.generator-comparison}
  Left Kan extension along $q_{\rmtag}$ has the following effect on the DCH/DH generators:
  \begin{enumerate}
    \item for every $n\geq0$, the image of $F_{\AInv}(\partial\Delta^n)\to F_{\AInv}(\Delta^n)$ is canonically isomorphic to $\FdagSSet(\partial\Delta^n)\to\FdagSSet(\Delta^n)$;
    \item for every inner horn $0<k<n$, the image of $F_{\AInv}(\Lambda^n_k)\to F_{\AInv}(\Delta^n)$ is canonically isomorphic to the dagger Joyal trivial cofibration $\FdagSSet(\Lambda^n_k)\to\FdagSSet(\Delta^n)$;
    \item the image of $d_{\ai}:F_{\AInv}(\Delta^0) \to F_{\AInv}(\bbJ)$ is canonically isomorphic to $\FdagSSet(\Delta^0)\to\FdagSSet(\bbJ)$, which is a dagger Joyal cofibration but not a dagger Joyal weak equivalence.
  \end{enumerate}
\end{proposition}

\begin{proof}
  The three identifications follow from $q_{\rmtag,!}F_{\AInv}\cong\FdagSSet$ in \cref{pointset.prop.free-completions} (3).
 
  The maps in (1) generate the free cofibrations.
  The proof of \cref{dj.cor.fibrant-underlying} shows that the maps in (2) are dagger Joyal trivial cofibrations.
  The map in (3) is a free cofibration by \cref{dj.cor.Fdag-free}.

  It remains to prove that the map in (3) is not a dagger Joyal weak equivalence.
  Suppose that the map in (3) were a dagger Joyal weak equivalence.
  Applying dagger rigidification and its compatibility with free dagger completion would give a dagger DK-equivalence
  \begin{align}
    \mathbf1_{\dag} 
    \cong 
    \frakC_{\dag}\FdagSSet(\Delta^0)
    \to
    \frakC_{\dag}\FdagSSet(\bbJ)
    \cong
    \FdagSCat\frakC(\bbJ).
  \end{align}
  Hence the objects $0$ and $1$ of the target would be coherently unitarily equivalent.

  Compose the ordinary counit $\frakC(\bbJ)\to\mathbf J$ with the functor $\mathbf J\to\UdagSCat\calD$ which sends the generating isomorphism to $0\in\Hom_{\calD}(0,1)=\bbZ$.
  The free--forgetful adjunction extends this composite uniquely to a dagger simplicial functor $\FdagSCat\frakC(\bbJ) \to \calD$, where $\calD$ is the dagger groupoid of \cref{pointset.prop.weak-counterexample}.
  Dagger functors preserve coherent-unitary equivalences.
  Since $\calD$ has discrete mapping spaces, coherent unitarity in $\calD$ is strict unitarity by \cref{bg.lem.discrete}.
  We would therefore obtain a unitary morphism from $0$ to $1$ in $\calD$, contradicting \cref{pointset.prop.weak-counterexample}.
\end{proof}

\subsubsection{Model comparisons}

\begin{theorem}\label{pointset.cor.no-naive-quillen}
  Neither the adjunction
  \begin{align}
    q_{\rmtag,!}:
    \sSet^{\AInv}_{\DCH}
    \rightleftarrows
    \sSet^{\dag}_{\Joyal}:
    q_{\rmtag}^*
    \quad \text{nor} \quad 
    q_{\rmtag}^*:
    \sSet^{\dag}_{\Joyal}
    \rightleftarrows
    \sSet^{\AInv}_{\DCH}
    :q_{\rmtag,*}
  \end{align}
  is a Quillen adjunction.
  In particular, the dagger Joyal model structure is not obtained by restricting the DCH model structure to its strict fixed-vertex objects.
\end{theorem}

\begin{proof}
  The map $d_{\ai}$ is a DCH trivial cofibration \cite[Definition~5.10]{DH25}.
  Its image under $q_{\rmtag,!}$ is not a dagger Joyal weak equivalence by \cref{pointset.thm.generator-comparison}.
  Hence $q_{\rmtag,!}\dashv q_{\rmtag}^*$ is not a Quillen adjunction.

  On the other hand, by \cref{pointset.cor.cofibrant-obstruction}, $\FdagSSet(\Delta^0)$ is dagger Joyal cofibrant, whereas $q_{\rmtag}^*\FdagSSet(\Delta^0)$ is not DCH cofibrant.
  Thus $q_{\rmtag}^*$ does not preserve $\varnothing\to\FdagSSet(\Delta^0)$.
  The weak equivalences also disagree on strict fixed-vertex objects by \cref{pointset.prop.weak-counterexample}.
\end{proof}

\section*{References}
\begingroup
\printbibliography[heading=none]
\endgroup

\end{document}